\documentclass{article}
\usepackage{todonotes}

\usepackage{geometry}

\usepackage[utf8]{inputenc} % allow utf-8 input
\usepackage[T1]{fontenc}    % use 8-bit T1 fonts
\usepackage{hyperref}       % hyperlinks
\usepackage{url}            % simple URL typesetting
\usepackage{booktabs}       % professional-quality tables
\usepackage{amsfonts, amsmath, amsthm, amssymb}       % blackboard math symbols
\usepackage{nicefrac}       % compact symbols for 1/2, etc.
\usepackage{microtype}      % microtypography
\usepackage{xcolor}         % colors
\usepackage{mathtools}
\usepackage[makeroom]{cancel}
\usepackage{booktabs}
\mathtoolsset{showonlyrefs}

\usepackage{algorithm}
\usepackage{algorithmic}

\newcommand{\E}{\mathbb{E}}

\newcommand{\R}{\mathbb{R}}
\newcommand{\N}{\mathbb{N}}

\renewcommand{\d}{\, \mathrm{d}}
\newcommand{\ex}{\mathbb E}

\newcommand{\supp}{\mathrm{supp}}

\numberwithin{equation}{section} %equations numbered with sections
\numberwithin{table}{section} %tables numbered with sections
\newcounter{dummy} 
\numberwithin{dummy}{section}

\newtheorem{theorem}[dummy]{Theorem}
\newtheorem{lemma}[dummy]{Lemma} 
\newtheorem{proposition}[dummy]{Proposition}
\newtheorem{remark}[dummy]{Remark}

\newtheorem{corollary}[dummy]{Corollary}

\title{Telegrapher's Generative Model via Kac Flows}

\author{Richard Duong\footnotemark[1]
\and
Jannis Chemseddine\footnotemark[1]
\and
Peter K. Friz
\and
Gabriele Steidl
 }
\date{\today}

\begin{document}

\maketitle
\footnotetext{Institute of Mathematics,
TU Berlin,
Stra{\ss}e des 17. Juni 136, 
10623 Berlin, Germany,
\{duong, chemseddine, friz, steidl\}@math.tu-berlin.de
}

\renewcommand*{\thefootnote}{\fnsymbol{footnote}}
\footnotetext[1]{With equal contributions. Corresponding authors.}
\renewcommand*{\thefootnote}{\arabic{footnote}}

\begin{abstract}
We break the mold in flow-based generative modeling by proposing a new model based on the damped wave equation, also known as telegrapher's equation. Similar to the diffusion equation and Brownian motion, there is a Feynman-Kac type relation between the telegrapher's equation and the stochastic Kac process in 1D.
The Kac flow evolves stepwise linearly in time, so that the probability flow is Lipschitz continuous in the Wasserstein distance and, in contrast to diffusion flows, the norm of the velocity remains globally bounded.
Furthermore, the Kac model has the diffusion model as its asymptotic limit.
We extend these considerations to a multi-dimensional stochastic process which consists of independent 1D Kac processes in each spatial component.
We show that this process gives rise to an absolutely continuous curve in the Wasserstein space and analytically compute the conditional velocity field when starting in a Dirac point. Using the framework of flow matching, we train a neural network to approximate the velocity field and use it for sample generation.
Our numerical experiments demonstrate the scalability of our approach, and show its advantages over diffusion models.
\end{abstract}

%-------------------------------------------------------------------------
\section{Introduction}
Among generative neural models, flow-based techniques as score-based diffusion
and flow matching 
stand out for their simple applicability and good scaling properties.
The general idea is to construct a probability flow which starts in a simple latent distribution and leads to a possibly complicated target distribution. A neural network is then trained to learn this flow while using only a few samples from the target. 
%With the learned flow at hand, sampling from the target simply reduces to sampling from the latent.
In \emph{score-based diffusion models} \cite{pmlr-v37-sohl-dickstein15, SE2019}, 
the so-called score is learned which directs the drift-diffusion flow backwards in time from the latent to the target via a certain reverse SDE. In \emph{flow matching} \cite{lipman2023flow,liu2022rectified,liu2023flow} the velocity field is learned 
%(matched) 
enabling the calculation of the flow's trajectories via a certain flow ODE. These two approaches produce state-of-the-art results not only in pure sample generation, but can also be incorporated into inverse problem solvers, see e.g.\
\cite{MBHS2025,zhu2023denoising}. A unifying perspective can be given via the ``stochastic interpolants'' framework, see \cite{ABV2023,AV2022}.

Very little research has been dedicated to exploring new approaches to flow modeling using other physics-inspired PDEs. In \cite{Poisson2022}, a physics-inspired generative model  was proposed using the Poisson equation.
The authors of \cite{GenPhys2023} pointed to further
\emph{physical processes and PDEs} and generally examined them with respect to their applicability in generative modeling, among them approaches based on the Schr\"odinger equation, Helmholtz equation, and the above-mentioned Poisson and diffusion equation. More generally, they  identify two conditions mandatory for a model to be suited for data generation:
\begin{enumerate}
    \item ``Well-behaved density flow'': Initial probability densities stay probability densities, and their evolution is well-behaved.
    \item ``Asymptotic behavior'': Long-term convergence of the flow to a (known) stationary distribution independent of the initial distribution.%
    %\todo{Not sure about ML lingo, but "smoothing" seems a strange buzz word here. Better is``Ergodicity'' or "Convergence to Stationarity / Invariant Distribution". Since ergodicitiy has a precise meaning in Markov process theory, ``ergodic-type behaviour" also OK}
    
\end{enumerate}
The first condition can be split into two separable conditions: \emph{conservation of mass} and \emph{regularity of the flow}. 
While the meaning of the former is clear, the latter essentially depends on the choice of metric. The %\emph{Wasserstein space} $\mathcal P_2$ and its 
\emph{Wasserstein distance} $W_2$ provides a suitable metric for measuring the regularity of a probability flow $\mu_t$, and  
it is known that the metric derivative of the flow 
$\mu_t$ 
can be equivalently described by the $L_2$-norm $\|v_t\|_{L_2(\mu_t)}$ of the corresponding velocity field $v_t$. 
It has been observed by the authors of \cite{CWDS2025} that a blow-up of this norm can cause severe problems for the neural network to approximate the flow, leading to issues like approximation errors or even mode collapse. Intriguingly, also diffusion models suffer from a velocity explosion at times close to the target, see \cite{soft2021}. 

This motivated us to establish a new physics-inspired model based on the \emph{damped wave equation}, also known as the \emph{telegrapher's equation}.
This hyperbolic PDE models the evolution of a physical wave which gets dampened over time and  travels with a finite fixed speed. It has long been known in the physics community that (mathematical) diffusion can be seen as the limit of the damped wave model with \emph{infinite} damping and \emph{infinite} propagation speed, and it has been found \emph{unphysical} that particles in diffusion can travel with unbounded velocities, contradicting Einstein's principles of relativity. In contrast, our newly proposed \emph{telegrapher's generative model} exhibits a favorable regularity with \emph{bounded} velocities and a suitable long-term behavior. For sampling, we are able to utilize the flexible framework of flow matching, 
but %not applied to the usual straight-line trajectories, 
applied to a stochastic process intrinsically connected to the telegrapher's equation, namely the \emph{Kac process} \cite{KAC1974}, see Figure \ref{kac_walk_vs_brownian_end}.

\begin{figure}[ht!] \label{kac_walk_vs_brownian_end}
    \centering
    \includegraphics[width = 0.18\textwidth]{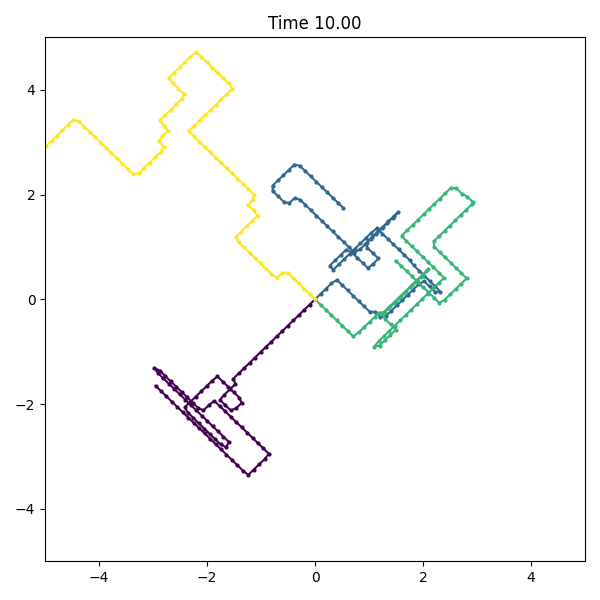}
    \includegraphics[width = 0.18\textwidth]{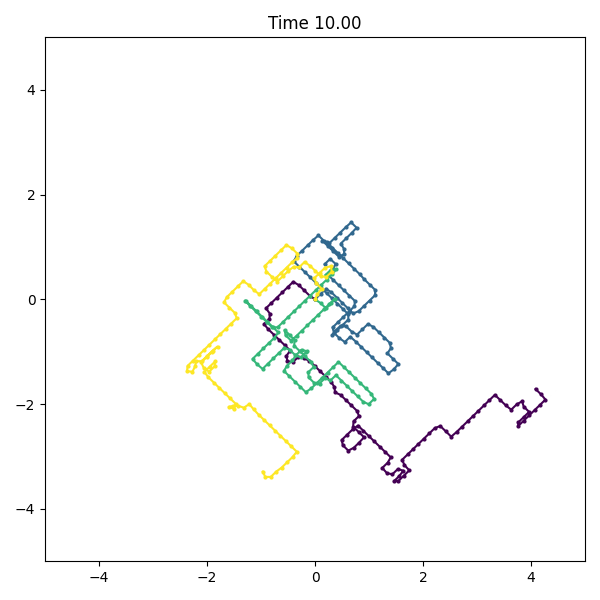}
    \includegraphics[width = 0.18\textwidth]{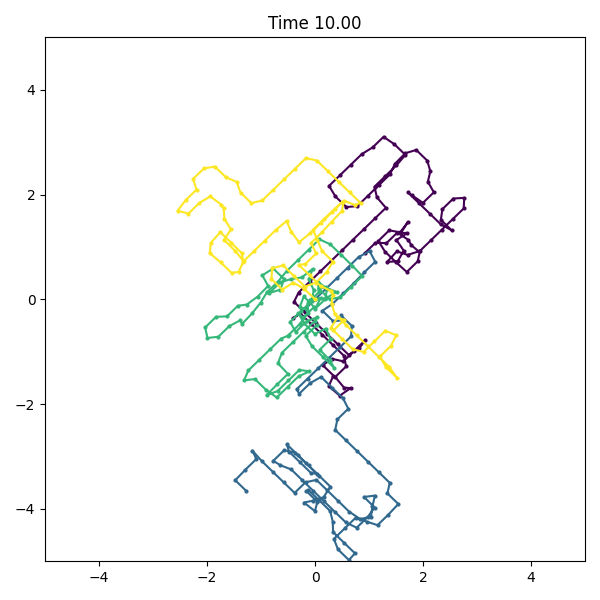}
    \includegraphics[width = 0.18\textwidth]{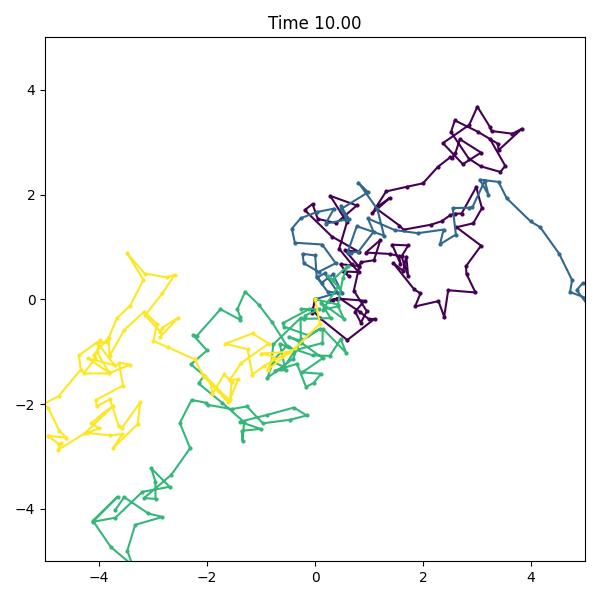}
    \includegraphics[width = 0.18\textwidth]{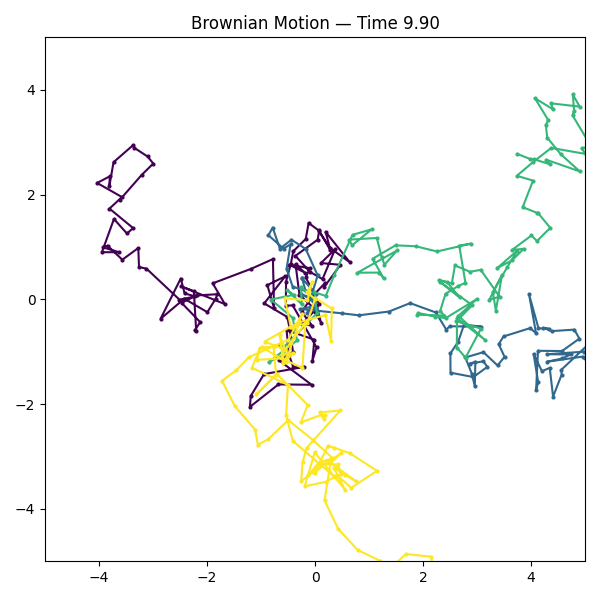}       
    \caption{Paths of the componentwise Kac walk in 2D, simulated until time $T=10$ with damping/velocity parameters $(a,c) = (1,1), (2,1 ), (4,2), (25,5)$, and a standard Brownian motion (right).}
\end{figure}

\paragraph{Contributions} 
\begin{enumerate}
     \item We show that the Kac probability flow is a Lipschitz (hence absolutely) continuous curve in the Wasserstein space. In particular it fulfills the continuity equation with a bounded velocity field for which we determine a corresponding \emph{conditional} velocity field for 1D flows starting in Dirac measures \emph{analytically} (Theorem \ref{velo_dirac}).
     \item We revisit the connection to the telegrapher's equation and work out that the Kac flow converges to a diffusion flow for suitably related large damping and velocity coefficients, and also for large times (Theorem \ref{convergence-heat}).
     \item We extend the above regularity and convergence results to a multi-dimensional stochastic process which consists of independent 1D Kac processes 
     in each spatial component (Proposition \ref{prop:lipschitz}, Theorem \ref{convergence-multi-d}). Moreover, we show that its conditional velocity field decomposes into univariate components (Theorem \ref{thm14}), and therefore obtain an \emph{explicit} expression.
     \item We use the \emph{decomposed and explicitly given} conditional vector field to train a neural network to approximate the true velocity field using ideas from flow matching and an efficient sampling strategy for the Kac process. This decomposition method is generally applicable to other componentwise forward processes and velocity fields.
     \item  We implement our generative model and demonstrate its generation quality and scalability by numerical examples.  In particular, we demonstrate its advantage over diffusion models in specific toy experiments.
\end{enumerate}

\paragraph{Outline of the paper.}
Section \ref{sec:prelim} provides necessary preliminaries on curves in Wasserstein spaces.
In Section \ref{sec:motivation} we explore the explosions of velocity fields of several popular neural flows which initially motivated us to study \emph{different} (physics-inspired) flows.
In Section \ref{sec:kac1d} we examine the 1D Kac model, its relation to the telegrapher's equation, and derive explicit formulas for the probability flow and conditional velocity field.
A generalization to higher dimensions is provided in 
Section \ref{sec:kac_multi}. Here we recall Kac's insertion method connecting the 1D Kac process to the multi-dimensional telegrapher's equation. Since the latter lacks a conservation of mass, we study a componentwise model in Subsection \ref{subsec:kac_multi_2}. An explicit representation of the velocity field of the multi-dimensional process is derived in Section \ref{sec:velo}.
In Section \ref{sec:numerics}, we use flow matching to approximate the corresponding velocity field and conduct numerical experiments.

%_______________________________________________________
\section{Preliminaries: Curves in Wasserstein Spaces}\label{sec:prelim}
%_______________________________________________________
Let $\mathcal P_2(\R^d)$ denote the complete metric space of probability measures with finite second moments
equipped with the Wasserstein distance  
$$W_2^2(\mu,\nu) := \min_{\pi \in \Pi(\mu,\nu)} \int_{\R^d \times \R^d} \|x-y\|^2 \d \pi(x,y),
$$ 
where  $\Pi(\mu,\nu)$ denotes the set of all probability measures on $\R^d \times \R^d$ having marginals $\mu$ and $\nu$, called couplings or transport plans between $\mu$ and $\nu$.
The push-forward measure of $\mu \in \mathcal P_2(\R^d)$ by a measurable map $T:\R^d \to \R^d$ is defined by
$T_\sharp \mu := \mu \circ T^{-1}$.
Let $I$ be a bounded interval in $\R$.
The space $AC^2(I;\mathcal P_2(\R^d))$
of \emph{absolutely continuous} curves consists of all $\mu_t: I \to \mathcal P_2(\R^d)$ for which there exists $m \in L_2(I)$ such that 
\begin{equation} \label{eq:abs_cont}
W_2(\mu_s,\mu_t) 
\leq \int_s^t m(r) \, \d r \quad \text{for all }  s\leq t, ~ s,t \in I.
\end{equation}
Here $L_2(I)$ denotes the space of (equivalence classes of) square-integrable real-valued functions on $I$.
It is well known that a narrowly\footnote{In stochastic analysis, narrow convergence is also known as \emph{weak convergence of measures}.} continuous curve $\mu_t:I \to \mathcal P_2(\R^d)$ 
belongs to $AC^2(I;\mathcal P_2(\R^d))$,
iff there exists a Borel measurable vector field $v: I\times \R^d \to \R^d$ such that 
$(\mu_t,v_t)$ satisfies the \emph{continuity equation}
\begin{align}\label{eq:ce}
\partial_t \mu_t + 
\nabla_x \cdot (\mu_t v_t)=0
\end{align}
in the sense of distributions and
\begin{equation}\label{eq:explode_1}
\|v_t\|_{L_2(\mu_t)} \in L_2(I),
\end{equation}
see \cite[Definition 1.1.1; Theorem 8.3.1]{BookAmGiSa05}.
If in addition
\begin{align} \label{eq:explode_2}
\int_I ~\, \sup_{x \in B} \|v_t(x)\|+\mathrm{Lip}(v_t,B)\d t<\infty \quad \text{for all compact } B\subset \R^d,
\end{align}
then the \emph{flow ODE} 
\begin{align}\label{eq:flow_ode} 
\partial_t \varphi(t,x)=v_t(\varphi(t,x)), \quad \varphi(0,x)=x,
\end{align}
has a solution
$\varphi:I\times\R^d\to \R^d$ 
and 
$\mu_t = \varphi(t,\cdot)_\sharp\mu_0$,
see
\cite[Proposition 8.1.8]{BookAmGiSa05}.
Note that often, absolutely continuous functions in $AC^1(I;\mathcal P_2(\R^d))$ with $m \in L_1(I)$ in \eqref{eq:abs_cont} were considered in the literature.
However, we need the stronger assumption of square integrability to define appropriate loss functions later.

Given an appropriate  velocity field $v_t$, we can sample from a target
distribution by starting from a sample $x_0$ from the latent distribution and applying an ODE solver to \eqref{eq:flow_ode}.
However, in this paper, we will reverse the flow direction by starting in the target distribution and ending (approximately) in the latent distribution at some final time $t=T$, similar to diffusion models. If we have access to the forward velocity $v_t$, then the reverse flow is simply given by the opposite velocity field $-v_{T-t}$.

\emph{Notational conventions.}
We frequently use the convention $f_t \coloneqq f(t,\cdot)$ for functions $f(t,x)$ depending on time and space.
Given some measure $\mu_t$ we denote its density (if it exists) by $p_t$, and we are occasionally imprecise when writing $p_t = \mu_t$.
If the measure $\mu_t$ (possibly) possesses singular parts, we denote the density, which then only exists as a \emph{generalized function} (distribution), by $u_t$.

%_______________________________________________________
\section{Motivation} \label{sec:motivation}
%_______________________________________________________
Our research was motivated by the explosion of velocity fields of certain flow models. 
If condition \eqref{eq:explode_2} or even worse the
condition \eqref{eq:explode_1} is not fulfilled, we speak about an \emph{exploding velocity field}. 
In the following, we give four examples of popular flows used in generative modeling.

\paragraph{1. Flows from couplings.}
Learning absolutely continuous curves induced by couplings
of the latent and target measure is a task addressed in \emph{flow matching},
see \cite{liu2023flow}, and for an overview, \cite{WS2025}.
Let us consider the latent $\mu_0= \mathcal{N}(0,I_d)$ and the target distribution $\mu_1 = \delta_0$, and 
their independent coupling $\alpha := \mu_0\times \mu_1$. 
Then the curve $\mu_t := e_{t,\sharp}\alpha$ with
$e_t = (1-t) x + ty$
admits  the density 
\begin{align} \label{eq:pt}
p_t(x) &= \left(2\pi (1-t)^2 \right)^{-\frac{d}{2}} {\rm e}^{-\frac{\|x\|^2}{2(1-t)^2}}, \quad t \in [0,1), \quad {p}_0 = \mathcal{N}(0,I_d),
\end{align}
and with the velocity field induced by $\alpha$, see \cite[Example 4.12]{WS2025},
\begin{align}\label{velo-ind}
  v_t(x)  &= \nabla \Big( \frac{1-t}{t}\log p_t(x) + \frac{1}{2t} \|x\|^2 \Big) 
  = -\frac{x}{1-t},
\end{align}
it fulfills
the continuity equation \eqref{eq:ce}. 
Unfortunately, \eqref{eq:explode_2} is \emph{not} fulfilled due to the singularity at $t=1$.
Nevertheless, the trajectory $\varphi_t$ of the ODE \eqref{eq:flow_ode}  is  given by $$\varphi_t(x) = (1-t)x,$$ 
and therefore it holds $v_t(\varphi_t(x)) = -x$. Hence, we have a constant velocity field along the curve, and by switching to spherical coordinates we see that
$$
\|v_t\|_{L_2(p_t)}^2 = \int_{\R^d} \|{v}_t({\varphi}_t(x))\|^2 \d {p}_0(x)
= \int_{\R^d} \|x\|^2 \d {p}_0(x)= d,
$$
such that \eqref{eq:explode_1} is still valid.

\paragraph{2. Diffusion flows.}
\emph{Diffusion models} or \emph{score-based models} are based on the Fokker-Planck equation of the Wiener process or Ornstein-Uhlenbeck process: the diffusion equation (with drift), see e.g.\ \cite{song2021scorebasedgenerativemodelingstochastic}.
Let us consider the probability flow $p_t$ solving the standard diffusion equation 
\begin{equation}\label{standard-diffusion}
    \partial_t p_t = \frac{\sigma^2}{2} \, \nabla \cdot (p_t \nabla \log p_t ) =  \frac{\sigma^2}{2} \, \Delta p_t, \quad t \in (0,1], \quad \lim_{t \downarrow 0} p_t = \delta_0,
\end{equation}
starting in the target $\delta_0$ given by a Dirac distribution at 0. Here, the limit for $t \downarrow 0$ is taken in the sense of distributions and $\frac{\sigma^2}{2}$ is the diffusion constant.
The solution is analytically given by
\begin{equation}
    p_t(x) = (2 \pi \sigma^2 t)^{-\frac{d}{2}} {\rm e}^{-\frac{\|x\|^2}{2 \sigma^2 t}}, \quad t \in (0,1].
\end{equation}
We choose $\sigma^2 = 1$ 
such that the latent is $p_1 = \mathcal{N}(0,I_d)$. 
The velocity field is $v_t(x) = \frac{x}{2t}$.
After the reparameterization $\tilde{p}_t = p_{1-t}$, the reverse flow reads as
\begin{equation}
    \tilde{p}_t(x) = \big( 2\pi (1-t) \big)^{-\frac{d}{2}} {\rm e}^{-\frac{\|x\|^2}{2(1-t)}}, \quad t \in [0,1), \quad \tilde{p}_0 = \mathcal{N}(0,I_d),
\end{equation}
and satisfies $\tilde{p}_1 = \delta_0$ in a distributional sense. 
The reverse velocity field is given by
\begin{equation} \label{velo2}
\tilde{v}_t(x) = -\frac{x}{2(1-t)},
\end{equation}
and the reverse trajectory $\tilde{\varphi}_t$ solving the ODE \eqref{eq:flow_ode} 
by
$$\tilde{\varphi}_t(x) = x \sqrt{1-t}.$$
Here it holds 
$\tilde{v}_t(\tilde{\varphi}_t(x)) = - \frac{x}{2 \sqrt{1-t}}$.
Thus, again by switching to spherical coordinates, we see that
\begin{align}\label{diffusion-local-explosion}
    \|\tilde{v}_t\|_{L_2(\tilde{p}_t)}^2
    &=  \int_{\R^d} \|\tilde{v}_t(\tilde{\varphi}_t(x))\|^2 \d \tilde{p}_0(x)
     %=  \int_{\R^d} \frac{1}{4(1-t)}\|x\|^2 \d \tilde{p}_0(x)
     =  \frac{d}  {4(1-t)},
\end{align} 
%which could be also seen from the second moment of $\tilde p_t$.
so that both conditions \eqref{eq:explode_1} and \eqref{eq:explode_2} are violated, and $\tilde p_t \notin AC^2([0,1);\mathcal P_2(\R))$.

In practice, instability issues concerning the learning of $\tilde{v}_t$ are caused by this explosion at times close to the target, and need to be avoided by e.g.\ \emph{time truncations}, see e.g.\ \cite{soft2021}.
For a heuristic analysis that also includes drift-diffusion flows, we refer to \cite{P2022}.

\paragraph{3. Flows from linear interpolation of Boltzmann energies.}
For Boltzmann densities 
$p_t = {\rm  e}^{-f_t}/Z_t$, 
where $Z_t \coloneqq \int {\rm  e}^{-f_t} \d x$,
straightforward computation gives the Rao-Fisher differential equation
\begin{align} \label{fr1}
 \partial_t p_t 
&=- \left(\partial_t f_t  - \E_{p_t}[\partial_t f_t] \right) \, p_t.
\end{align}
Since on the other hand, the continuity equation becomes
\begin{equation} 
\partial_t p_t = -\nabla\cdot(p_t v_t) = 
- \left( -\langle \nabla f_t, v_t \rangle + \nabla\cdot v_t \right) p_t,
\end{equation}
the vector field has to fulfill
\begin{equation} \label{hi} 
\partial_t f_t  - \E_{p_t}[\partial_t f_t] = -\langle \nabla f_t, v_t \rangle + \nabla\cdot v_t.
\end{equation}  
    Sampling from unnormalized target Boltzmann densities  $p_1$  
    often uses the geometric density interpolation  
    $p_t = p_0^{1-t} p_{1}^t$ 
    which corresponds to the \emph{linear interpolation} 
    $
    f_t = (1-t) f_0 + t f_{1}      
    $. Then \eqref{hi} reads as
      \begin{align}  \label{heat_pde_f_constant}
 f_1 - f_0 - Z_t  +\langle \nabla f_t,v_t\rangle - \nabla\cdot v_t = 0,
\end{align}
where $Z_t$ is the normalization constant.
There are a variety of approaches for approximating the vector field corresponding to the linear interpolation, see e.g.\ \cite{MM2024,nüsken2024steintransportbayesianinference} for kernel-based methods.  However, for certain asymmetric constellations of prior and target measure, a mode collapse of the neural model was numerically observed in \cite{máté2023learning}. 
    Indeed, it was analytically shown in \cite{CWDS2025}   that for the Laplacian $p_0 \propto e^{-|x|}$ centered at $0$ and a Laplacian-like target distribution $p_1 \propto e^{-2\min\{|x|, |x-m|\}}$ with a mode at $0$ and a second one at $m > 0$, the velocity norm grows exponentially with $m$, i.e.\ $\|v_1\|_{L_2(p_1)}^2 \gtrsim  e^{\alpha m}$. 
    Using Gaussian mixtures instead of Laplacians, an explosion of the type $\int_0^1 \|v_t\|_{L_2(p_t)}^2 \d t \gtrsim m^4 e^{\alpha m^2}$ was proven in \cite[Proposition 9]{GTC2025}.
    
    Note that by replacing the linear interpolation with a more flexible variant including an additional learnable term, the mode collapse was numerically avoided in \cite{máté2023learning}.
    For a control-theoretic perspective on \eqref{hi} that leverages the explosion effect, we refer to \cite{DVK2022,BBRN2025, sun2024dynamicalmeasuretransportneural}.

\paragraph{4. Poisson flows.}
In \cite{Poisson2022}, another physics-inspired generative model  
was proposed using the Poisson equation on $\R^{d}$ augmented by an additional dimension $z$:
\begin{equation}\label{Poisson-eq}
    -\Delta \varphi (x,z) = p_{\rm data}(x) \delta_0(z), \quad x\in \R^d, z \in \R.
\end{equation}
Here, the target distribution $p_{\rm data}$ is interpreted as electrical charges on the $z=0$ hyperplane describing an electric field $\varphi$. The time-independent, but $z$-dependent velocity field $v(x,z) = -\nabla \varphi(x,z)$ 
then determines the particle trajectories through the $(d+1)$-dimensional space. This augmentation has been found crucial for the model to work properly, but has the following explosion of the velocity field as a consequence:
the Poisson velocity field $v(x,z)$ has the analytical expression
\begin{equation}\label{poisson-field}
    v(x,z) = \frac{1}{S_d} \int_{\R^{d+1}} \frac{(x,z) - (y, z')}{\|(x,z) - (y, z')\|^{d+1}} 
    p_{\rm data}(y) \delta_0(z') \d y \d z',
\end{equation}
where $S_d$ is the surface area of the $d$-sphere. For the uniform distribution $p_{\rm data}$ on  $[0,1]^d$, this results for any $x \in [0,1]^d$ and $z=0$ in
\begin{align}\label{poisson-explosion}
    \int_{\R^{d+1}} \Big \| \frac{(x,0) - (y, z')}{\|(x,0) - (y, z')\|^{d+1}} 
    p_{\rm data}(y) \delta_0(z')  \Big\| \d y \d z' 
    &=  \int_{[0,1]^d} \frac{1}{\|x - y\|^{d}}  \d y = \infty.
\end{align}
Hence, the Poisson field \eqref{poisson-field} might not be well-defined on the support of the target distribution, and can get arbitrarily large around it. Notice how the augmentation by $z$ causes this explosion \eqref{poisson-explosion}, while in the non-augmented case the corresponding integral stays finite.

Similarly to the diffusion case, the explosion \eqref{poisson-explosion} at heights $z$ close to the data ($z=0$) makes suitable \emph{truncations} necessary when approximating the velocity field by a neural network, see \cite[Ch.\ 3.2]{Poisson2022}.

In the next sections,
we introduce a physics-inspired generative model based on the damped wave equation which does not suffer from exploding velocity fields.

%-----------------------------------------------------------
\section{Telegrapher's Equation and Kac Model in 1D} \label{sec:kac1d}
%-----------------------------------------------------------
In \cite{GenPhys2023}, it was contemplated to consider other physics-inspired models next to above diffusion (parabolic) and Poisson (elliptic) equations for generative modeling. 
Here, we further pursue this idea and focus 
on a process related to the \emph{telegrapher's equation}\footnote{Historically, the equation \eqref{tele-eq-1D} originates from the modeling of voltage and current in electrical transmission lines (telegraphs).}, also known as \emph{damped wave equation}. This is a linear hyperbolic PDE of the form
\begin{align}\label{tele-eq-1D}
    \partial_{tt} u(t,x) + 2a \, \partial_{t} u(t,x) &= c^2 \partial_{xx} u(t,x), \quad x \in \R, ~ t > 0, \nonumber\\
    u(0,x) &= f_0(x), \\
    \partial_t u(0,x) &= 0.
\end{align}  
Here, $a>0$ is the damping coefficient, and $c>0$ describes the velocity of the wave front. 
It is known that if the damping $a$ and velocity $c$ go suitably to infinity, see Remark \ref{a-c-to-infinity}, then the diffusion model is retrieved.
Hence, diffusion can be seen as \emph{``an infinitely $a$-damped wave with infinite propagation speed $c$''}. The idea of particles traveling with infinite speed has already encountered resistance in the physics community \cite{C1958, V1958, C1963, MW1996, TL2016}, since it violates Einstein's laws of relativity. 
But also in the field of generative modeling, this can be seen as a reason for the explosion  of velocity fields of diffusion models.
Hence, if we restrict ourselves to above model \eqref{tele-eq-1D}, where waves travel only with \emph{finite} speed $c$, we expect the velocity field $v_t$ to behave well. 
Indeed, we will show that the telegrapher's model does not admit a velocity explosion, but shares the same long-time asymptotics as diffusion. 

Of central importance for the theoretical analysis and for generative modeling is the stochastic representation of \eqref{tele-eq-1D} given by  Kac \cite{KAC1974} which we recall next. 

\subsection{Kac Flow in 1D}
Suppose that a particle starts at $x_0 = 0$ and always moves with velocity $c > 0$ in either positive or negative direction. The initial direction is determined by a random symmetric coin-flip. Each step is of duration $\Delta t$ and therefore we have $\Delta x = c \Delta t$. 
The particle continues to move in the same direction with probability $1- a\Delta t$, and reverses with the probability $a\Delta t$, where $a>0$ is a fixed parameter. 
Intuitively, for small $\Delta t$ the particle tends to move in the same direction (resembling some form of inertia in the movement), until it collides with some obstacle and reverses its direction. 

For $n \in \mathbb N$, we consider a random variable $S_n$ taking the possible displacement of a particle after  $n$ steps as values. Let us first consider the case that the initial direction is positive, so that $S_0 = c \Delta t$, where we consider this as step 0.
Then we have
\begin{align}\label{displacement-discrete_0}
    S_n = c \Delta t (1 + \varepsilon_1 + \varepsilon_1 \varepsilon_2 + \ldots + \varepsilon_1 \ldots \varepsilon_n) ,
\end{align}
where $\varepsilon_k$ are random variables taking the value
$-1$ with probability $a \Delta t$ and $1$ with probability $1- a \Delta t$, or alternatively
\begin{align}\label{displacement-discrete}
S_n = c \Delta t \sum_{k=0}^n (-1)^{N_k},
\end{align}
where $N_0 = 0$ and 
$N_k \sim B(k,a\Delta t)$, $k \in \mathbb N\setminus \{0\}$, are the binomially distributed random variables denoting the number of direction changes within step 1 to $k$. 
We consider the continuous limit of this discrete process 
as $n \to \infty, \, \Delta t \to 0$ such that $n a \Delta t \to at$. Then it is well known that $N_n$ 
converges (in distribution) to the Poisson distribution ${\rm Poi}(at)$ with expectation value $at$. 
%i.e.
%$$
%\binom{n}{k} (a \Delta t)^k(1-a \Delta t)^k \to \frac{(at)^k}{k!} {\rm e}^{-at}
%= P(X=k)
%$$
%for a Poisson distributed random variable $X \in \N$.
The number of reversals up to time $t$ is given by
a homogeneous \emph{Poisson point process} $\left( N(t) \right)_{t\ge 0}$ with rate $a$, i.e.
\begin{itemize}
    \item[i)] $N(0) = 0$;
    \item[ii)] The increments of $N(t)$ are independent, i.e. for every $m \in \N$ and all $0 \le t_0 < t_1< \ldots < t_m$ the random variables
$N(t_1) - N(t_0)$, $N(t_2) - N(t_1)$, \ldots, $N(t_m) - N(t_{m-1})$
are independent;
    \item[iii)] $N(t) - N(s) \sim {\rm Poi}\big(a(t-s) \big)$ for all $0 \le s < t$.
\end{itemize}
Then the corresponding particle displacement and continuous analogue of \eqref{displacement-discrete} is given by
\begin{align}\label{stepwise-linear}
    S(t) = c \tau_t \quad \text{with} \quad 
    \tau_t := \int_0^t (-1)^{N(s)} \d s.
\end{align}
We refer to $\tau_t$ as \emph{random time}.
Finally, for a symmetric random initial direction, the displacement at time $t$ is determined by the \emph{Kac process starting in} $0$,
\begin{equation} \label{eq:kac}
K(t) := {\rm B}_{\frac{1}{2}} S(t) = {\rm B}_{\frac{1}{2}} c \,\tau_t,
\end{equation}
where ${\rm B}_{\frac{1}{2}} \sim {\rm Ber}(\frac12)$ is a Bernoulli random variable
taking the values $\pm 1$.
Naturally, for a random variable $X_0$ we call 
\begin{equation} \label{eq:kac_1}
X_t := X_0 + K(t) = X_0 + {\rm B}_{\frac{1}{2}} S(t)
\end{equation} 
the \emph{Kac process starting in} $X_0$.

\subsection{Relation of Kac' Flow to 1D Telegrapher's Equation}
Similar to the case of Brownian motions and diffusion equations, there is a Feynman-Kac type relation between the Kac process and the tele\-grapher's equation in 1D. Let $H^2(\R)$ denote the space
of two times weakly differentiable functions with derivatives in $L_2(\R)$.

\begin{theorem}\label{insertion-1D}
  For any initial function $f_0 \in H^2(\R)$, the function
\begin{equation}\label{tele-sol-1D}
    u(t,x) = \frac{1}{2}  \ex \Big[f_0 \big(x + S(t) \big) \Big] + \frac12 \ex \Big[f_0 \big(x - S(t) \big) \Big]  
\end{equation}
solves the 1D telegrapher's equation \eqref{tele-eq-1D}.

Further, if $f_0$ is a probability density of a random variable $X_0$ independent of $S(t)$, 
then \eqref{tele-sol-1D} is the probability density of the Kac process $(X_t)_{t \ge 0}$ starting in $X_0 \sim f_0$ defined by \eqref{eq:kac_1}.
\end{theorem}

\begin{proof}
The first part of claim was originally conjectured by  Kac \cite{KAC1974} and proven in \cite{K1993} using martingales.
For the second part, note that the Kac process starting in $X_0 \sim f_0$ with initial positive direction is given by $X_0 + S(t)$ and with initial negative direction by $X_0 - S(t)$. Then we obtain with the usual push-forward $P_{S(t)} := S(t) _\sharp \mathbb P$ that
    \begin{align} 
       \ex \Big[f_0 \big(x - S(t) \big) \Big]  
       &= \int_\Omega f_0(x - S(t)(\omega)  \d \mathbb{P}(\omega) 
       = \int_{S(t)(\Omega)} f_0(x - y ) \d P_{S(t)}(y)\\
       &= (f_0 \ast P_{S(t)})(x)  
       = p_{X_0 + S(t)}(x),
    \end{align}
    where we have used the independence of $X_0$ and $S(t)$ in the last equality.
    Similarly, we get
     \begin{align}
       \ex \Big[f_0 \big(x + S(t) \big) \Big]  
       &= (f_0 \ast P_{-S(t)})(x)  = p_{X_0 - S(t)}(x).
    \end{align} 
    Hence, it follows 
    $u(t,x) = \frac{1}{2}(p_{X_0 + S(t)}(x) + p_{X_0 - S(t)}(x)) = p_{X_0 + {\rm B}_{\frac{1}{2}} S(t)}(x)$ by the law of total probability.
\end{proof}
By the next lemma, there is an analytic representation of the distributional solution of the telegrapher's equation \eqref{tele-eq-1D}. We refer to \cite{MW1996}, and to \cite{N2020} for the proof via distributions. For a characterization of  measures as (tempered) distributions, see also \cite[Chapter 4.4]{PlPoStTa23}.

\begin{lemma}\label{explicit-representation}
The telegrapher's equation \eqref{tele-eq-1D} starting in a generalized function (distribution) $f_0$ has the distributional solution
   \begin{align}\label{explicit-formula-1d}
       u(t,x) &= \frac12 e^{-at} \Big(
       f_0(x+ct) + f_0(x-ct)  + \beta ct \int_{x-ct}^{x+ct}  \frac{I_0'(\beta r_t(x-y) )}{r_t(x-y)} f_0(y) \d y \notag\\
       &\quad + \beta \int_{x-ct}^{x+ct} I_0(\beta r_t(x-y)) \, f_0(y)\d y
       \Big),
   \end{align}
    where 
    $$
    r_t(x) \coloneqq \sqrt{c^2t^2-x^2}, \quad \beta \coloneqq \frac{a}{c},
    $$
    and $I_k$ denotes the $k$-th \emph{modified Bessel function of first kind}. {\rm (}Note that $I_0' = I_1$.{\rm )}  \\   
    In particular, for the Dirac distribution $f_0 = \delta_0$ in zero, it holds 
   \begin{align}\label{telegraph-formula-1d-dirac}
       u(t,x) &= \frac12 e^{-at} \big( \delta_0(x+ct) + \delta_0(x-ct)  \big) + \tilde u(t,x),
   \end{align}
    where 
   \begin{align} \label{telegraph-formula-1d-dirac-cont}
     \tilde u(t,x) &\coloneqq \frac12 e^{-at} \Big(\beta ct \frac{I_0'(\beta r_t(x))}{r_t(x)}  
     + \beta I_0(\beta r_t(x))
     \Big) {1}_{[-ct,ct]}(x).
   \end{align}
   This distribution is associated to a  probability measure $\mu_t$ with an atomic part given by the Dirac measures $\delta_{-ct}, \delta_{ct}$ and an absolutely continuous part $\tilde u(t,\cdot)$, i.e.
   $$
   \mu_t (A) =  \frac{e^{-at}}{2}\left( \delta_{-ct}(A) + \delta_{ct}(A) \right) 
   +  \int_{A} \tilde u(t,x) \d x
   \quad \text{for all Borel sets } A \subset \R.
   $$
  The measure $\mu_t$ is exactly the probability distribution of the Kac process $K(t)$ starting in $0$ from \eqref{eq:kac}, see \cite[Section 2]{J1990}.
\end{lemma}
Note that the Kac process $K(t)$ starting in $0$ does not admit a density itself.\\
The following remark states an important relation with the left- and right-moving Kac processes.

\begin{remark}
By \cite[Section 2]{J1990}, the absolute continuous part of the distribution of the \emph{initially right-moving} and \emph{initially left-moving} Kac process $\pm S(t)$ starting in $0$ is given by
    \begin{align}
     \tilde u_{\pm}(t,x)
     &= \frac12 e^{-at} \Big( \frac{\beta (ct \pm x)}{r_t(x)}\,I_{0}'(\beta r_t(x)) + \beta \,I_{0}(\beta r_t(x)) \Big) {1}_{[-ct,ct]}(x),
\end{align}
respectively, and it holds $\tilde u = \frac12 (\tilde u_+ + \tilde u_-)$
with $\tilde u$ in \eqref{telegraph-formula-1d-dirac-cont}.
By \cite[Eq.\ (12) \& (11)]{KAC1974}, the marginal densities $\tilde u_{\pm}$ satisfy the PDEs
\begin{align}\label{source-sinks}
  \partial_{t} \tilde u_{+} + c\,\partial_{x} \tilde u_{+} = -a\, \tilde u_{+}+a\, \tilde u_{-},
\quad
\partial_{t} \tilde u_{-}-c\,\partial_{x} \tilde u_{-}= a\,\tilde u_{+}-a\,\tilde u_{-}.  
\end{align}
Intuitively, the equations \eqref{source-sinks} can be understood as continuity equations for the \emph{currently right-moving} and \emph{currently left-moving particles} of $K(t)$ including source and sink terms, see \cite{OH2000}.
Adding the equations  yields
\begin{align*}
\partial_{t}(\tilde u_{+}+\tilde u_{-})
+\,c\,\partial_{x} (\tilde u_{+}
-\,\tilde u_{-})
& =0.
\end{align*}
Thus, with the \emph{continuous flux}
\begin{align}\label{flux-formula}
  \tilde J(t,x)
  &\coloneqq \frac{c}{2}\left(\tilde u_{+}(t,x) - \tilde u_{-}(t,x)\right)
  = \frac12 e^{-at} \,\frac{\beta c x}{r_t(x)}\,I_{0}'(\beta r_t(x)), \quad x \in (-ct,ct),
\end{align}
we arrive at the continuity equation
\begin{equation}\label{flux-continuity-equation}
\partial_{t}\tilde u+\partial_{x}\tilde J=0.
\end{equation}
\hfill $\diamond$
\end{remark}

After these preparations, we can explicitly describe the velocity field of the Kac process starting in a Dirac point which will be important for learning velocity fields later. The Kac velocity in \eqref{velo_1d} may be compared with those of coupling and diffusion flows given in \eqref{velo-ind}
and \eqref{velo2}, respectively.

\begin{theorem} \label{velo_dirac}
   Let $f_0 =  \delta_0$ be the Dirac measure at $0$ in $\R$. Then $\mu_t$ associated
   to  $u(t,\cdot)$ in \eqref{telegraph-formula-1d-dirac} solves the continuity equation   \begin{equation}\label{continuity_eq_dirac}
     \partial_t \mu_t = - \partial_x (\mu_t v_t)  
   \end{equation}
   in a weak sense, where the velocity field $v_t(x) = v(t,x)$ is given by
   \begin{align} \label{velo_1d}
       v(t,x) &\coloneqq 
       \left\{
       \begin{array}{cl}
           \frac{x}{
           t +  \frac{r_t(x)}{c} \frac{I_0(\beta r_t(x))}{ I_0'(\beta r_t(x))} }
           & \text{if} \quad x \in (-ct, ct),\\
           \; \; \, c &\text{if} \quad x=ct, \\
           -c & \text{if} \quad x=-ct,\\
            \; \; \, \text{arbitrary} & \text{otherwise}.
       \end{array}
       \right.
   \end{align}   
  \end{theorem}

\begin{proof}   
   At this time, we will \emph{assume} the existence of a velocity field $v_t$ satisfying the continuity equation \eqref{continuity_eq_dirac} in a weak sense, and show that in this case $v_t$ must admit the form \eqref{velo_1d}. Later, in Proposition \ref{prop:lipschitz}, we generally and independently prove that $\mu_t$ is indeed absolutely continuous and admits a velocity field $v_t$.

   Now, given some velocity field $v_t$ satisfying \eqref{continuity_eq_dirac}, we consider $\tilde v_t \coloneqq \frac{1}{2} (v_t(x) - v_t(-x))$, which is antisymmetric by definition. Using that $v_t$ solves \eqref{continuity_eq_dirac} and the symmetry of $u(t,x)$ in $x$, it follows that also $\tilde v_t$ solves \eqref{continuity_eq_dirac}. Hence, in the following, we can assume the existence of some \emph{antisymmetric} velocity field $v_t$ solving the continuity equation \eqref{continuity_eq_dirac}.
   
   Fix $t > 0$. Formal differentiation of \eqref{telegraph-formula-1d-dirac} leads to
   \begin{align}\label{partial-t}
       \partial_t u(t,x) =
       &-a u(t,x) + 
       \frac12 e^{-at} \big( c\delta_0'(x+ct) - c\delta_0'(x-ct) \\
       &+ (a +a^2 t ) z_1(t,x) + at z_2(t,x) 
       \big) {1}_{[-ct,ct]}(x),
   \end{align}
   where 
   \begin{align}\label{z1_z2}
       z_1(t,x) \coloneqq \frac{ I_0'(\beta r_t(x))}{r_t(x)}
       \quad \text{and} \quad
       z_2(t,x) \coloneqq \frac{act \, I_0''(\beta r_t(x) ) - c^2t \, z_1(t,x)}{r_t(x)^2}.
   \end{align}
   First, we consider the case $x \in (-ct,ct)$.
   Then, all the Dirac terms disappear and we only need to work with $\tilde u$ given in \eqref{telegraph-formula-1d-dirac-cont}, and $\partial_t \tilde u$ is given in \eqref{partial-t} without the Dirac terms.
   Using l' Hospital's rule and 
   $I_0'(0) = I_0^{(3)}(0)= 0$, 
   $I_0''(0) = \frac{1}{2}$,  $ I_0^{(4)}(0) = \frac{3}{8}$,
   it is easy to check that both functions $z_1, z_2$ are bounded in $x \in (-ct, ct)$, so that $\partial_t \tilde u(t,\cdot)$ is integrable.
   %,  see Appendix \ref{proof_thm_veloDirac}.    
   By our above arguments, the velocity field $v(t,x)$ can be assumed to be antisymmetric in $x$. Together with the symmetry of $\tilde u(t,x)$ in $x$, we conclude by \eqref{continuity_eq_dirac} that
   \begin{equation}\label{one}
       2 \tilde u(t,x) v(t,x) = - \int_{-x}^x \partial_t \tilde u(t,y) \d y \quad \text{ for all } x \in (-ct, ct).
   \end{equation}
At the same time, we obtain from \eqref{flux-continuity-equation} and by the anti-symmetry of $\tilde J(t,\cdot)$ that
     \begin{equation}
       - \int_{-x}^x \partial_t \tilde u(t,y) \d y 
       = \int_{-x}^x \partial_x \tilde J(t,y) \d y 
       = 2 \tilde J(t,x)
       \quad \text{ for all } x \in (-ct, ct).
   \end{equation}
Together with \eqref{one},  \eqref{flux-formula} and \eqref{telegraph-formula-1d-dirac-cont}, this implies
   \begin{equation}
      v(t,x)
    =\frac{\tilde J(t,x)}{\tilde u(t,x)}
    = \frac{e^{-at}\,\frac{\beta c x}{r_t(x)}I_{0}'(\beta r_t(x))}
           {e^{-at}\!\bigl[\frac{\beta ct}{r_t(x)}I_{0}'(\beta r_t(x))+\beta I_{0}(\beta r_t(x))\bigr]}
    =\frac{x}{t + \frac{r_t(x)}{c} \, \frac{I_{0}(\beta r_t(x))}{I_{0}'(\beta r_t(x))}},
   \end{equation} 
   which shows \eqref{velo_1d} for all $x \in (-ct,ct)$.
   
   Next, we consider the case $x = \pm ct$. For any test function $\phi \in C_c^\infty(\R)$, we get
\begin{align}
    \int_{-\infty}^\infty \partial_t u(t,x) \phi(x) \d x
    &= - \frac{ae^{-at}}{2} (\phi(-ct) + \phi(ct)) + \frac{e^{-at}}{2} (-c \phi'(-ct) + c \phi'(ct)) \\
    & \quad +\int_{-ct}^{ct} \partial_t \tilde u (t,x) \phi(x) \d x.
\end{align}
On the other hand, the continuity equation \eqref{continuity_eq_dirac} gives
\begin{align}
    \int_{-\infty}^\infty \partial_t u(t,x) \phi(x) \d x
    &= \int_{-\infty}^\infty - \partial_x (u(t,x) v(t,x)) \phi(x) \d x
    = \int_{-\infty}^\infty u(t,x) v(t,x) \phi'(x) \d x \\
    &= \frac{e^{-at}}{2} (v(t, -ct) \phi'(-ct) + v(t, ct) \phi'(ct) ) + 
    \int_{-ct}^{ct} \tilde u(t,x) v(t,x) \phi'(x) \d x.
\end{align}
Now we can choose a sequence of test functions $\phi_n \in C_c^\infty(\R)$ such that $\phi_n \equiv 0$ on $[-ct + \frac{1}{n}, ct - \frac{1}{n}]$, $\|\phi_n\|_\infty < \frac{1}{n}$, $\phi_n'(ct) = 1$ and $\phi_n'(-ct) = -1$. Then the anti-symmetry $v(t, -ct) = -v(t, ct)$ and above identities in the limit $n \to \infty$ yield
\begin{align}
  \frac{e^{-at}}{2}  (-c \cdot (-1) + c \cdot 1) = \frac{e^{-at}}{2} (-v(t, ct) \cdot (-1) + v(t,ct) \cdot 1),
\end{align}
and hence, $v(t, ct) = c$. By anti-symmetry, it holds $v(t, -ct) = -c$, and the proof of \eqref{velo_1d} is finished. 
\end{proof}

%----------------------------------------------------------------------------------------
\section{Multi-dimensional Generalizations of the \texorpdfstring{Telegrapher's \\Model}{Telegrapher's Model}} \label{sec:kac_multi}
%----------------------------------------------------------------------------------------
The  connection between the telegrapher's equation and the Kac process can be generalized to higher dimensions via Kac's insertion method, which we briefly
describe in Subsection \ref{subsec:kac_multi_1}. In contrast to the one-dimensional case, initial probability densities do not remain probability densities in the PDE formulation. Yet, we state a convergence result towards diffusion, which we will use subsequently.
To circumvent the lack of mass conservation, we propose two types of componentwise extensions in Subsection \ref{subsec:kac_multi_2}, and we prove regularity and convergence results for our componentwise processes.

%-------------------------------------------------------------
\subsection{Kac's Insertion Method}\label{subsec:kac_multi_1}
%-------------------------------------------------------------
By inserting the random time \eqref{stepwise-linear} into the undamped wave solution, Theorem \ref{insertion-1D} can be generalized to multiple dimensions.

\begin{theorem}\label{insertion-multi-d}
For any initial function $f_0 \in H^2(\R^d)$, $d \ge 1$, let  $w_c(t,x)$ be
the solution of the \emph{undamped} wave equation with velocity $c > 0$ given by
\begin{align}\label{undamped-wave}
    \partial_{tt} w(t,x) &= c^2 \Delta w(t,x), \quad x \in \R^d, ~ t > 0,\nonumber\\
    w(0,x) &= f_0(x),\\
    \partial_t w(0,x) &= 0.
\end{align}
Then, the function $u: [0,\infty) \times \R^d \to \R$ defined by
\begin{equation}\label{tele-sol-general}
    u(t,x) := \ex \big[ w_c(\tau_t, x)\big]
\end{equation}
solves the multi-dimensional damped wave equation 
\begin{align}\label{damped-wave}
    \partial_{tt} u(t,x) + 2a \, \partial_{t} u(t,x) &= c^2 \Delta u(t,x), \quad x \in \R^d, ~ t > 0,\nonumber\\
    u(0,x) &= f_0(x), \\
    \partial_t u(0,x) &= 0.
\end{align} 
\end{theorem}

This claim was again conjectured in \cite{KAC1974} and a proof was given in \cite[Theorem 4]{GH1971} using semigroup theory. \hfill $\diamond$

The following observation states a \emph{defect} of the multi-dimensional telegrapher's equation \eqref{damped-wave} and crucially motivates our componentwise approach proposed in the next Subsection \ref{subsec:kac_multi_2}.

\begin{remark}\label{drawback-no-density}
  In 1D, the solution of \eqref{undamped-wave} is given by $w_c(t,x) = \frac{1}{2} (f_0(x+ct) + f_0(x - ct) )$, so that \eqref{tele-sol-general} reproduces \eqref{tele-sol-1D}.
Unfortunately, in higher dimensions, the telegrapher's PDE solution \eqref{tele-sol-general} might not stay a probability density anymore.
For example, the mass $\int_{\R^3} u(t,x) \d x$ in 3D behaves like $1-e^{-t}$, see \cite[Ch.\ 2.2]{TL2016}. \hfill $\diamond$
\end{remark}

Nevertheless, we will see that using an analogous insertion method for the diffusion, it is possible to show that the diffusion and telegrapher's PDE solutions converge to each other.

Recall that a Wiener process (Brownian motion) 
$\mathbf{W} = \left(\mathbf{W}(t) \right)_{t\ge 0}$, $\mathbf{W}(t) \in \R^d$, $d \ge 1$, is characterized by i) $\mathbf{W}(0) = \mathbf{0}$,
ii) independent increments, iii) $\mathbf{W}(t) - \mathbf{W}(s) \sim \mathcal N(0,(t-s)I_d)$
for all $0 \le s < t$, and iv) $\mathbf{W}$ is almost surely continuous in $t$.
From now on, we emphasize vector-valued quantities using bold letters.

\begin{lemma}\label{insertion-diffusion}
For $f_0 \in H^2(\R^d)$,
let  $w$ be the solution of the undamped wave equation
\eqref{undamped-wave}, where now, $c = 1$ is fixed.
Let $W$ be a  one-dimensional Wiener process.
Then
\begin{align} \label{dichte_heat}
    h(t,x):= \mathbb{E}\left[w \left(\sigma W(t),x \right) \right], \quad \sigma >0,
\end{align}
solves the heat equation
\begin{align}\label{heat-eq}
    \partial_t h(t,x) &= \frac{\sigma^2}{2} \Delta h(t,x), \quad x \in \R^d, ~ t > 0,\\
    h(0,x) &= f_0(x).
\end{align}    
\end{lemma}

The statement can be found in \cite{J1990}, and a proof was given in \cite[Theorem 5]{GH1971}.\hfill $\diamond$

\noindent Note that if $f_0$ is a probability density and $\mathbf{X}_0 \sim f_0$,
then the solution $h$ given by \eqref{dichte_heat} is the probability density of a $d$-dimensional Wiener process $\sigma \mathbf{W}(t)$ starting in $\mathbf{X}_0$.

Interestingly the functions $u$ and $h$ obtained by the insertion methods \eqref{tele-sol-general} and \eqref{dichte_heat} converge to each other for $t \to \infty$. 
The proof of this fact crucially exploits the properties of the undamped wave solution \eqref{undamped-wave}. %L2-boundedness of w, and relations between w_c and w 

\begin{theorem} \label{convergence-heat}
Fix $c_0 > 0$. 
For $f_0 \in H^2(\R^d)$,
%let $w$ be the solution of the undamped wave equation \eqref{undamped-wave}
let $u$ be given by \eqref{tele-sol-general} with $a>0, \, c \ge c_0$, and $h$ by \eqref{dichte_heat} with $\sigma^2:= \frac{c^2}{a}$. 
Then there exists a constant $C = C(c_0, d)$ such that
\begin{align} \label{bound_convergence_heat}
    \|h(t, \cdot) - u(t,\cdot)\|_{L_2(\R^d)}\leq C \Big(\frac{1}{at}+\frac{1}{c(ta)^{\frac{1}{2}}}\Big) \quad \text{ for all } t > 0.
\end{align}
\end{theorem}

\begin{proof}
It is well-known that the Laplacian $\Delta : H^2(\R^d) \subset L_2(\R^d) \to L_2(\R^d)$ is a self-adjoint linear operator with non-positive spectrum $\sigma(\Delta) = (-\infty, 0]$.
    By \cite[Corollary 2]{Lutz1977} the solution of \eqref{undamped-wave} with $c=1$ is given by $w(t, \cdot) = C(t) f_0$, where $(C(t))_{t \ge 0}$ is a uniformly bounded so-called cosine operator function with $\|C(t)\|_{L_2 \to L_2} \le 1$ for all $t \ge 0$. In particular, it follows that $\|w(t,\cdot) \|_{L_2(\R^d)} \le \|f_0\|_{L_2(\R^d)}$ is uniformly bounded in $t$. Now, \cite[Theorem 2]{J1990} yields the claim.
\end{proof}

\begin{remark}\label{a-c-to-infinity}
Combining Theorems \ref{insertion-multi-d}, \ref{convergence-heat} and Lemma \ref{insertion-diffusion}, we see that for any $t \ge 0$, the solution $u^{a,c}(t,\cdot)$ of the telegrapher's equation \eqref{damped-wave} converges to the solution $h(t,\cdot)$ of the diffusion equation \eqref{heat-eq} 
for $a,c \to \infty$ with fixed $\sigma^2 = \frac{c^2}{a}$.\,\hfill $\diamond$
\end{remark}

%------------------------------------------------------------------------------------
\subsection{Multi-Dimensional Kac Process} \label{subsec:kac_multi_2}
%-----------------------------------------------------------------------------------
\begin{figure}[ht!] 
    \centering
    \includegraphics[width = 0.18\textwidth]{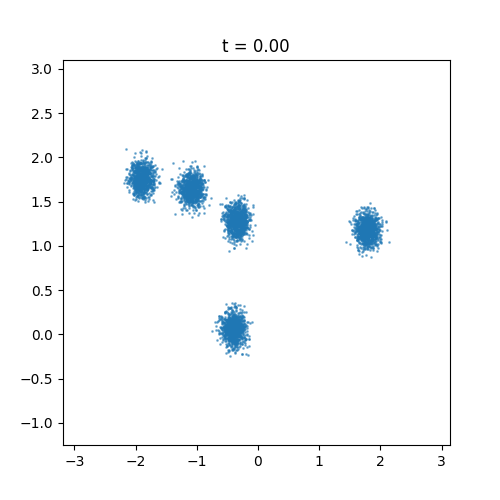}
    \includegraphics[width = 0.18\textwidth]{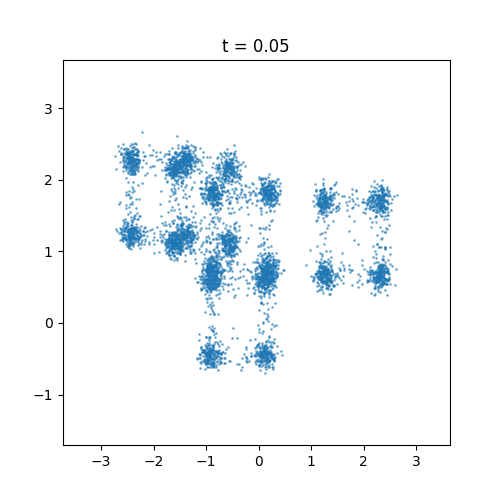}
    \includegraphics[width = 0.18\textwidth]{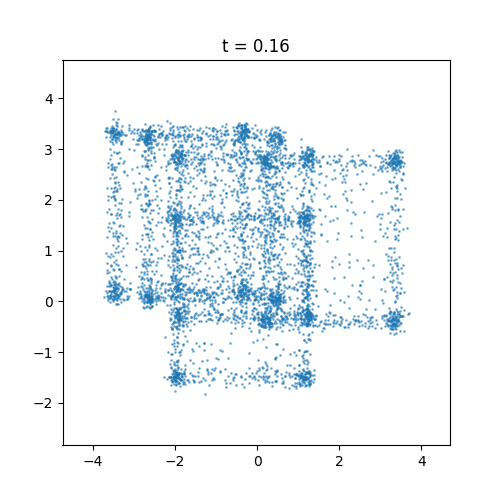}
    \includegraphics[width = 0.18\textwidth]{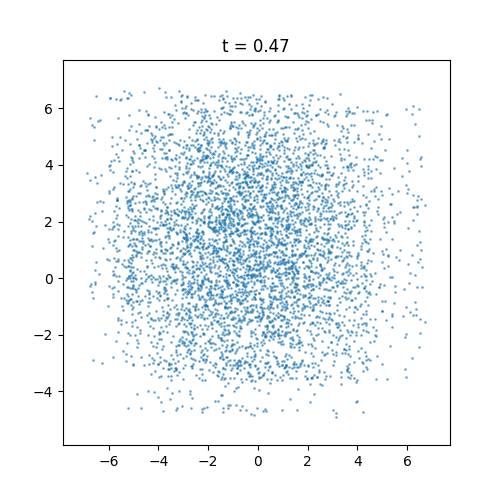}
    \includegraphics[width = 0.18\textwidth]{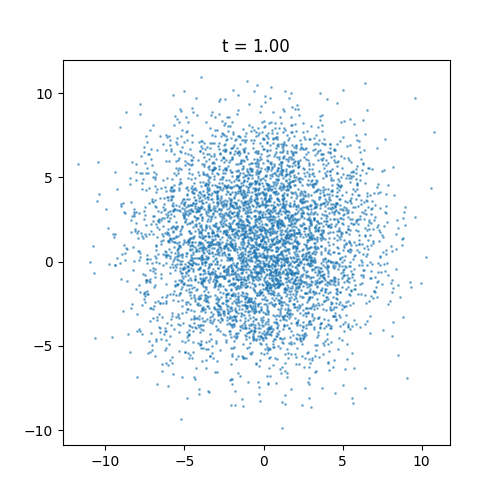}
    
    \includegraphics[width = 0.18\textwidth]{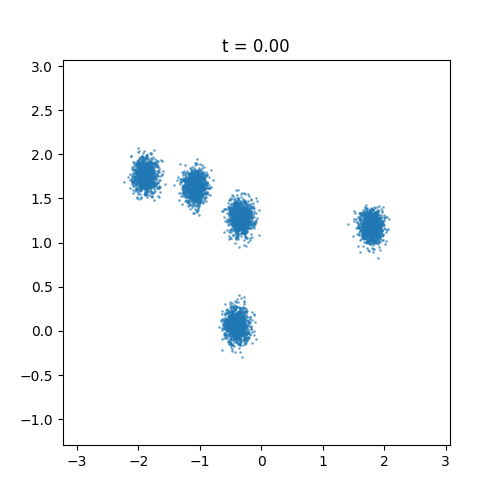}
    \includegraphics[width = 0.18\textwidth]{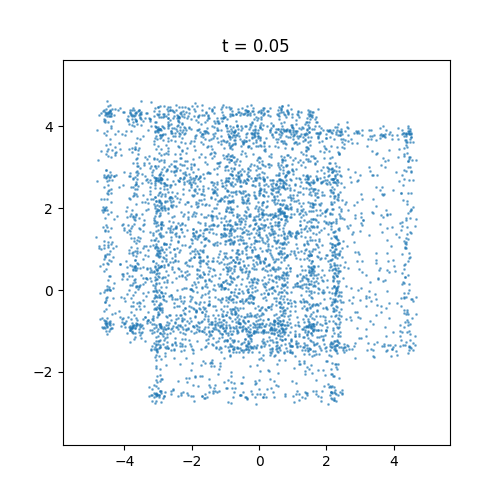}
    \includegraphics[width = 0.18\textwidth]{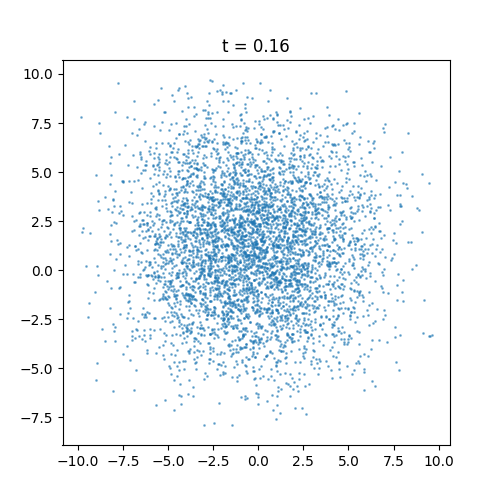}
    \includegraphics[width = 0.18\textwidth]{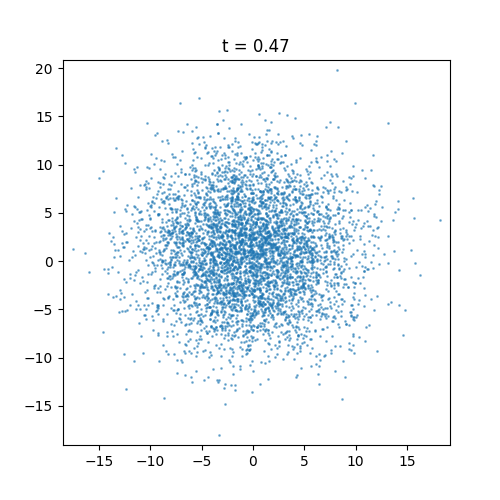}
    \includegraphics[width = 0.18\textwidth]{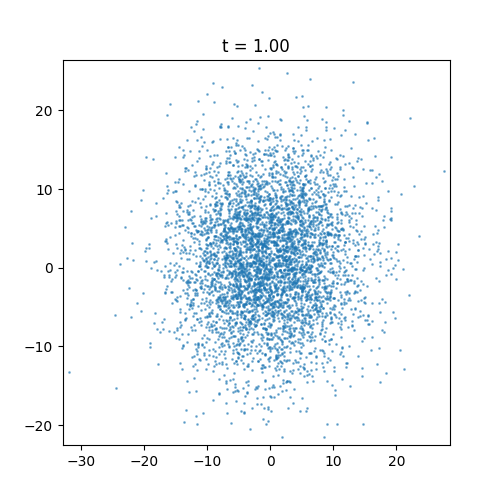}
    \caption{The distribution of the Kac process $\mathbf{X}_t$ starting in a 2D Gaussian mixture, simulated until time $T=10$ with $(a,c) =(1,1)$ (upper row) and $(a,c) = (5,5)$ (lower row). We observe a so-called \emph{ballistic-to-diffusive crossover} of the Kac process, see e.g.\ \cite{MW1996}.}
    \label{flow_low_ac}
\end{figure}
For generative modeling, we propose to circumvent the lack of mass conservation of the multi-dimensional telegrapher's equation \eqref{damped-wave}, see Remark \ref{drawback-no-density}, by using a multivariate random process which consists of independent 1D Kac processes in each spatial component -- in complete analogy to multi-dimensional Brownian motions. 

More precisely, we consider two types of componentwise processes, commonly referred to as \emph{variance-exploding} (VE) and \emph{variance-preserving} (VP). For both, we prove Lipschitz regularity in the Wasserstein-2 space in the Propositions \ref{prop:lipschitz} and \ref{prop:lipschitz_2}. 
Typically, in the VE case, the latent distribution is obtained by sampling the marginal distribution of the process in the long-time limit $t \to \infty$. Hence, mostly relevant for the VE case, we prove the convergence Theorem \ref{convergence-multi-d} based on Theorem \ref{convergence-heat}.

For a random variable $\mathbf{X}_0 \in \R^d$, $d \ge 1$, 
we consider i.i.d.\footnote{We use i.i.d.\ Kac processes just for simplicity. But our framework also allows for independent processes $K^i$ each with their own parameters $(a_i, c_i)$.}\ copies $K^i(t)$, $i = 1,...,d$, of a Kac process \eqref{eq:kac} starting in 
$\mathbf{0}$. Then, we call 
\begin{equation}\label{kac-process}
    \mathbf{K}(t) := (K^1(t), ..., K^d(t))
\end{equation}
a  \emph{$d$-dimensional Kac process starting in} $\mathbf{0} \in \R^d$. Corresponding to the VE case, we consider the respective \emph{Kac process starting in} $\mathbf{X}_0$, or \emph{VE Kac process}, given as
\begin{equation}\label{eq:VE-KAC}
    \mathbf{X}_t := \mathbf{X}_0 + \mathbf{K}(t), \quad t \in[0,\infty).
\end{equation}
For the VP case, we define the \emph{mean‐reverting} Kac process, or \emph{VP Kac process}, by
\begin{equation}\label{eq:VP-KAC}
  \mathbf{M}_t \coloneqq (1-t)\, \mathbf{X}_0 + \mathbf{K}(t),
  \quad t\in[0,1].
\end{equation}
The stepwise linear movement of the Kac process $(\mathbf{K}(t))_{t \ge 0}$ 
translates into a favorable regularity of both probability flows $(P_{\mathbf{X}_t})_t$ and $(P_{\mathbf{M}_t})_t$.
In particular, they are absolutely continuous in the Wasserstein space $\mathcal{P}_2(\R^d)$ and there is no local explosion of the velocity field as in \eqref{diffusion-local-explosion}. We first state the result for the VE Kac process \eqref{eq:VE-KAC}.

\begin{proposition}\label{prop:lipschitz}
   Let $P_{\mathbf{X}_0} \in \mathcal{P}_2(\R^d)$.
   Then, for the probability distribution flow  $(\mu_t)_{t\ge 0} \coloneqq (P_{\mathbf{X}_t})_{t\ge 0}$ of the $d$-dimensional Kac process $\mathbf{X}_t$ starting in $\mathbf{X}_0$,  the following holds true:
   \begin{itemize}
    \item[i)]
        $(\mu_t)_{t\ge 0}$ is Lipschitz continuous in the Wasserstein space $(\mathcal{P}_2(\R^d),W_2)$ with constant $d^{\frac{1}{2}} c$, and $\mathbf{X}_t$ is almost surely Lipschitz continuous in $\R^d$ with the same constant.
     \item[ii)] $(\mu_t)_{t\ge 0}$ is an absolutely continuous curve in $AC^2([0,T];\mathcal P_2(\R^d))$ for any $T>0$, and admits a velocity field fulfilling
     \begin{equation}\label{kac-boundedness}
    \|v_t\|_{L_2(\mu_t)}^2  \le d c^2 \quad \text{ for a.e. } t \in [0, \infty).
\end{equation}
    \item[iii)]
    The support of   $(\mu_t)_{t\ge 0}$ grows at most linearly in time.
   \end{itemize}
\end{proposition}

%Let $P_{\zb X_0}$ have compact support. %%% RICHARD: Teil iii) gilt ja auch z.B. wenn support X_0 = [0, \infty).

\begin{proof}
i)  Fix $0 \le s < t$. By construction of $\mathbf{X}_t$ via 
    \eqref{stepwise-linear}, we obtain
    for $\mathbb{P}$-a.e. $\omega$ that
    \begin{align}\label{lipschitz-estimate}
        \|\mathbf{X}_t(\omega) - \mathbf{X}_s (\omega)\|^2
        &= \|\mathbf{K}(t) (\omega)- \mathbf{K}(s) (\omega)\|^2
        = \sum_{i=1}^d |K^i(t)(\omega) - K^i(s)(\omega)|^2 \nonumber\\
        &
        = \sum_{i=1}^d \big|c \,  {\rm B}_{\frac12}^i (\omega) \int_s^t (-1)^{N^i(z)(\omega)} \d z
       \big|^2 \nonumber\\
        &\le d \, c^2|t-s|^2.
    \end{align}
    This shows that $\mathbf{X}_t$ is almost surely Lipschitz continuous in $\R^d$ with constant $d^{\frac{1}{2}} c$.

    Now, we
    consider the canonical transport plan  $\pi : \R^d \times \R^d \to \R$ defined by
    $\pi(B) := 
    \mathbb{P} \big( (\mathbf{X}_s, \mathbf{X}_t) \in B \big)$
    for all Borel measurable $B \in \mathcal B(\R^d \times \R^d)$.
    Indeed, it holds $\pi(A \times \R^d) = P_{\mathbf{X}_s}(A)$ 
    and 
    $\pi(\R^d \times A) = P_{\mathbf{X}_t}(A)$ 
    for all $A \in  \mathcal B(\R^d)$.
    Furthermore, we have
    $\pi \in \mathcal{P}_2(\R^d \times \R^d)$:
     it holds that
    \begin{align}
       &\int_{\R^d \times \R^d} \|(x,y)\|^2 \d \pi(x,y) 
       %= \int_{\Omega} \|(\mathbf{X}(s),\mathbf{X}(t))\|^2 \d \mathbb{P}
       = \int_{\Omega} \|(\mathbf{X}_0 + \mathbf{K}(s),\mathbf{X}_0 + \mathbf{K}(t))\|^2 \d \mathbb{P}\\
       &\le \int_{\Omega} 2 \|\mathbf{X}_0\|^2 + 2\|\mathbf{K}(s)\|^2 + 2 \|\mathbf{X}_0\|^2 + 2\|\mathbf{K}(t)\|^2\d \mathbb{P} 
       < \infty,
    \end{align}
    since $P_{\mathbf{X}_0} \in \mathcal{P}_2(\R^d)$, and by using \eqref{lipschitz-estimate} with $s=0$ or $t=0$, respectively.  
    Hence, $\pi$ is an admissible transport plan between $\mu_s$ and $\mu_t$.
    
    Then, we conclude, again using \eqref{lipschitz-estimate}, that
    \begin{equation}
        W_2^2(\mu_s, \mu_t) 
        \le 
        \int_{\R^d \times \R^d} \|x-y\|^2 \d \pi(x,y)
        = \int_{\Omega} \|\mathbf{X}_s - \mathbf{X}_t\|^2 \d \mathbb{P}
        \le d \, c^2|t-s|^2.
    \end{equation}
    Hence, $(\mu_t)_{t\ge 0}$ is Lipschitz continuous in $\mathcal{P}_2(\R^d)$ with constant $d^{\frac{1}{2}} c$.
    \\[1ex]
 ii) By definition \eqref{eq:abs_cont} and Part i), we see immediately that
 $(\mu_t)_{t\ge 0}$
 is an absolutely continuous curve.
 It then follows from \cite[Theorem 8.3.1]{BookAmGiSa05}
 that there exists an (optimal) velocity field in the
 continuity equation fulfilling
\eqref{kac-boundedness}.
 \\[1ex]
iii) 
 Let $\bar \Omega_0 \coloneqq \supp(P_{\mathbf{X}_0})$. 
 Consider the covering 
 $\bar \Omega_t \coloneqq \bigcup_{x \in \bar \Omega_0} \mathbb B_{d^\frac{1}{2} c t}(x)$ of $\bar \Omega_0$ with balls of radius $d^\frac{1}{2} c t$. By \eqref{lipschitz-estimate} it follows that $\mathbf{X}_t \in \bar \Omega_t$ $\mathbb{P}$-almost surely, yielding the claim $\supp(P_{\mathbf{X}_t}) \subseteq \bar \Omega_t$.
\end{proof}

Similar to the VE Kac process \eqref{eq:VE-KAC}, the VP Kac process \eqref{eq:VP-KAC} also admits a Lipschitz property in $\mathcal{P}_2$ and hence a bounded velocity norm $\|v_t\|_{L_2(\mu_t)}$. The statement and proof is given in Proposition \ref{prop:lipschitz_2} in the appendix.

By the next theorem  the convergence result
from Theorem \ref{convergence-heat} stated for the PDE solution \eqref{tele-sol-general} 
also holds true for the VE Kac process \eqref{eq:VE-KAC}. 

\begin{theorem}\label{convergence-multi-d}
Let $f_0$ be a probability density of a random variable $\mathbf{X}_0 = (X_0^1,...,X_0^d)$ with independent components such that $f_0 (x)= f_0(x^1, \ldots,x^d) = \prod_{i=1}^d f_0^i(x^i)$ and  
 $f_0^i \in H^2(\R)$, $i=1,\ldots,d$.
Let $\mathbf{X}_t$ be a $d$-dimensional Kac process starting in $\mathbf{X}_0$ with parameters $a>0$ and $c \ge c_0 >0$ and with  probability density $u_t$. 
Further, let $\sigma^2:= \frac{c^2}{a}$ be fixed and $\mathbf{X}_0 + \sigma \mathbf{W}(t)$ be a $d$-dimensional Wiener process starting in $\mathbf{X}_0$ with probability density $h_t$. 
Assume that $\mathbf{X}_0$ is independent of $\mathbf{K}(t)$ and $\sigma \mathbf{W}(t)$, respectively. 
Then it holds
\begin{align}
    \|h(t, \cdot)  - u(t,\cdot) \|_{L_2(\R^d)} 
    \le  C_{d} \left(\frac{1}{at^{1+\frac{d-1}{4}}}+\frac{1}{ca^{\frac{1}{2}} t^{\frac{1}{2}+\frac{d-1}{4}}}\right) + R_d(a,c,t),
\end{align}
where the constant $C_{d}$ depends on $c_0$ and $\sigma$, and the function $R_d(a,c,t)$ satisfies $\lim_{a,c \uparrow \infty} R_d(a,c,t) = 0$ for any $t$, and $R_d(a,c, \cdot) \in \mathcal{O}(t^{-(\frac{1}{2}+\frac{d}{4})})$ for any $a,c$. Hence, the density of the VE Kac process \eqref{eq:VE-KAC} $L_2$-converges to the density of the Wiener process with rate $\mathcal{O}(t^{-(\frac{1}{2}+\frac{d-1}{4})})$.
\end{theorem}

\begin{proof}
First, by assumption, $(X_0^1,...,X_0^d, K^1(t),..., K^d(t))$ is (mutually) independent. Hence, $\mathbf{X}_t =  (X_0^1 + K^1(t),...,X_0^d + K^d(t))$ has independent components $X_0^i + K^i(t)$ that are one-dimensional Kac processes starting in $X_0^i \sim f_0^i$. By Theorem \ref{insertion-1D}, they admit a probability density $u^i$ solving the 1D telegrapher's equation \eqref{damped-wave} with initial functions $f_0^i \in H^2(\R)$, and hence, $\mathbf{X}_t$ admits the probability density $u(t,x) = \prod_{i=1}^d u^i(t,x^i)$ by the independence of the components. 

Analogously, we conclude that $\mathbf{X}_0 + \sigma \mathbf{W}(t)$ has the density $h(t,x) = \prod_{i=1}^d h^i(t,x^i)$,
where the marginal densities $h^i$ solve the 1D heat equation \eqref{heat-eq} with initial functions $f_0^i \in H^2(\R)$.
Hence, by Theorem \ref{convergence-heat} and Remark \ref{a-c-to-infinity}, the bound \eqref{bound_convergence_heat} holds componentwise with some $C_1=C_1(c_0) > 0$, i.e.
\begin{align}
    \|h^i(t,\cdot) - u^i(t,\cdot)\|_{L_2(\R^1)}
    \leq C_1 \left(\frac{1}{at}+\frac{1}{c(ta)^{\frac{1}{2}}}\right)    \quad \text{for all} \quad  i = 1,...,d .
\end{align}
Now we show the claim by induction on the dimension $d$. The above inequality implies the assertion for $d=1$. 
Let the claim be true for $d \ge 1$. 
Using for $a_i,b_i \in \R$ the identity 
\begin{align}
    \prod_{i=1}^{d+1} a_i - \prod_{i=1}^{d+1} b_i
    &= \big(\prod_{i=1}^{d} a_i  \big) (a_{d+1} - b_{d+1}) + b_{d+1}( \prod_{i=1}^{d} a_i -  \prod_{i=1}^{d} b_i) ,
\end{align}
it follows by Minkowski's inequality that
\begin{align}
    \|h(t, \cdot) - u(t,\cdot)\|_{L_2(\R^{d+1})} 
    &= \Big( \int_{\R^{d+1}} |  \prod_{i=1}^{d+1} h^i(t,x^i) - \prod_{i=1}^{d+1} u^i(t,x^i) |^2 \d x\Big)^{\frac{1}{2}}\\
    &\le \Big( \int_{\R^{d+1}} |  \big(\prod_{i=1}^{d} h^i(t,x^i)\big)(h^{d+1}(t,x^{d+1})-u^{d+1}(t,x^{d+1}))|^2 \d x\Big)^{\frac{1}{2}}\\
    &\quad + \Big( \int_{\R^{d+1}} | u^{d+1}(t,x^{d+1}) \big(\prod_{i=1}^{d} h^i(t,x^i) - \prod_{i=1}^{d} u^i(t,x^i)\big)|^2 \d x\Big)^{\frac{1}{2}}\\
    &= \Big(\prod_{i=1}^{d} \|h^i(t,x^i)\|_{L_2(\R^{1}) } \Big)
    \|h^{d+1}(t,x^{d+1})-u^{d+1}(t,x^{d+1}) \|_{L_2(\R^1)}\\
    &\quad + \| u^{d+1}(t,x^{d+1}) \|_{L_2(\R^1)} \Big( \int_{\R^{d}} |\big(\prod_{i=1}^{d} h^i(t,x^i) - \prod_{i=1}^{d} u^i(t,x^i)\big)|^2 \d x\Big)^{\frac{1}{2}}.
\end{align}
Note that with the one-dimensional heat kernel $K_\sigma(t,\cdot)$, 
Young's inequality for convolutions gives 
$$\|h^i(t,\cdot)\|_{L_2} 
= \|K_\sigma(t,\cdot) \ast f_0^i\|_{L_2} 
\le \|K_\sigma(t,\cdot)\|_{L_2} \|f_0^i\|_{L_1} = \|K_\sigma(t,\cdot)\|_{L_2} 
= \frac{1}{(4 \pi \sigma^2 t)^\frac{1}{4}}.
$$ Hence, the first summand can be estimated by
\begin{align}
    \frac{C_1}{(4 \pi \sigma^2)^\frac{d}{4}} \left(\frac{1}{at^{1+\frac{d}{4}}}+\frac{1}{c a^{\frac{1}{2}} t^{\frac{1}{2} + \frac{d}{4}}}\right).
\end{align}
By the induction hypothesis,  the second summand can be estimated by
\begin{align}
    & \quad \Big( \| u^{d+1}(t,x^{d+1})-h^{d+1}(t,x^{d+1}) \|_{L_2(\R^1)} +\| h^{d+1}(t,x^{d+1}) \|_{L_2(\R^1)} \Big) \\ 
    & \quad \cdot  \Big(C_d \Big(\frac{1}{at^{1+\frac{d-1}{4}}}+\frac{1}{ca^{\frac{1}{2}} t^{\frac{1}{2}+\frac{d-1}{4}}}\Big) + R_d(a,c,t) \Big)\\
    &\le \Big( C_1 \Big(\frac{1}{at}+\frac{1}{c(ta)^{\frac{1}{2}}}\Big) + \frac{1}{(4 \pi \sigma^2 t)^\frac{1}{4}} \Big)\\
    &\quad \cdot  \Big( C_d \Big(\frac{1}{at^{1+\frac{d-1}{4}}}+\frac{1}{ca^{\frac{1}{2}} t^{\frac{1}{2}+\frac{d-1}{4}}}\Big) + R_d(a,c,t) \Big).
\end{align}
Straightforward multiplication finishes the proof.
\end{proof}

The above assumption on the independence of the initial components $(X_0^1,...,X_0^d)$ was made for simplicity. The convergence result for $t\to \infty$ appears to hold true more generally and we forgo the technical details. Furthermore, weak convergence results concerning the paths of a rescaled Kac process towards a Brownian motion can be found in \cite{HLQ24,S1982}.

Lastly, we note that the long-time asymptotics of Theorem \ref{convergence-multi-d} only reasonably apply to the VE Kac process \eqref{eq:VE-KAC}, since the VP Kac process \eqref{eq:VP-KAC} is constructed to reach its terminal distribution $P_{\mathbf{K}(1)}$ in finite time.\footnote{Note however that Theorem \ref{convergence-multi-d} also covers the limit $a,c \uparrow \infty$.} 
 
\begin{figure}[ht!] 
    \centering
    \includegraphics[width = 0.18\textwidth]{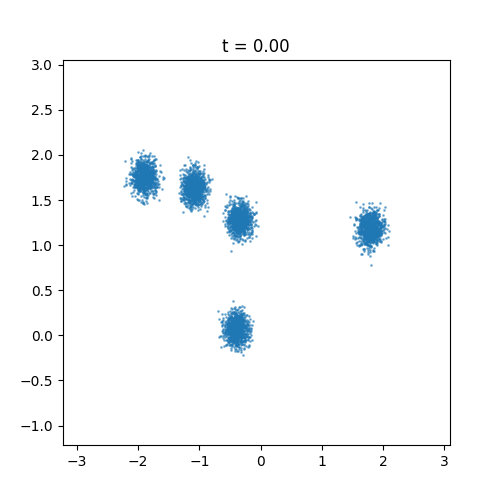}
    \includegraphics[width = 0.18\textwidth]{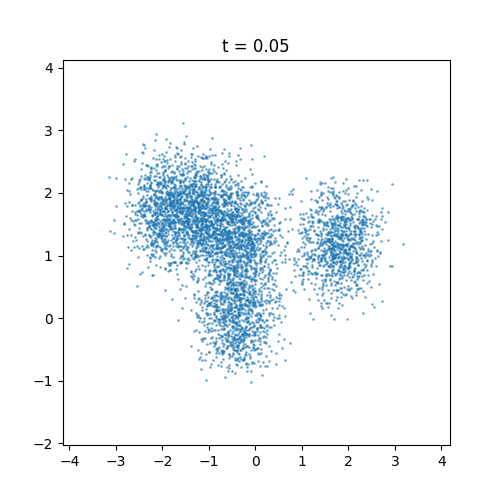}
    \includegraphics[width = 0.18\textwidth]{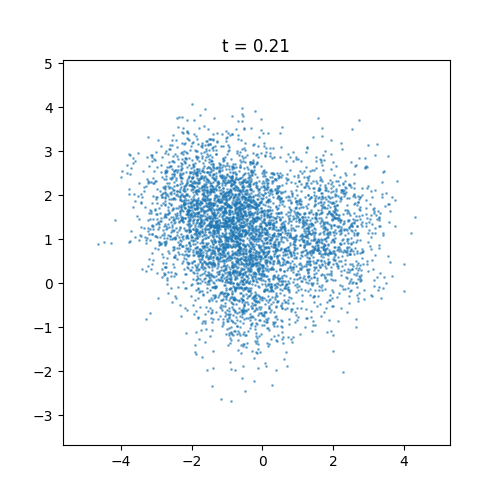}
    \includegraphics[width = 0.18\textwidth]{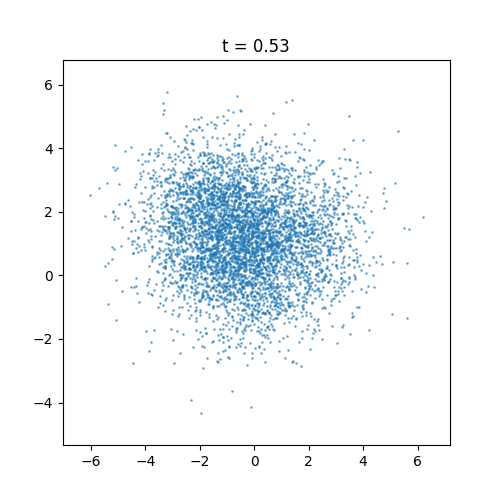}
    \includegraphics[width = 0.18\textwidth]{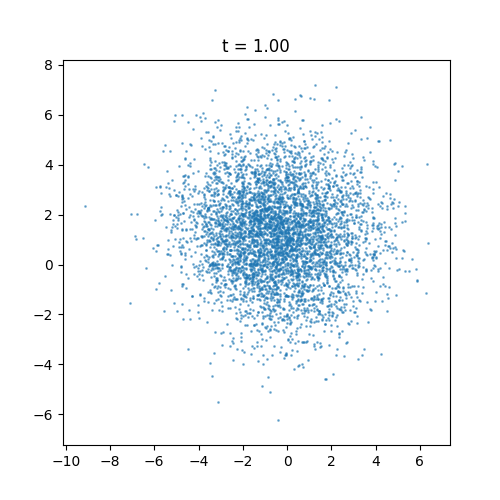}
    
    \caption{The distribution of $\mathbf{X}_t$ starting in a 2D Gaussian mixture, simulated until time $T=3$ with 'large' parameters $(a,c) =(100,10)$. As described in Theorem \ref{convergence-multi-d}, we see a 'diffusion' like process with approximate variance $\sigma^2 t = \frac{c^2}{a}t = t$.}
    \label{flow_high_ac}
\end{figure}

%----------------------------------------------------------------
\section{Velocity Field of the Multi-Dimensional Kac Flow}\label{sec:velo}
%---------------------------------------------------------------
In this section, we want to develop a method how to retrieve the velocity field $v_t$ of the Kac flow in multiple dimensions. By conditioning on single points $x_0 \in \R^d$, the multi-dimensional velocity field $v_t(x|x_0)$ decomposes into univariate velocities $v_t^i(x^i|x_0^i)$ which are analytically and numerically accessible. These conditional velocity fields can then be used to approximate the true velocity field $v_t$ within the framework of conditional flow matching. But note that conceptually, our velocities are conditioned to start at $t=0$ in the target measure -- reversing the direction of usual flow matching.

The following general lemma  states that the velocity field of a multi-dimensional Wasserstein flow can be retrieved by its 1D components if the probability flow admits a multiplicative decomposition (by independence).
Related factorization results appear in the generator matching framework \cite[Proposition 4]{holderrieth2025generatormatchinggenerativemodeling} and in the discrete flow matching framework \cite[Section 7.5]{lipman2024flowmatchingguidecode}. 
These approaches consider probability evolutions arising from Markovian dynamics or from discrete flow formulations.
Since the Kac process \eqref{eq:kac} is a non-Markovian\footnote{Yet, it becomes Markovian by considering the joint process $(K(t), \dot K(t))$.} process \cite[Chapter 12]{EthierKurtz1986} in a continuous setting, and for completeness, we provide a rigorous and general proof here.

\begin{lemma} \label{cont_eq_decomposes}
    Let $\mu_t \in AC^2(I; \mathcal P_2(\R^d))$ such that     
    \begin{align}\label{product-decomposition}
        \mu_t(x) = \prod_{i=1}^d \mu^i_t(x^i) \quad \text{ for all } x = (x^1,..., x^d).
    \end{align}
    Furthermore, assume that the one-dimensional measure flows $\mu^i_t$ satisfy $\mu^i_t \in AC^2(I; \mathcal P_2(\R))$ and a continuity equation with some vector field $v^i_t$, i.e. 
    \begin{align}\label{decomposed-CE}
        \partial_t \mu^i_t(x^i) = - \partial_{x^i} (\mu^i_t(x^i) v^i_t(x^i) ) \quad \text{ for all } x^i \in \R.
    \end{align}
    %Also, as a technical assumption {\rm (T)}, suppose that $v_t(x)$ is bounded in $x$ so that uniform convergence in $C_c(\R)$ implies convergence in $L_1(\mu_t)$ and in $L_1(v_t \mu_t)$.
    Then $\mu_t$ satisfies the continuity equation $\partial_t \mu_t(x) = - \nabla \cdot \left(\mu_t(x) v_t(x) \right)$
    with the velocity field  
    \begin{align}\label{velo-decomposition}
        v_t(x):= (v^1_t(x^1), ... , v^d_t(x^d)) \quad \text{ for all } x = (x^1,..., x^d).
    \end{align}    
\end{lemma}

\begin{proof}
First, we give a formal derivation of the result, and justify the calculations more rigorously afterwards. We compute $\partial_t \mu_t(x)$ formally by use of the product rule and obtain
\[
\partial_t \mu_t(x) = \sum_{i=1}^d \Big( \prod_{j \neq i}^d \mu_t^j(x^j) \Big) \partial_t \mu_t^i(x^i).
\]
Using the continuity equation for $\mu_t^i$, we obtain
\[
\partial_t \mu_t(x) = - \sum_{i=1}^d \Big( \prod_{j \neq i}^d \mu_t^j(x^j) \Big) \partial_{x^i} ( \mu_t^i(x^i) v_t^i(x^i) ).
\]
Pulling the product which does not depend on $x^i$ into the $\partial_{x^i}$-derivative, formally yields the claim
\[
\partial_t \mu_t(x) = - \sum_{i=1}^d \partial_{x^i} ( \mu_t(x) v_t^i(x^i) ) = - \nabla \cdot ( \mu_t(x) v_t(x) ).
\]
More precisely, we can test above equations with product test functions $\phi(t,x) = \phi_0(t) \prod_{i=1}^d \phi_i(x^i)$, where $\phi_i \in C_c^\infty(\R)$, and obtain the result in the weak formulation. Finally, since the linear subspace generated by such products is dense in $C_c^\infty((0,\infty) \times\R^d)$ and by Lebesgue's dominated convergence, this yields the claim. We refer to Section \ref{sect:decomposition-lemma} in the appendix for the full mathematical details.
\end{proof}

The following main theorem will be of fundamental importance for the generative modeling via the Kac flow. In particular, we obtain an analytical formula for the conditional velocity field $v_t(x|x_0)$ of a Kac process conditioned to start in a point $x_0 \in \R^d$.

\begin{theorem}\label{thm14}
For a fixed $x_0 \in \R^d$, let $\mathbf{X}_t \coloneqq {x_0} + \mathbf{K}(t)$ be the VE Kac process starting in $x_0$.
 Then, $\mu_t(\cdot|x_0) \coloneqq P_{\mathbf{X}_t}$  satisfies a continuity equation \eqref{eq:ce} with the velocity field
    \begin{align}\label{velo-decomposition-kac}
        v_t(x|x_0):= (v^1_t(x^1|x_0^1), ... , v^d_t(x^d|x_0^d)) \quad \text{ for all } x = (x^1,..., x^d),
    \end{align} 
    where the univariate components $v^i_t(x^i|x_0^i)$ are  given by 
    \begin{align} \label{velo_1d_shifted}
       v_t^i(x^i|x_0^i) &\coloneqq 
       \left\{
       \begin{array}{cl}
           \displaystyle (x^i-x_0^i) \Big(t + \frac{r_t(x^i-x_0^i)}{c} \cdot \frac{I_{0}(\beta r_t(x^i-x_0^i))}{I_{0}'(\beta r_t(x^i-x_0^i))}\Big)^{-1} 
           & \text{if} \quad x^i \in (x_0^i-ct, x_0^i+ct),\\
           \; \; \, c &\text{if} \quad x^i= x_0^i + ct, \\
           -c & \text{if} \quad x^i=x_0^i -ct,\\
             \; \; \, \text{arbitrary} & \text{otherwise}.
       \end{array}
       \right.
   \end{align}
   Furthermore, it holds
   \begin{equation} \label{eq:conditional-velocity-def_revised_sec}
    v^{i}_t(x^{i} | x_0^{i}) = \E[\dot{X}^{i}(t) \mid X_t^{i} = x^{i}, X_0^{i} = x_0^{i}].
\end{equation}
\end{theorem}

\begin{proof}
    By Proposition \ref{prop:lipschitz}, $\mu_t(\cdot|x_0)$ admits a velocity field $v_t(\cdot|x_0)$ satisfying the continuity equation \eqref{eq:ce}. By construction, the components $(X_t^1,...,X_t^d)$ of $\mathbf{X}_t$ are mutually independent. Hence, $\mu_t(x|x_0) = \prod_{i=1}^d \mu_t^i(x^i|x_0^i)
    $, 
    %satisfies the product decomposition \eqref{product-decomposition}, 
    where the one-dimensional probability flows $\mu_t^i(\cdot|x_0^i)$ fulfill a continuity equation with 1D velocity fields $v_t^i(\cdot|x_0^i)$.
    By Lemma \ref{cont_eq_decomposes}, it holds the decomposition \eqref{velo-decomposition-kac}, and Theorem \ref{velo_dirac} immediately gives the formula \eqref{velo_1d_shifted}.
    Finally, \cite[Theorem 3.3]{liu2023flow}, resp.\ \cite[Theorem 6.3]{WS2025} provides \eqref{eq:conditional-velocity-def_revised_sec}.
    \end{proof}
    
We will derive an analogous result concerning the VP/mean-reverting Kac process \eqref{eq:VP-KAC} later in Subsection \ref{subsec:image_gen}, where we also include the use of time schedules.
    
In stochastic analysis, the relation \eqref{eq:conditional-velocity-def_revised_sec} has been known in connection with the  \emph{mimicking of stochastic processes}, see e.g.\ \cite{BS2013,G1986}. Here, the mimicking process is given by the flow ODE \eqref{eq:flow_ode}.   

Having identified the conditional vector field $v_t(\cdot|x_0)$ for the conditional distribution $\mu_t(\cdot|x_0)$ of the Kac process, we get the following result
by employing \cite{lipman2023flow}, see also
\cite[Theorem 5.1, Remark 5.3]{WS2025}, with slight adjustments to our setting.

\begin{corollary}\label{final_velo}
Let $\mathbf{X}_t$ be the multi-dimensional VE Kac process starting in $\mathbf{X}_0 \sim \mu_0 \in \mathcal{P}_2(\mathbb{R}^d)$.  Assume that $\mathbf{X}_0$ is independent of $\mathbf{K}(t)$, and assume that $\mathbf{X}_0$ has a density $f_0$.
Then, the marginal distribution of $\mathbf{X}_t$ 
admits the density $u_t$ given by
\begin{equation}\label{total_prob}
    u_t(x) = \int u_t(x | x_0) f_0(x_0)  \d x_0,
\end{equation}
where $u_t(\cdot | x_0)$ is the conditional distribution of $\mathbf{X}_t$ conditioned to $\mathbf{X}_0 = x_0$, and by symmetry\footnote{This is \emph{not} an application of Bayes' rule.},
$u_t(x| \cdot)$ is exactly the distribution $u_t(\cdot| x)$ for any $x \in \R^d$.\\
Denote the conditional velocity field corresponding to $u_t(\cdot|x_0)$ by $v_t(\cdot | x_0)$. 
By independence of $\mathbf{X}_0$ and $\mathbf{K}(t)$, it is given via \eqref{velo_1d_shifted}, and finally, 
the vector field
\begin{equation}\label{velo-FM}
    v_t(x) := \frac{1}{u_t(x)} \int u_t(x | x_0) \, v_t(x | x_0) f_0(x_0) \d x_0
\end{equation}
satisfies the continuity equation
$$
    \partial_t u_t(x) = - \nabla \cdot \big( u_t(x) v_t(x) \big)
$$
in the distributional sense.
\end{corollary}

%{\color{blue} For a more general formulation via Markov kernels, see also \cite[Remark 5.3]{WS2025}.}

Now we have all the ingredients to define a feasible loss function for learning the above velocity field $v_t$ by a neural network $v^\theta_t$.
Clearly, it is desirable to use the loss
\begin{equation}
        \mathcal{L}(\theta) \coloneqq
        \mathbb{E}_{(t,x) \sim \mathcal{U}(0,T) \times_t 
       u_t} \left[ \left\| v^\theta(t,x) - v(t,x) \right\|^2 \right],
\end{equation}
where $\mathcal{U}(0,T) \times_t u_t$ is understood in the sense of \emph{disintegration}, see \cite[Ch.\ 2.3]{WS2025}; but here we have no direct access to $v$. However, following the same lines as in \cite{lipman2023flow}, we see that
$$
\mathcal{L}(\theta) =      \mathcal{L}_{\mathrm{CFM}}(\theta) + \text{const}
$$
with a constant not depending on $\theta$, and the Conditional Flow Matching (CFM) loss given by
\begin{align} \label{fm_loss}
     \mathcal{L}_{\mathrm{CFM}}(\theta) = \mathbb{E}_{t \sim \mathcal{U}(0,T), \, x_0 \sim \mu_0, \, x \sim u_t(\cdot| x_0)} \left[ \left\| v^\theta(t,x) - v(t,x| x_0) \right\|^2 \right].    
\end{align}     
Here, we have access to both $v(t,x| x_0)$ and $u_t(\cdot| x_0)$ via \eqref{velo-decomposition-kac} and \eqref{telegraph-formula-1d-dirac}, respectively, which we will exploit in the next section.
\begin{remark}
    Notice that the decomposition of velocity fields from Lemma \ref{cont_eq_decomposes} holds generally and independently from our specific Kac model. Hence, our approach of decomposing the conditional velocity in conjunction with conditional flow matching is \emph{generally applicable} to any model with a \emph{componentwise} forward process (i.e.\ independent components); a setting which is already routinely employed in diffusion models.
\end{remark}

%----------------------------------------------------------------
\section{Numerical Experiments} \label{sec:numerics}
%----------------------------------------------------------------
\begin{figure}[ht!] 
    \centering
    \includegraphics[width = 0.19\textwidth]{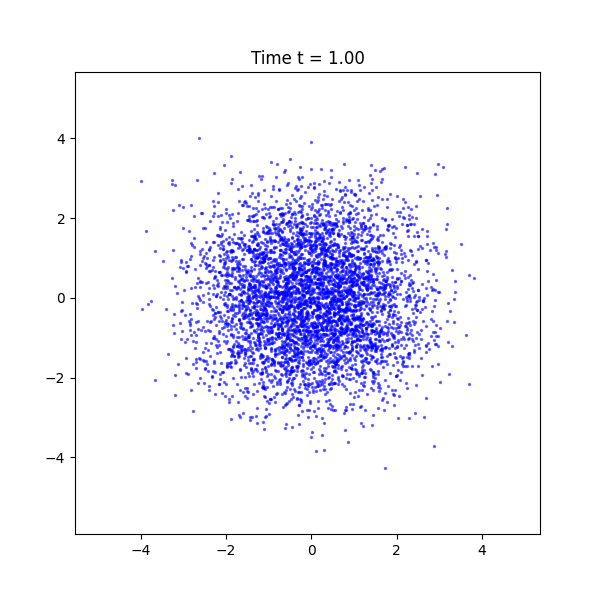}
    \includegraphics[width = 0.19\textwidth]{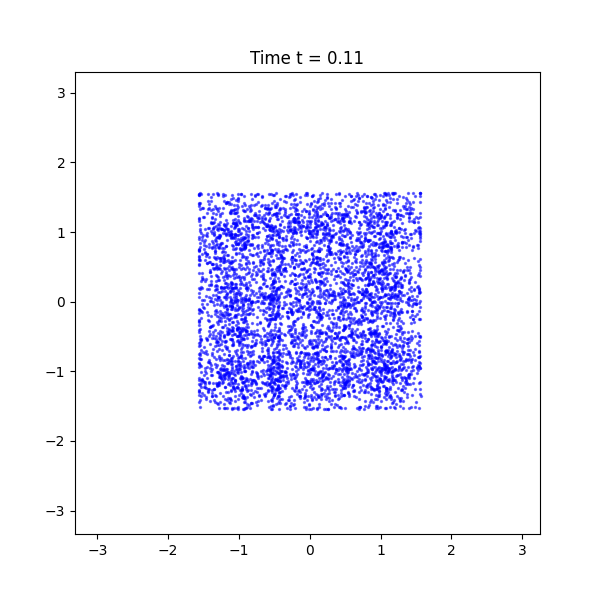}
    \includegraphics[width = 0.19\textwidth]{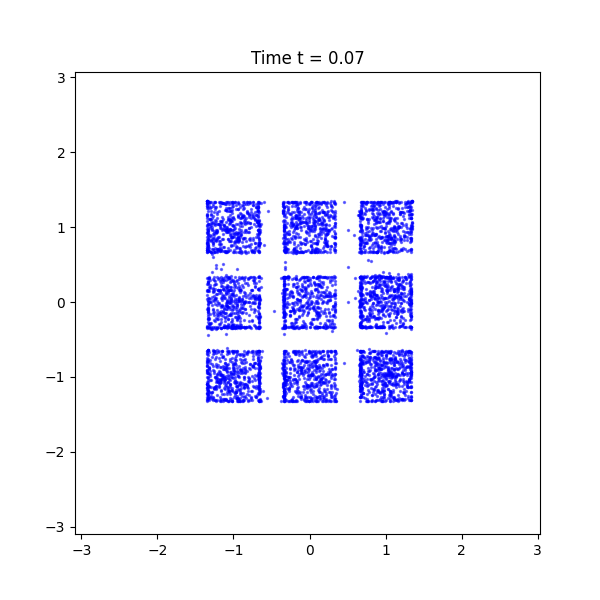}
    \includegraphics[width = 0.19\textwidth]{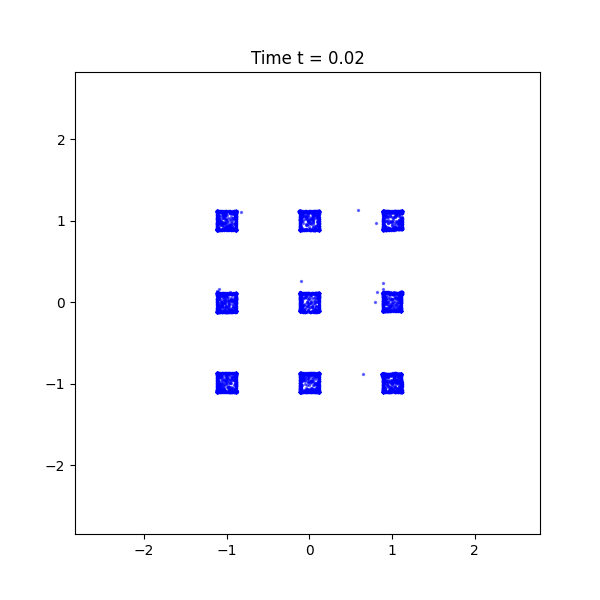}
    \includegraphics[width = 0.19\textwidth]{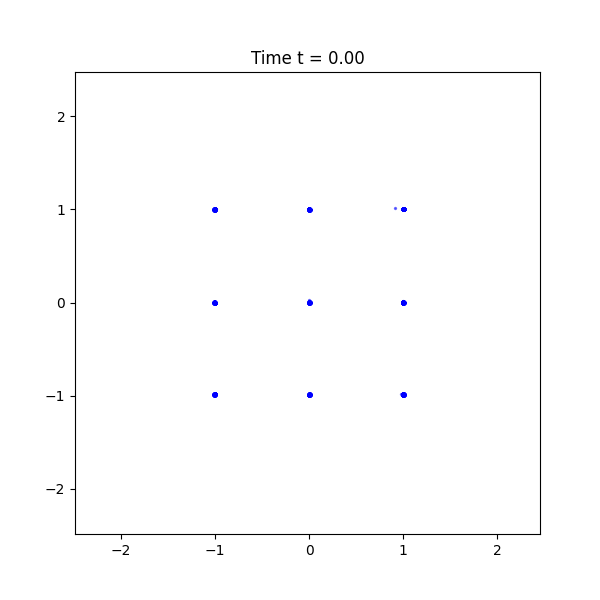} 
  
    \caption{Backward evolution of the \emph{learned} Kac flow for $(a,c) =(25,5)$, see also Figure \ref{fig:2d_experiment}.}
    \label{backward_evolution_2d}
    
\end{figure}

In this section, we use the flow matching framework to train a neural network to approximate the velocity field from Corollary \ref{final_velo} by minimizing the CFM loss given in \eqref{fm_loss}. 
To this end, we compute the conditional velocity $v_t(x|x_0)$ via its analytic formula \eqref {velo_1d_shifted}. In order to sample from $\mathbf{X}_t \sim u_t(\cdot | x_0)$, where $x_0$ is fixed, we only need to sample i.i.d.\ components $K_t^i$ of the Kac process starting in $0$. This can be done efficiently, since we have access to the analytic form \eqref{telegraph-formula-1d-dirac} of the one-dimensional Kac distribution starting in $0$ (use e.g.\ inverse transform sampling).  
Alternatively, if we wish to sample fixed realizations of a Kac process for varying times $t$, we can efficiently simulate the random walk itself as 
explained in Appendix \ref{app:simulating_kac}.

During inference, we sample the latent variable from the Kac process \( K_T \) starting in zero, at a final time \( T \), and then integrate the flow ODE \eqref{eq:flow_ode} backwards in time using an ODE solver with the learned vector field as the drift. Using the standard Kac flow introduces an approximation error, namely the distance between $X_T$ and $K_T$. This can be avoided by applying the mean reverting Kac flow \eqref{eq:mean-rev-kac} which we will reintroduce below with time schedules.

\subsection{Diffusion vs. Kac Process: A 2D Experiment with Dirac-like Modes}\label{sec:2d_exp}

In \eqref{diffusion-local-explosion}, we verified that  velocity fields of diffusion flows admit a local explosion at small times $t \approx 0$, when starting in Dirac points. In contrast, the multi-dimensional Kac flow enjoys a globally bounded velocity field as proven in Proposition \ref{prop:lipschitz}. We want to demonstrate this difference in the qualitative behavior by the following toy problem in 2D:
as ground truth we choose an equally weighted Gaussian mixture with $9$ modes on a $3 \times 3$ grid, where each mode is given by an isotropic Gaussian with \emph{small} standard deviation $\sigma = 0.0001$, representing some Dirac points, see Figure \ref{fig:2d_experiment}.\\
We apply the conditional flow matching framework to both the Kac and diffusion based model. Here we use the VE Kac process \eqref{eq:VE-KAC} starting in the data with parameters $\frac{c^2}{a}=1$, and the standard diffusion with parameter $\sigma=1$ and the velocity field \eqref{velo2}.

We trained each model for $500$k iterations, and used the negative log-likelihood as a validation loss. In Figure \ref{fig:2d_experiment} we plot the resulting generated samples and report the associated training iterations. The complete experimental setup can be found in the Appendix \ref{implementation}. For the backward evolution of the learned $(25,5)$-Kac flow, see Figure \ref{backward_evolution_2d}.

\begin{figure}[ht!]
  \centering
  %--- first image ---
  \begin{minipage}[b]{0.24\textwidth}
    \centering
    \includegraphics[width=\linewidth]{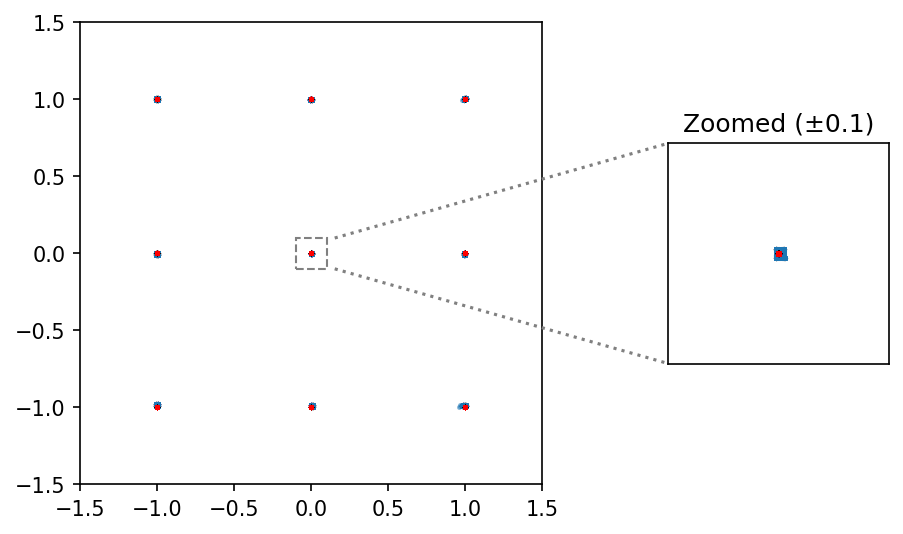}
    \par\smallskip
    {\small \( (a,c)=(25,5)\)}\\
    {\small 80k iterations}
  \end{minipage}\hfill
  %--- second image ---
  \begin{minipage}[b]{0.24\textwidth}
    \centering
    \includegraphics[width=\linewidth]{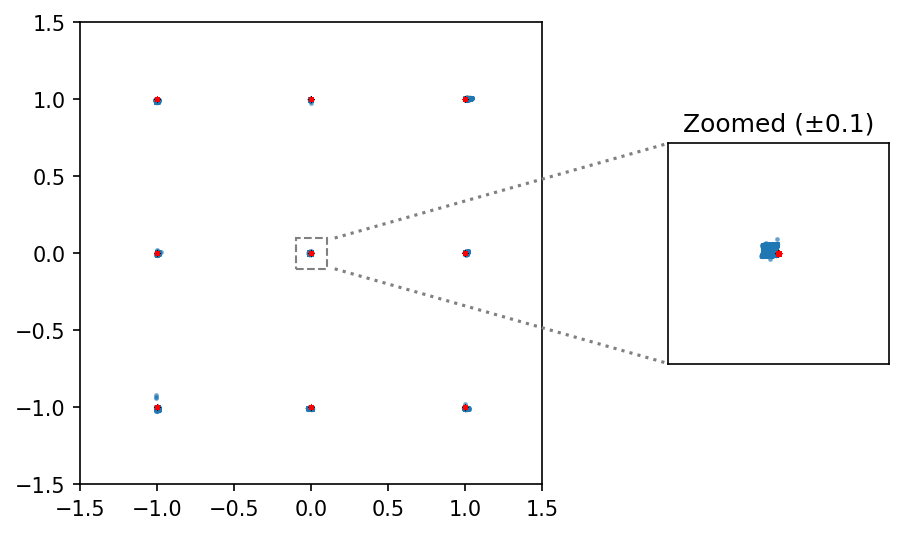}
    \par\smallskip
    {\small \((a,c)=(64,8)\)}\\
    {\small 280k iterations}
  \end{minipage}\hfill
  %--- third image ---
  \begin{minipage}[b]{0.24\textwidth}
    \centering
    \includegraphics[width=\linewidth]{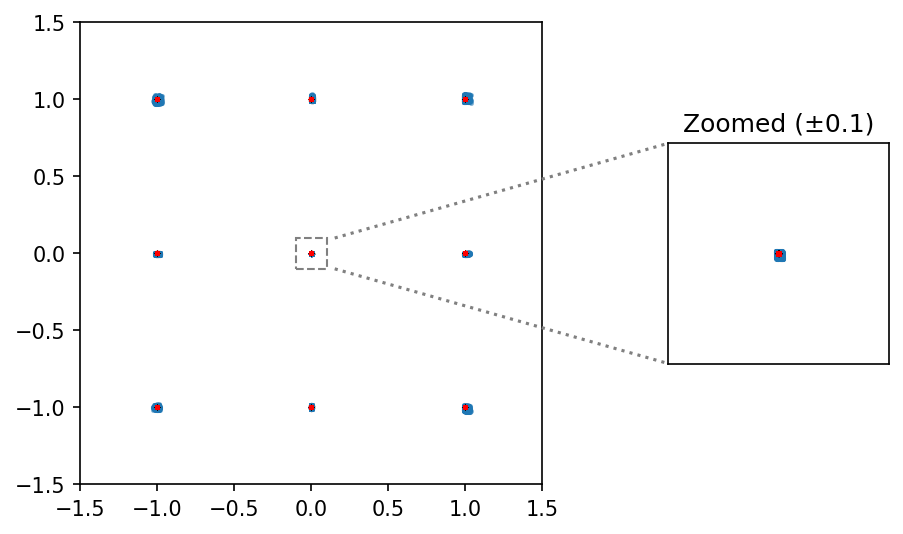}
    \par\smallskip
    {\small \( (a,c)=(1600,40)\)}\\
    {\small 460k iterations}
  \end{minipage}\hfill
  %--- fourth image (diffusion) ---
  \begin{minipage}[b]{0.24\textwidth}
    \centering
    \includegraphics[width=\linewidth]{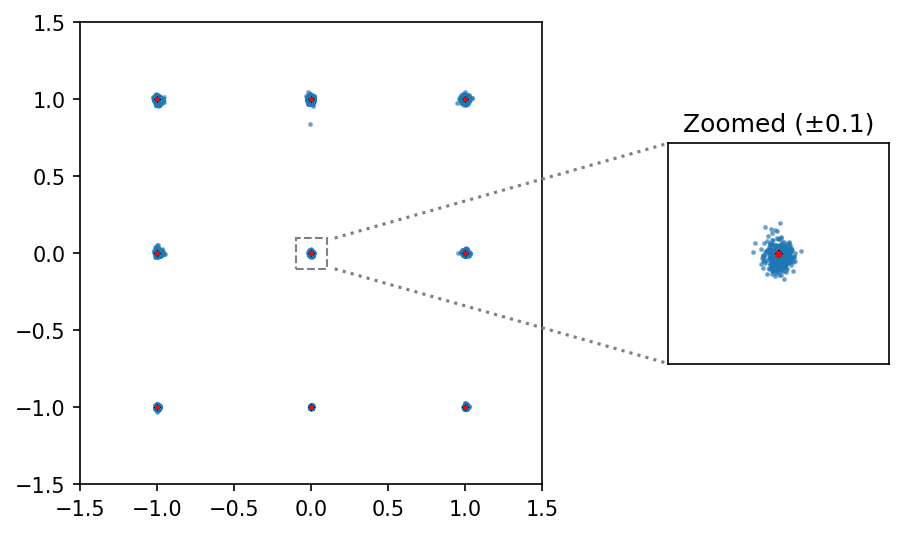}
    \par\smallskip
    {\small diffusion}\\
    {\small 420k iterations}
  \end{minipage}
  \caption{Generated samples (blue) vs.\ ground truth (red) at the indicated iteration for each model. The Kac model can precisely recover the small modes, while the diffusion model creates  ``blobs''.}
  \label{fig:2d_experiment}
\end{figure}

In contrast to our Kac generative model, the diffusion based model is not able to precisely approximate the Dirac-like modes, even with long training times and vanishing truncation times\footnote{By the explosion \eqref{diffusion-local-explosion} time truncations were necessary, since otherwise the associated adaptive step size ODE solver diverged with intended step size going to $0$.} $\epsilon \approx 0$.
Clearly, large values of parameters $(a,c)$ lead to diffusion-like results
 as theoretically underlined in Theorem \ref{convergence-multi-d}.

\subsection{Image Generation}\label{subsec:image_gen}
In general, the \emph{'standard'} VE Kac process $X_t = X_0 + K_t$ 
and \(K_t\) themselves are not equal at the final time \(T\).  It is common practice in generative modeling to augment the process so that the data no longer contributes to the final distribution (compare with VP diffusion processes, e.g.\ the Ornstein–Uhlenbeck process).  Additionally, we introduce a time‐scheduling function for the noising process. 
Let $X_0$ and $K_t$ be independent, 
and for simplicity, we fix \(T=1\) for the rest of this section.

\paragraph{The Mean‐Reverting Kac Process.}

Generalizing the VP Kac process \eqref{eq:VP-KAC}, 
we define the \emph{mean‐reverting} Kac process by
\begin{equation}\label{eq:mean-rev-kac}
  M_t \;\coloneqq\; f(t)\,X_0 \;+\; K_{g(t)},
  \quad t\in[0,1],
\end{equation}
with suitably smooth \emph{time-scheduling} functions \(f,g\) satisfying
\[
  f(0)=1,\quad f(1)=0 \quad \text{and} \quad g(0)=0,\quad g(1) = 1.
\]
In particular, the latent distribution is exactly given by the law of \(M_1=K_{1}\).  
Note that $K_{1}$ is \emph{not} Gaussian, but supported on $[-c,c]^d$ with two atomic components per coordinate (Diracs at $\pm c$) and an absolutely continuous part on $(-c,c)^d$, see \eqref{telegraph-formula-1d-dirac} and Figure \ref{role-of-a}.
Differentiating \eqref{eq:mean-rev-kac} gives
\[
  \dot M_t = \dot f(t)\,X_0 \;+\; \dot g(t)\,\dot K_{g(t)},
\]
and hence the conditional velocity field is given by
\begin{align}\label{eq:velo-OU}
  v_M(t,x\mid x_0)
  &= \mathbb{E}\bigl[\dot M_t \mid M_t=x,\;X_0=x_0\bigr]\nonumber\\
  &= \mathbb{E}\bigl[\dot f(t)\,x_0 + \dot g(t)\,\dot K_{g(t)}
    \,\bigm|\,K_{g(t)}=x-f(t)x_0,\;X_0=x_0\bigr]\nonumber\\
  &= \dot f(t)\,x_0 \;+\; \dot g(t)\,v_K\bigl(g(t),\,x - f(t)x_0\mid0 \bigr),
\end{align}
where \(v_K(t,x\mid  0)\) is the conditional Kac velocity field from \eqref{velo-decomposition-kac}.  
Note that in the last equality, the independence of $X_0$ and $K_t$ is used to drop the condition $X_0 = x_0$.
Now, by the same arguments as in Corollary \ref{final_velo}, one recovers the corresponding vector field for the mean‐reverting Kac process.

Similar to the VE Kac process starting in $X_0$, the mean-reverting Kac process \eqref{eq:mean-rev-kac} admits a Lipschitz property in $\mathcal{P}_2$ and hence a bounded velocity norm $\|v_t\|_{L_2(\mu_t)}$, for details we refer to Proposition \ref{prop:lipschitz_2} in the appendix.

\paragraph{Connection to Flow Matching.}

Now let \(R_t \coloneqq  f(t) X_0 + B_{g(t)}\) with a standard Brownian motion \(B_t\). If we choose
\[
  f(t)=1-t,\qquad g(t)=t^2,
\]
then it holds
\begin{equation}\label{eq:rt-def}
  R_t = (1-t)\,X_0 + B_{t^2}
      \overset{d}{=} (1-t)\,X_0 + t\,B_1,\quad t\in[0,1].
\end{equation}
Notice that \eqref{eq:rt-def} matches exactly the forward process used in recent flow matching works, see \cite{lipman2023flow, WS2025}.  We experimented with various schedules and found that the above flow matching pair
\[
  (f,g) = (1-t,\;t^2)
\]
also yielded good generation results for our Kac model. In the following, we compare our approach via \( (M_t)_t \) with that using the process \( (R_t)_t \), in particular with flow matching based on the independent coupling.

\paragraph{CIFAR-10.} We test our approach on generating CIFAR-10 images. We train the velocity field using conditional flow matching with the vector field \eqref{eq:velo-OU} and compare a variety of choices for $(a,c)$ and the schedules $g(t) =t$ and $g(t)=t^2$. We apply a UNET architecture to parameterize the velocity field and use an ODE solver to simulate the corresponding flow ODE. We aligned our hyperparameters with those employed in \cite{tong2024improving} to train our models.  Our code is available online\footnotemark[1]\footnotetext[1]{https://github.com/JChemseddine/telegraphers}. For further details we refer to the implementation details in Appendix \ref{implementation}.

In Table \ref{fid_results} we report the FID results of our Kac methods via $M_t$ as well as the diffusion baselines via $R_t$. We implemented the baseline models ourselves and denote them by ``Diff'' for $g(t)=t$ and ``FM'' for $g(t)=t^2$. In Figures \ref{fig:cifar-1}, \ref{fig:cifar-2} we plot a selection of the generated images. Recall that the parameter $a$ corresponds to the damping coefficient and $c$ to the propagation speed. The schedule $g(t) =t$ generally led to worse performance, especially for the baseline and our methods with diffusion-like parameters. However, certain combinations of high damping and low velocity pairs $(a,c)$ such as $(25,2)$ were able to perform well. In contrast, the schedule $g(t)=t^2$ improved the baseline substantially, and our methods across the board. With this experiment we aim to demonstrate the scalability of our method, an extensive study of the effect of the parameters $(a,c)$ is beyond the scope of the paper and left to future work. For further visual examples and a nearest neighbour analysis, see Appendix \ref{further_examples}.

\paragraph{Heuristics regarding the role of the damping $a > 0$.}
The damping given by the parameter $a$ seems to have a strong impact on the generation results of our method. Without the damping, i.e.\ $a=0$, the flow ODE \eqref{eq:flow_ode} becomes numerically ill-posed: the latent distribution in 1D would become pure shockwaves/wave fronts at $\pm cT$, with no absolutely continuous part in between. Hence, sampling from these two Dirac points could not produce e.g.\ absolutely continuous target distributions. Mathematically, this is due to the exploding Lipschitz constants of the flow map $\varphi(t, \cdot)$ and velocity field $v(t, \cdot)$ at times $t$ close to the terminal/latent time $T$.

By the addition of the damping $a > 0$, the wave fronts vanish with the factor $e^{-aT}$, getting absorbed into the absolutely continuous part $\tilde u$ as described in \eqref{telegraph-formula-1d-dirac}. The flow map/velocity field's spatial regularity improves with increasing dampings $a>0$ around terminal time $T$, see Figure \ref{role-of-a} in the appendix, making the flow ODE \eqref{eq:flow_ode} numerically well-behaved and explaining the better results for high $a$'s.

\begin{table}[ht!]
\centering
\small

\setlength{\tabcolsep}{2pt}%
\renewcommand{\arraystretch}{1.15}%

\begin{minipage}[t]{0.48\linewidth}
\centering
\vspace{0pt}%
\begin{tabular}{@{}p{2.6cm}p{0.8cm}|p{2.6cm}p{0.8cm}@{}}
\hline
\multicolumn{4}{c}{\textbf{Schedule: }$g(t)=t^2$} \\
\hline
\textbf{Method} & \textbf{FID} & \textbf{Method} & \textbf{FID} \\
\hline
$a=900,c=10$ & \textbf{7.26} & $a=100,c=10$ & 10.01 \\
$a=900,c=20$ & 7.46 & $a=25,c=1$   & 8.60\\
$a=900,c=30$ & 8.05 & $a=25,c=2$   & 8.65 \\
$a=100,c=1$  & 11.41& $a=25,c=3$   & 9.15 \\
$a=100,c=3$  & 7.77 & $a=25,c=4$   & 9.56 \\
$a=100,c=5$  & 7.73 & $a=25,c=5$   & 10.70\\
$a=100,c=7$  & 8.58 & FM (our impl.)    & 7.59 \\

\hline
\end{tabular}
\end{minipage}
\hfill
\begin{minipage}[t]{0.48\linewidth}
\centering
\vspace{0pt}%
\begin{tabular}{@{}p{2.6cm}p{0.8cm}|p{2.6cm}p{0.8cm}@{}}
\hline
\multicolumn{4}{c}{\textbf{Schedule: }$g(t)=t$} \\
\hline
\textbf{Method} & \textbf{FID} & \textbf{Method} & \textbf{FID} \\
\hline
$a=900,c=10$ & 17.63 &  $a=100,c=10$ & 39.93\\
$a=900,c=20$ & 42.46 &  $a=25,c=1$   & 8.31\\
$a=900,c=30$ & 64.60 & $a=25,c=2$   & \textbf{6.42} \\
$a=100,c=1$  & 11.64 & $a=25,c=3$   & 7.95 \\
$a=100,c=3$  & 8.56 & $a=25,c=4$   & 10.59 \\
$a=100,c=5$  & 14.66 & $a=25,c=5$   & 16.81 \\
$a=100,c=7$  & 24.44 & Diff (our impl.) & 82.27 \\
 
\hline
\end{tabular}
\end{minipage}
\caption{FID (100-step Euler) for \((a,c)\)-Kac under two time schedules.}
\label{fid_results}
\end{table}

\begin{figure}[ht!] 
  \centering
  %--- first experiment ---
  \begin{minipage}[b]{0.3\textwidth}
    \centering
    \includegraphics[width=\linewidth]{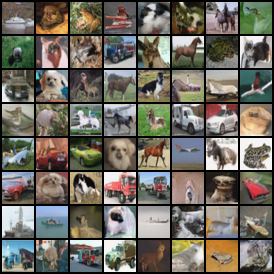}
    \par\smallskip
   {\small \( (a,c)=(900,10)\)}
  \end{minipage}
 \hfill
  %--- second experiment ---
  \begin{minipage}[b]{0.3\textwidth}
    \centering
    \includegraphics[width=\linewidth]{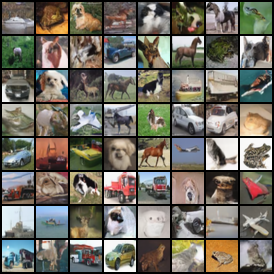}
    \par\smallskip
    {\small \( (a,c)=(900,20)\)}
  \end{minipage}
  \hfill
   %--- third experiment ---
  \begin{minipage}[b]{0.3\textwidth}
    \centering
    \includegraphics[width=\linewidth]{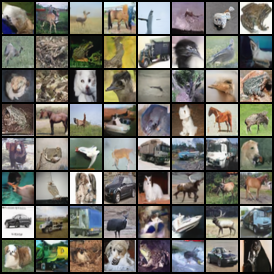}
    \par\smallskip
    {\small FM}
  \end{minipage}

  \caption{Generation results of the best performing mean-reverting Kac models and the diffusion baseline (flow matching) for $g(t)=t^2$. }
  \label{fig:cifar-1}
\end{figure}

\begin{figure}[ht!] 
  \centering
  %--- first experiment ---
  \begin{minipage}[b]{0.3\textwidth}
    \centering
    \includegraphics[width=\linewidth]{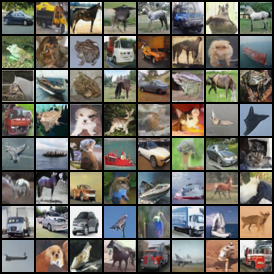}
    \par\smallskip
   {\small \( (a,c)=(25,2)\)}
  \end{minipage}
 \hfill
  %--- second experiment ---
  \begin{minipage}[b]{0.3\textwidth}
    \centering
    \includegraphics[width=\linewidth]{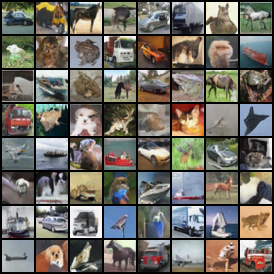}
    \par\smallskip
    {\small \( (a,c)=(25,3)\)}
  \end{minipage}
  \hfill
   %--- third experiment ---
  \begin{minipage}[b]{0.3\textwidth}
    \centering
    \includegraphics[width=\linewidth]{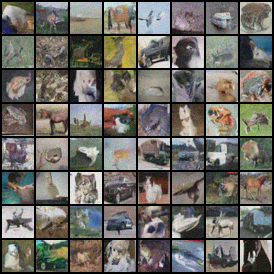}
    \par\smallskip
    {\small Diff}
  \end{minipage}

  \caption{Generation results of the best performing mean-reverting Kac models and the diffusion baseline for $g(t)=t$. }
  \label{fig:cifar-2}
\end{figure}

%----------------------------------------------------------------
\section{Conclusions} \label{sec:conclusion}
%----------------------------------------------------------------
This work introduces a novel generative model based on the damped wave equation (telegrapher's equation) and its stochastic counterpart, the Kac process. In contrast to diffusion models, the Kac process features a \emph{finite-speed} propagation and therefore \emph{bounded velocity fields}. We theoretically derive this regularity result, as well as an explicit formula for the conditional Kac velocity fields (Theorem \ref{velo_dirac}), which we employ in a flow matching framework for image generation. 

Future work can include more extensive studies of %the theoretical sample complexity, and of 
the effect of the model's damping and speed parameters $(a,c)$. The model's flexibility in adjusting these parameters (per component) may prove to be worthwhile within applications. The present work \cite{distillkac2025} demonstrates the applicability of our Kac framework to model distillation.
Further natural extensions of our model may include applications in conditional generation, inverse problems and sampling of Boltzmann distributions.

\paragraph{Acknowledgement.}
J. Chemseddine acknowledges funding  of the DFG project STE 571/17-2 within the SPP 2298 ``Theoretical Foundations of Deep Learning''.
The work of R. Duong, P. Friz and G. Steidl is
funded by the Deutsche Forschungsgemeinschaft (DFG, German Research Foundation) under Germany's Excellence Strategy -- The Berlin Mathematics Research Center MATH+ (EXC-2046/1, EXC-2046/2, project ID: 390685689). The work of P. Friz is also supported by DFG CRC/TRR 388, project A02. Our thanks go to Gregor Kornhardt, Matthias Liero, and Nicolaj Rux for fruitful discussions. We are especially grateful to Weiqiao Han for correcting an error in our implementation,
and to Flinn Raffael for correcting a mistake in the proof of Lemma \ref{cont_eq_decomposes}.

\bibliographystyle{abbrv}
\bibliography{bibliography}

@article{distillkac2025,
title={{DistillKac}: FEW-STEP IMAGE GENERATION VIA DAMPED WAVE EQUATIONS},
author={Weiqiao Han and Chenlin Meng and Christopher D. Manning and
Stefano Ermon},
journal={ICLR},
year={2026}
}

@book{EthierKurtz1986,
  author       = {Ethier, Stewart N. and Kurtz, Thomas G.},
  title        = {Markov Processes: Characterization and Convergence},
  publisher    = {Wiley},
  year         = {1986},
  chapter      = {12.1},
  pages        = {469--470},
}

@article{lipman2024flowmatchingguidecode,
      title={{Flow Matching Guide and Code}}, 
      author={Yaron Lipman and Marton Havasi and Peter Holderrieth and Neta Shaul and Matt Le and Brian Karrer and Ricky T. Q. Chen and David Lopez-Paz and Heli Ben-Hamu and Itai Gat},
      year={2024},
      journal={arXiv:2412.06264},      
}

@article{AV2022,
title = 	 {Building normalizing flows with stochastic interpolants},
author = {M. S. Albergo and E. Vanden-Eijnden},
journal = {ICLR},
year = {2022}
}

@article{ABV2023,
title = {Stochastic interpolants: A unifying framework for flows and diffusions},
author=
{Michael S Albergo and Nicholas M Boffi and Eric Vanden-Eijnden},
journal = {arXiv:2303.08797},
year = {2023}
}

@article{holderrieth2025generatormatchinggenerativemodeling,
      title={Generator Matching: Generative modeling with arbitrary {M}arkov processes}, 
      author={Peter Holderrieth and Marton Havasi and Jason Yim and Neta Shaul and Itai Gat and Tommi Jaakkola and Brian Karrer and Ricky T. Q. Chen and Yaron Lipman},
      year={2025},
      journal = {ICLR},      
}

@article{chen2025scaleadaptivegenerativeflowsmultiscale,
      title={Scale-Adaptive Generative Flows for Multiscale Scientific Data}, 
      author={Yifan Chen and Eric Vanden-Eijnden},
      year={2025},
      journal={arXiv:2509.02971}, 
}

@book{Ambrosio2021LecturesOO,
  title="{Lectures on Optimal Transport}",
  author={Luigi Ambrosio and Elia Bru{\'e} and Daniele Semola},
  series={UNITEXT},
  publisher={Springer Nature},
  year={2021},
}

@InProceedings{pmlr-v37-sohl-dickstein15,
  title = 	 {Deep Unsupervised Learning using Nonequilibrium Thermodynamics},
  author = 	 {Sohl-Dickstein, Jascha and Weiss, Eric and Maheswaranathan, Niru and Ganguli, Surya},
  booktitle = 	 {Proceedings of the 32nd International Conference on Machine Learning},
  pages = 	 {2256--2265},
  year = 	 {2015},
  editor = 	 {Bach, Francis and Blei, David},
  volume = 	 {37},
  series = 	 {Proceedings of Machine Learning Research},
  address = 	 {Lille, France},
  month = 	 {07--09 Jul},
  publisher =    {PMLR}, 
}

@article{OH2000,
author = {Othmer, Hans G. and Hillen, Thomas},
title = {The Diffusion Limit of Transport Equations Derived from Velocity-Jump Processes},
journal = {SIAM Journal on Applied Mathematics},
volume = {61},
number = {3},
pages = {751-775},
year = {2000},
doi = {10.1137/S0036139999358167}
}

@misc{torchdiffeq,
	author={Chen, Ricky T. Q.},
	title={torchdiffeq},
	year={2018},
	url={https://github.com/rtqichen/torchdiffeq},
}

@article{MBHS2025,
  title={{PnP-Flow}: Plug-and-Play image restoration with flow matching},
  author={S. Martin and A. Gagneaux and P. Hagemann and G. Steidl},
  journal={ICLR},
  year={2025}
}

@inproceedings{zhu2023denoising,
  title={Denoising diffusion models for plug-and-play image restoration},
  author={Zhu, Yuanzhi and Zhang, Kai and Liang, Jingyun and Cao, Jiezhang and Wen, Bihan and Timofte, Radu and Van Gool, Luc},
  booktitle={IEEE/CVF Conference on Computer Vision and Pattern Recognition (CVPR) Workshop},
  pages={1219--1229},
  year={2023}
}

@article{liu2022rectified,
  title={Rectified flow: A marginal preserving approach to optimal transport},
  author={Liu, Qiang},
  journal={arXiv:2209.14577},
  year={2022}
}

@article{SE2019,
author={Song, Y. and Ermon, St.},
  title={Generative modeling by estimating gradients of the data distribution},
  journal={NeurIPS},
  year={2019}
}

@article{MW1996,
author = {J. Masoliver and G. H. Weiss},
title = {Finite-velocity diffusion}, 
 journal={European Journal of Physics},
volume = {17},
year = {1996}, 
pages = {190--196}
}

@article{sun2024dynamicalmeasuretransportneural,
      title={Dynamical Measure Transport and Neural {PDE} Solvers for Sampling}, 
      author={J. Sun and J. Berner and L. Richter and M. Zeinhofer and J. Müller and K. Azizzadenesheli and A. Anandkumar},
      year={2024},
      journal={arXiv:2407.07873},
      primaryClass={cs.LG},
}

@book{PlPoStTa23,
  author = {Gerlind Plonka and Daniel Potts and  Gabriele Steidl and Manfred Tasche},
  title = {Numerical Fourier Analysis},
  year = {2023},
  OPTseries = {Applied and Numerical Harmonic Analysis},
  publisher = {Birkh\"auser},
  address = {Basel},
  edition = {2nd},
  isbn = {978-3-031-35005-4},
  doi = {10.1007/978-3-031-35005-4},
}

@article{nüsken2024steintransportbayesianinference,
      title={Stein transport for {B}ayesian inference}, 
      author={N. Nüsken},
      year={2024},
      journal={arXiv:2409.01464},
       url={https://arxiv.org/abs/2409.01464},
}

@article{DVK2022,
author={Dockhorn, T. and Vahdat, A. and Kreis, K},
year = {2022}, 
title={Score-Based Generative Modeling with Critically-Damped {L}angevin Diffusion},
journal={ICLR}
}

@article{BBRN2025,
author={Blessing, D. and Berner, J. and Richter, L. and Neumann, G}, title={Underdamped diffusion bridges with applications to sampling}, 
journal={ICLR},
year = {2025}
}

@article{BS2013,
  title = {MIMICKING AN {I}T\^O PROCESS BY A SOLUTION OF
A STOCHASTIC DIFFERENTIAL EQUATION},
  volume = {23},
    number = {4},
  journal = {The Annals of Applied Probability},
   author = {Brunick, G. and Shreve, S.},
  year = {2013},
    pages = {1584-–1628}
}

@article{G1986,
  title = {Mimicking the one-dimensional marginal distributions of processes having an {It\^o} differential},
  volume = {71},
  journal = {Probabability Theory and Related Fields},
   author = {Gy\"ongy, I.},
  year = {1986},
 pages = {501--516}
}

@book{S1982,
  title = {Lectures on Topics in Stochastic Differential Equations},
  publisher = {Tata Institute of Fundamental Research \& Springer},
  author = {Stroock, D.W.},
  year = {1982},
 }

@article{HLQ24,
  title = {Functional large deviations for {K}ac–{S}troock approximation to a class of {G}aussian processes with application to small noise diffusions},
  volume = {37},
    number = {4},
  journal = {Journal of Theoretical Probability},
   author = {Hui,  J. and Lihu,  X. and Qingshan,  Y.},
  year = {2024}, 
  pages = {3015--3054}
}

@ARTICLE{J1990,
journal={Journal of Theoretical Probability}, 
Volume = {3},
number = {2},
title={The distance between the {Kac} process and
the {Wiener} process with applications to generalized
telegraph equations},
pages = {349--360},
author={A. Janssen},
 year={1990}
}

@article{Lutz1977,
  title = {Which operators generate cosine operator functions?},
  volume = {63},
    number = {5},
  journal = {Atti della Accademia Nazionale dei Lincei. Classe di Scienze Fisiche, Matematiche e Naturali. Rendiconti, Serie 8},
   author = {D. Lutz},
  year = {1977}, 
  pages = {314--317}
}

@article{lipman2023flow,
title={Flow matching for generative modeling},
author={Y. Lipman and R. Chen and H. Ben-Hamu and M. Nickel and M. Le},
journal={ICLR},
year={2023}
}

@article{liu2023flow,
title={Flow Straight and Fast: Learning to Generate and Transfer Data with Rectified Flow},
author={X. Liu and Ch. Gong and Q. Liu},
journal={ICLR},
year={2023},
}

@incollection{WS2025,
  title={{Flow Matching: Markov kernels, stochastic processes and transport plans}},
  author={Wald, C. and Steidl, G.},
  booktitle={Variational and Information Flows in Machine Learning and Optimal Transport, Oberwolfach Seminars. Vol. 56},
editors = {Wuchen Li and
Bernhard Schmitzer and
Gabriele Steidl and
Francois-Xavier Vialard and
Christian Wald},
publisher = {Birkh\"auser},
  pages={185--254},
  year={2025}
}

@article{MM2024,
  title={Sampling in unit time with kernel {Fisher-Rao} flow},
  author={Maurais, A. and Marzouk, Y.},
  journal={ICML},
  year={2024}
}

@article{máté2023learning,
      title={Learning interpolations between {B}oltzmann densities}, 
      author={B. Máté and F. Fleuret},
      year={2023},      
      journal={Transactions on Machine Learning Research},
      primaryClass={stat.ML}
}

@article{CWDS2025,
  title={{Neural sampling from Boltzmann densities: Fisher-Rao curves in the Wasserstein geometry}},
  author={Chemseddine, J. and Wald, C. and Duong, R. and Steidl, G.},
  journal={ICLR},
  year={2025}
}

@article{GTC2025,
  title={Complexity Analysis of Normalizing Constant Estimation: from {Jarzynski} Equality to Annealed Importance Sampling and beyond},
  author={W. Guo and M. Tao and Y. Chen},
  journal={arXiv:2502.04575},
  year={2025}
}

@article{C1958,
author={C. Cattaneo}, 
title={Sur une forme de l'équation de la chaleur éliminant le paradoxe d'une propagation instantanée}, 
journal = {Comptes Rendus.}, 
volume = {247},
pages = {431--433},
year = {1958}
}

@article{V1958,
author={P. Vernotte}, 
title={Les paradoxes de la theorie continue de l'équation de la chaleur}, 
journal = {Comptes Rendus.}, 
volume = {246},
pages = {3154--3155},
year = {1958}
}

@article{C1963,
author={M. Chester}, 
title={Second sound in solids}, 
journal = {Physical Review}, 
volume = {131},
pages = {2013--2015},
year = {1963}
}

@article{song2021scorebasedgenerativemodelingstochastic,
      title={Score-based generative Modeling through stochastic differential equations}, 
      author={Y. Song and J. Sohl-Dickstein and D. P. Kingma and A. Kumar and S. Ermon and B. Poole},
      year={2021},
      journal={ICLR},     
}

@article{soft2021,
      title={{Soft truncation: A universal training technique of score-based diffusion model for high precision score estimation}}, 
      author={D. Kim and S. Shin and K. Song and W. Kang and I.-C. Moon},
      year={2022},      
      journal={ICML},
      primaryClass={stat.ML}
}

@article{Poisson2022,
  title={{Poisson flow generative models}},
  author={Y. Xu and Z. Liu and M. Tegmark and T. Jaakkola},
  journal={NeurIPS},
  year={2022}
}

@article{GenPhys2023,
  title={{GenPhys: from physical processes to generative models}},
  author={Z. Liu and D. Luo and Y. Xu and T. Jaakkola and M. Tegmark},
  journal={arXiv:2304.02637},
  year={2023}
}

@article{KAC1974,
  title={A STOCHASTIC MODEL RELATED TO THE TELEGRAPHER'S EQUATION},
  author={Kac, M.},
  journal={Rocky Mountain Journal of Mathematics},
  Volume = {4},
number = {3},
pages = {497--509},
  year={1974}
}

@article{P2022,
      title={Score-Based Generative Models Detect Manifolds}, 
      author={J. Pidstrigach},
      year={2022},
      journal={NeurIPS},     
}

@article{ZYM2018,
  title={Revisiting {Kac}’s method: A {Monte Carlo} algorithm for solving the Telegrapher’s equations},
  author = {B. Zhang and W. Yu and M. Mascagni},
journal = {Mathematics and Computers in Simulation},
volume = {156},
pages = {176--198},
year = {2018}
}

@article{GH1971,
  title={Theory of Random Evolutions with Applications to Partial Differential Equations},
  author = {R. Griego and R. Hersh},
journal = {Transactions of the American Mathematical Society},
volume = {156},
pages = {405--418},
year = {1971}
}

@article{K1993,
  title={On the Probabilistic Representation of a Solution of the Telegraph Equation},
  author = {Y. M. Kabanov},
journal = {Theory of Probability and its Applications},
volume = {37},
pages= {379--380},
year = {1993}
}

@article{TL2016,
  title={Application of the three-dimensional telegraph equation to cosmic-ray transport},
  author = {R. C. Tautz and I. Lerche},
journal = {Research in Astronomy and Astrophysics},
volume = {16},
number = {10},
pages = {162--170},
year = {2016}
}

@article{N2020,
  title={DISTRIBUTIONAL SOLUTIONS FOR DAMPED
WAVE EQUATIONS},
  author = {M. Nualart},
journal = {Electronic Journal of Differential Equations},
volume = {2020},
number = {131},
pages = {1--16},
year = {2020}
}

@book{BookAmGiSa05,
    author = {L. Ambrosio and N. Gigli and G. Savaré},
    title = {Gradient Flows},
    subtitle = {in Metric Spaces and in the Space of Probability Measures},
    publisher = {Birkh\"auser, Basel},
    year = {2008},
    series = {Lectures in Mathematics ETH Zürich},
    edition = {2nd},
    doi = {10.1007/978-3-7643-8722-8}
   }

@article{tong2024improving,
  title={Improving and generalizing flow-based generative models with minibatch optimal transport},
  author={Tong, Alexander and Fatras, Kilian and Malkin, Nikolay and Huguet, Guillaume and Zhang, Yanlei and Rector-Brooks, Jarrid and Wolf, Guy and Bengio, Yoshua},
  journal={Transactions on Machine Learning Research},
  pages={1--34},
  year={2024}
}

@misc{obukhov2020torchfidelity,
  author={Anton Obukhov and Maximilian Seitzer and Po-Wei Wu and Semen Zhydenko and Jonathan Kyl and Elvis Yu-Jing Lin},
  year=2020,
  title={High-fidelity performance metrics for generative models in PyTorch}
}

@inproceedings{nichol2021improved,
  title={Improved denoising diffusion probabilistic models},
  author={Nichol, Alexander Quinn and Dhariwal, Prafulla},
  booktitle={International conference on machine learning},
  pages={8162--8171},
  year={2021},
  organization={PMLR}
}

%--------------------------------------------------------------------------
%--------------------------------------------------------------------------
\appendix
\section{Appendix: Additional Material}\label{sec:appendix}
\begin{figure}[H]
    \centering
    \includegraphics[width = 0.22\textwidth]{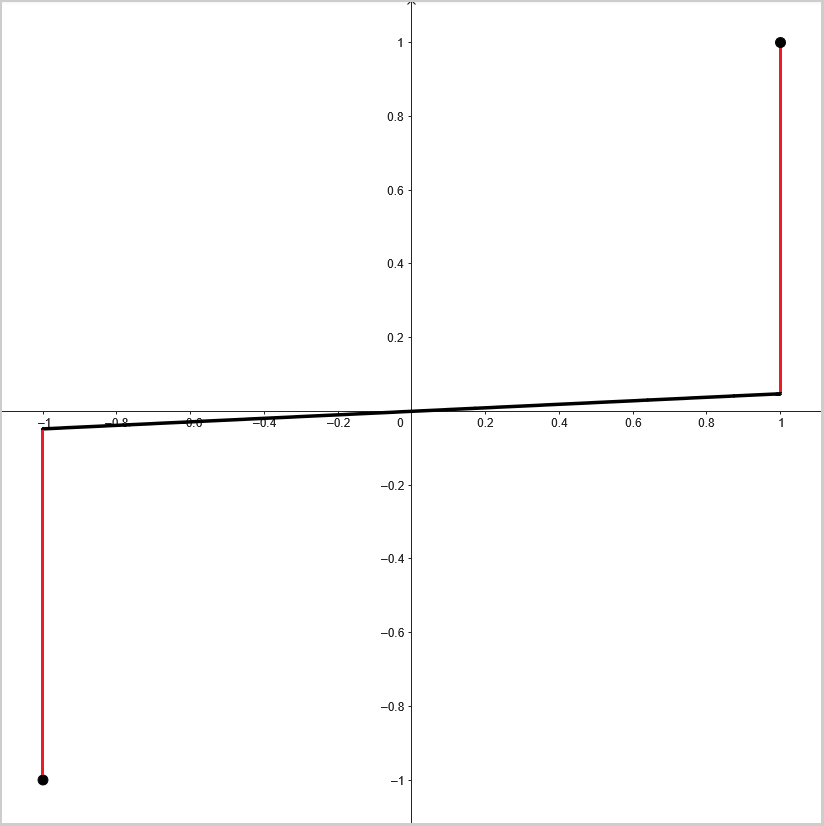} \hspace{2mm}
    \includegraphics[width = 0.22\textwidth]{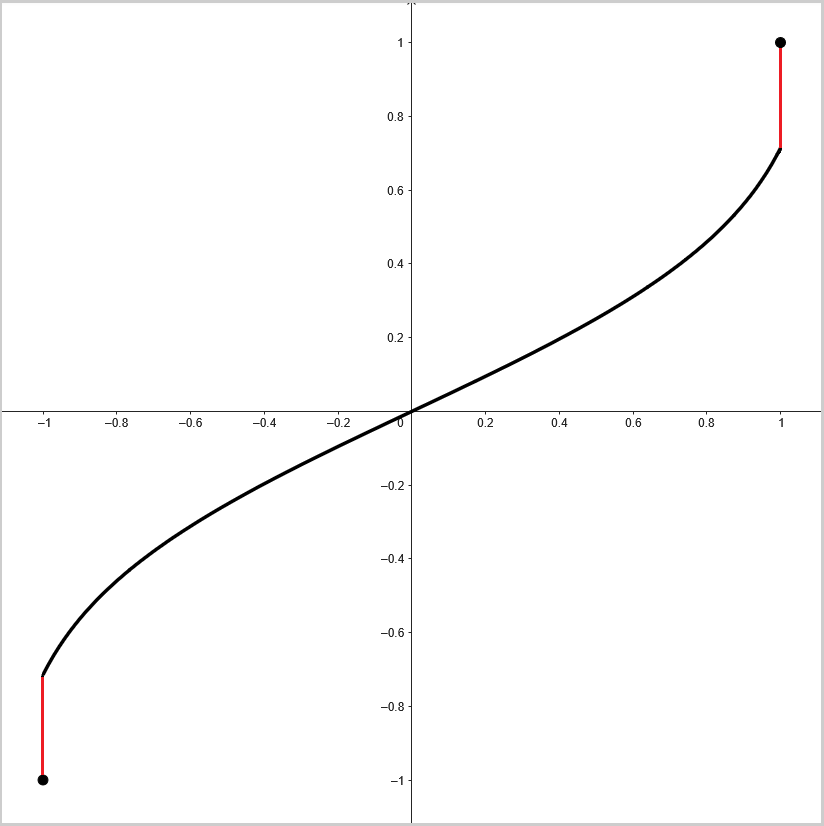} \hspace{2mm}
    \includegraphics[width = 0.22\textwidth]{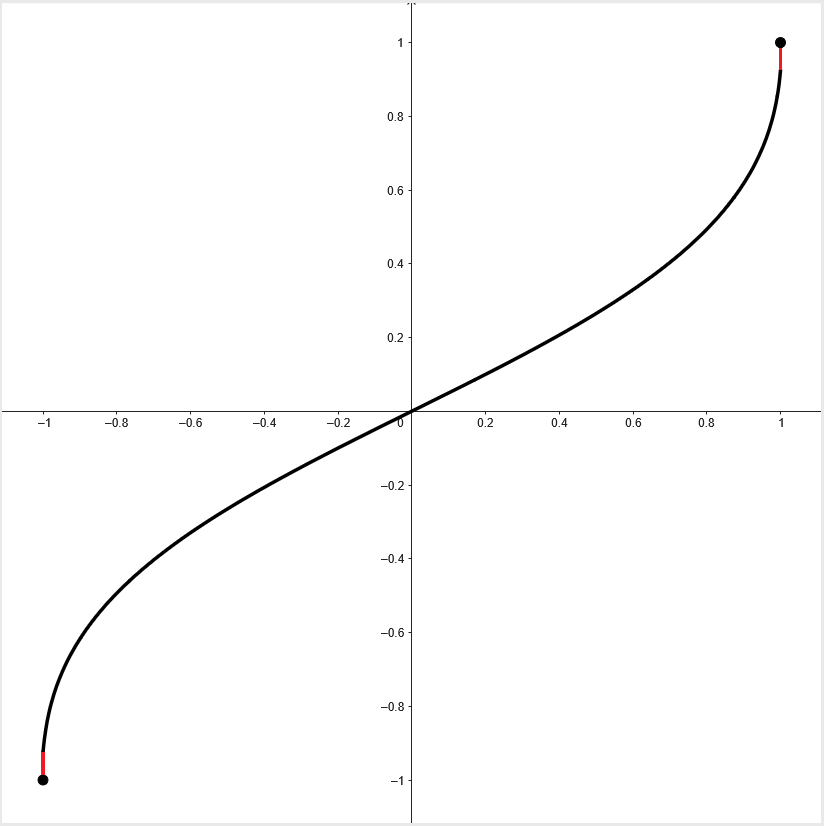} \hspace{2mm}
    \includegraphics[width = 0.22\textwidth]{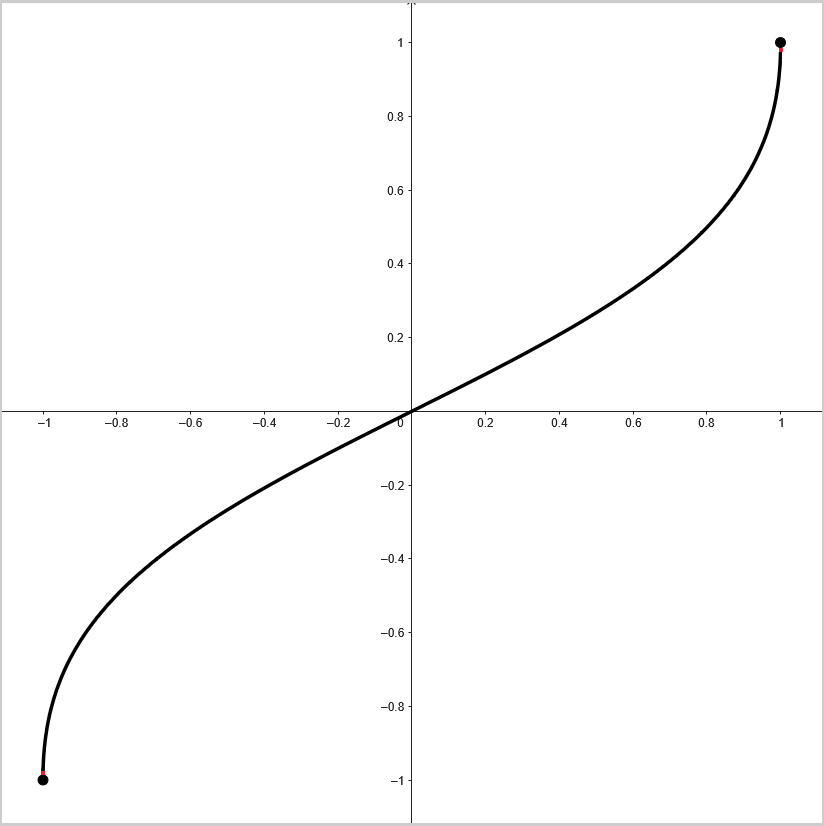} 
    \vspace{8mm}

    \hspace{-0mm}
    \includegraphics[width = 0.19\textwidth]{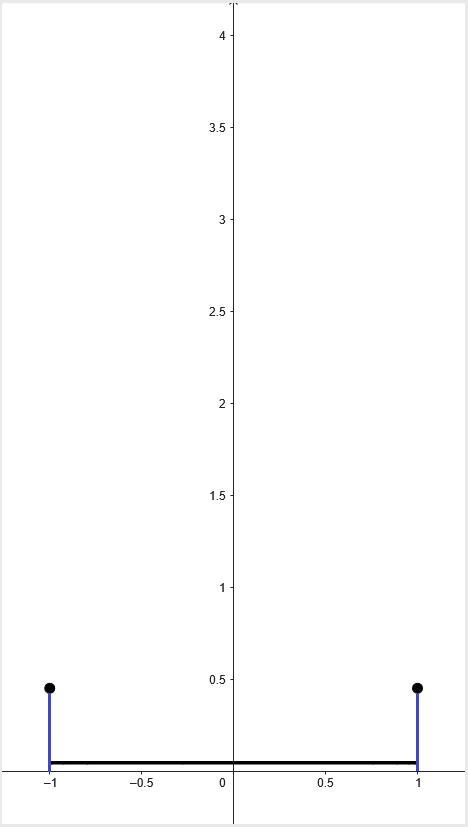} \hspace{7mm}
    \includegraphics[width = 0.19\textwidth]{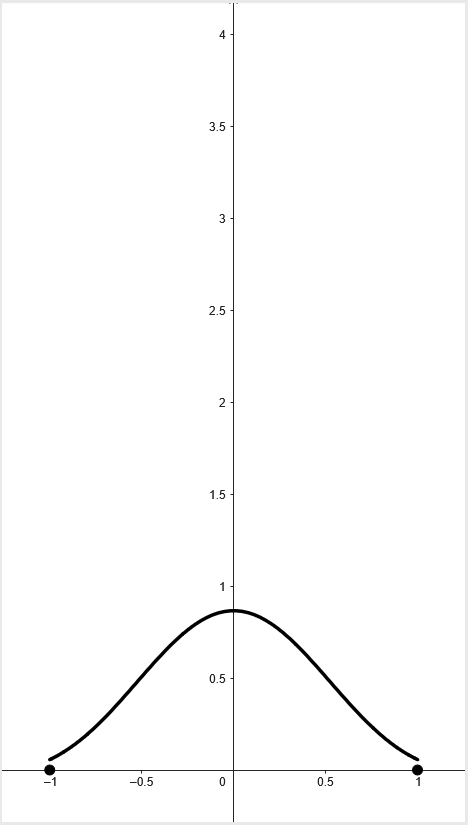} \hspace{7mm}
    \includegraphics[width = 0.19\textwidth]{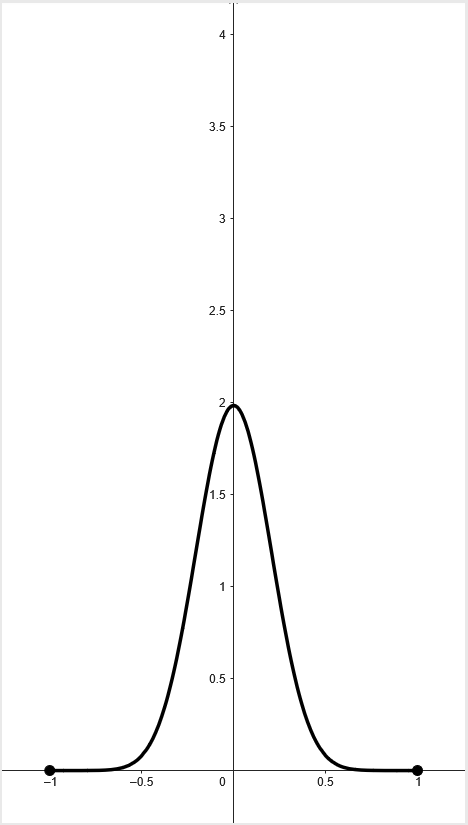} \hspace{8mm}
    \includegraphics[width = 0.19\textwidth]{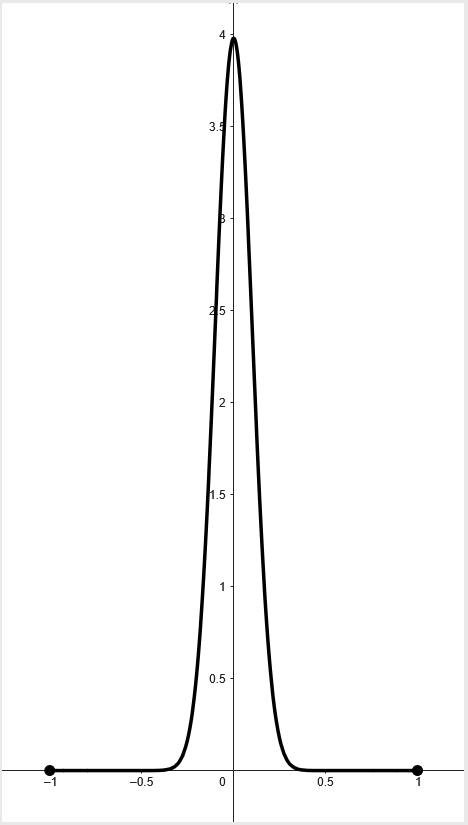}
       
    \caption{\emph{Top row:} 1D velocity fields of the latent variable $K_T$ for $T=1$, fixed $c=1$, and various dampings $a= \frac{1}{10}, 5, 25, 100$ (left to right). The red lines symbolize exploding Lipschitz constants of the velocity fields at times $t$ close to $T$. Higher values of $a$ seem to have a regularizing effect. \\
    \emph{Bottom row:} Corresponding probability distributions. For low dampings $a \approx 0$, the singularities of the velocity field are weighted by masses highly concentrated at the boundary.
    For higher dampings $a$, the mass becomes more concentrated at the origin where the velocitiy field is smooth.}
    \label{role-of-a}
\end{figure}

%--------------------------------------------------------------------------
%--------------------------------------------------------------------------

\subsection{Lipschitz property of the Mean-reverting Kac process}\label{app:lipschitz}

Let us consider the $d$-dimensional versions $\mathbf{X}_0 \coloneqq (X_0^1, \dots, X_0^d)$ and $\mathbf{M}_t \coloneqq (M_t^1, \dots, M_t^d)$ of the data and the mean-reverting Kac process \eqref{eq:mean-rev-kac}. Only for notational simplicity, we assume the same scheduling functions, damping and speed parameters for each component, meaning that $\mathbf{M}_t = f(t) \mathbf{X}_0 + \mathbf{K}_{g(t)}$. 
Similar to Proposition \ref{prop:lipschitz}, we have the following Lipschitz property.

\begin{proposition}\label{prop:lipschitz_2}
   Let $P_{\mathbf{X}_0} \in \mathcal{P}_2(\R^d)$, and let $f,g: [0,1] \to \R$ be Lipschitz continuous with constants $C_{f}, C_{g} > 0$, respectively. 
   Then, for the probability distribution flow  $(\mu_t)_{t \in [0,1]} \coloneqq (P_{\mathbf{M}_t})_{t \in [0,1]}$ of the $d$-dimensional mean-reverting Kac process $\mathbf{M}_t$, the following holds true:
   \begin{itemize}
    \item[i)]
    $(\mu_t)_{t \in [0,1]}$ is Lipschitz continuous in the Wasserstein space $(\mathcal{P}_2(\R^d),W_2)$ with Lipschitz constant 
    $\sqrt{2 C_f^2 \, \|\mathbf{X}_0\|_{L_2(\mathbb{P})}^2 + 2 C_g^2 \, d c^2}$, and for $\mathbb{P}$-a.e.\ 
    $\omega$, $\mathbf{X}_t(\omega)$ is Lipschitz continuous in $\R^d$ with the constant $\sqrt{2 C_{f}^2 \, \|\mathbf{X}_0(\omega)\|^2 + 2 C_{g}^2 \, d c^2}$.
    \item[ii)] $(\mu_t)_{t \in [0,1]}$ is an absolutely continuous curve in $AC^2([0,1];\mathcal P_2(\R^d))$, and admits a velocity field fulfilling
     \begin{equation}\label{kac-boundedness_2}
    \|v_t\|_{L_2(\mu_t)}^2  \le 2 C_f^2 \, \|\mathbf{X}_0\|_{L_2(\mathbb{P})}^2 + 2 C_g^2 \, d c^2 \quad \text{ for a.e. } t \in [0, 1].
    \end{equation}
    \item[iii)] 
    If $|\mathbf{X}_0|$ is bounded by some $R > 0$ almost surely, then $(\mu_t)_{t \in [0,1]}$ admits a velocity field which satisfies the uniform (in time and space) bound 
    \begin{equation}\label{kac-boundedness_3_uniform}
      |v_t(x)| \le C_f R + C_g d^\frac{1}{2} c  \quad \text{for all } x \in \R^d, \, t \ge 0.
    \end{equation} 
   \end{itemize}
The assumptions on $f,g$ are in particular fulfilled by linear/quadratic schedules $f(t) \coloneqq 1-t$ and $g(t) \coloneqq t^2$, or $g(t) \coloneqq t$.
\end{proposition}

\begin{proof}
i)  
    Fix $0 \le s < t \le 1$. Applying the estimate \eqref{lipschitz-estimate} and Young's inequality, we obtain for $\mathbb{P}$-a.e. $\omega$ that
    \begin{align}\label{lipschitz-estimate_2}
        \|\mathbf{M}_t(\omega) - \mathbf{M}_s (\omega)\|^2
        &= \|(f(t)-f(s)) \mathbf{X}_0(\omega) +    \mathbf{K}_{g(t)} (\omega)- \mathbf{K}_{g(s)} (\omega)\|^2 \\
        %&= \sum_{i=1}^d |K_i(t)(\omega) - K_i(s)(\omega)|^2 \nonumber\\
        %&= \sum_{i=1}^d \big|c \,  {\rm B_{\frac12}} (\omega) \int_s^t (-1)^{N(z)(\omega)} \d z \big|^2 \nonumber\\
        &\le 2 C_{f}^2 |t-s|^2 \|\mathbf{X}_0(\omega)\|^2 + 2 d  c^2 C_{g}^2  |t-s|^2.
    \end{align}
    This shows that for $\mathbb{P}$-a.e.\ $\omega$, $\mathbf{M}_t(\omega)$ is Lipschitz continuous in $\R^d$ with $\omega$-dependent Lipschitz constant
    $\sqrt{2 C_{f}^2 \, \|\mathbf{X}_0(\omega)\|^2 + 2 C_{g}^2 \, d c^2}$.

    Now, we
    consider the canonical transport plan  $\pi : \R^d \times \R^d \to \R$ defined by
    $\pi(B) := 
    \mathbb{P} \big( (\mathbf{M}_s, \mathbf{M}_t) \in B \big)$
    for all Borel measurable $B \in \mathcal B(\R^d \times \R^d)$.
    Indeed, it holds $\pi(A \times \R^d) = P_{\mathbf{M}_s}(A)$ 
    and 
    $\pi(\R^d \times A) = P_{\mathbf{M}_t}(A)$ 
    for all $A \in  \mathcal B(\R^d)$.
    Furthermore, we have
    $\pi \in \mathcal{P}_2(\R^d \times \R^d)$:
     it holds that
    \begin{align}
       &\int_{\R^d \times \R^d} \|(x,y)\|^2 \d \pi(x,y) 
       %= \int_{\Omega} \|(\mathbf{X}(s),\mathbf{X}(t))\|^2 \d \mathbb{P}
       = \int_{\Omega} \|(f(s)\mathbf{X}_0 + \mathbf{K}_{g(s)}, f(t)\mathbf{X}_0 + \mathbf{K}_{g(t)})\|^2 \d \mathbb{P}\\
       &\le \int_{\Omega} 2 \|f\|_\infty^2 \|\mathbf{X}_0\|^2 + 2\|\mathbf{K}_{g(s)}\|^2 + 2 \|f\|_\infty^2 \|\mathbf{X}_0\|^2 + 2\|\mathbf{K}_{g(t)}\|^2\d \mathbb{P}  \\
       &\le \int_{\Omega} 2 \|f\|_\infty^2 \|\mathbf{X}_0\|^2 + 2 d c^2 \|g\|_\infty^2 + 2 \|f\|_\infty^2 \|\mathbf{X}_0\|^2 + 2 d c^2 \|g\|_\infty^2 \d \mathbb{P}  ~ <  \infty,
    \end{align}
    since $P_{\mathbf{X}_0} \in \mathcal{P}_2(\R^d)$, and by using \eqref{lipschitz-estimate} with $s=0$ or $t=0$, respectively.  
    Hence, $\pi$ is an admissible transport plan between $\mu_s$ and $\mu_t$.
    
    Then, we conclude by \eqref{lipschitz-estimate_2} that
    \begin{align}
        W_2^2(\mu_s, \mu_t) 
        &\le 
        \int_{\R^d \times \R^d} \|x-y\|^2 \d \pi(x,y)
        = \int_{\Omega} \|\mathbf{M}_s - \mathbf{M}_t\|^2 \d \mathbb{P}\\
        &\le \left(2 C_f^2 \|\mathbf{X}_0\|_{L_2(\mathbb{P} )}^2 + 2 d c^2 C_g^2 \right)  |t-s|^2 .
    \end{align}
    Hence, $(\mu_t)_{t \in [0,1]}$ is Lipschitz continuous in $\mathcal{P}_2(\R^d)$ with constant 
    $(2 C_f^2 \, \|\mathbf{X}_0\|_{L_2(\mathbb{P})}^2 + 2 C_g^2 \, d c^2)^\frac{1}{2}$.
    \\[1ex]
    ii) By definition \eqref{eq:abs_cont} and Part i), we see immediately that $(\mu_t)_{t \in [0,1]}$ is an absolutely continuous curve. It then follows from \cite[Theorem 8.3.1]{BookAmGiSa05} that there exists an (optimal) velocity field in the continuity equation fulfilling \eqref{kac-boundedness_2}.  \\[1ex]
    iii) Note that the conditional velocity field \eqref{eq:velo-OU} satisfies the bound \eqref{kac-boundedness_3_uniform}. Now, the formulas \eqref{velo-FM} and \eqref{total_prob} applied to the process $\mathbf{M}_t$ yield the result.
\end{proof}

\begin{remark}[Comparison to bounds for diffusion]
1. The uniform Kac bound \eqref{kac-boundedness_3_uniform} is better than the growth estimate which can be obtained for the stochastic interpolant \eqref{eq:rt-def}: denoting its velocity field by $v_t^{\rm diff}$, it is known that 
 \begin{equation}
    |v_t^{\rm diff} (x)| < C_{t,1} |x| + C_{t,2},  \quad t \in (0,1),
  \end{equation}
where $C_{t,1}, C_{t,2}$ explode for $t$ close to inital (data) time, see \cite[Appendix C.1]{chen2025scaleadaptivegenerativeflowsmultiscale}. The example \eqref{velo-ind} shows that this cannot be improved to a uniform bound as in \eqref{kac-boundedness_3_uniform}. \\[0.5ex]
2. A bound on the norm $\|v_t\|_{L_2(\mu_t)}$ similar to \eqref{kac-boundedness_2} can also be obtained for the stochastic interpolant \eqref{eq:rt-def}, assuming that $f$ and $\sqrt{g}$ are Lipschitz continuous. Therefore, the choice $g(t) \coloneqq t$ is \emph{not} admissible for the diffusion case, but valid for the Kac framework. Table \ref{fid_results} exactly reflects this, and interestingly, the Kac model achieves its best FID result for the schedule $g(t) = t$ and $(a,c) = (25,2)$.
\end{remark}

%--------------------------------------------------------------------------
%--------------------------------------------------------------------------

\subsection{Rigorous proof of the decomposition Lemma \ref{cont_eq_decomposes}}\label{sect:decomposition-lemma}
The decomposition Lemma \ref{cont_eq_decomposes} is commonly known in the machine-learning literature. Yet, we want to give a mathematically rigorous proof of it here. Although the formal derivation of Lemma \ref{cont_eq_decomposes} is trivial, it relies on the assumption that the corresponding probability densities exist, and are smooth. Since this is not the case in our setting (recall that the Kac probability distributions \eqref{telegraph-formula-1d-dirac} admit singular Dirac points), we need to treat the continuity equation \eqref{eq:ce} in the weak formulation.

To this end, fix the dimension $d \ge 1$ and recall that the continuity equation \eqref{eq:ce} is said to 
hold on $(0,T) \times \R^d$ in the \emph{sense of distributions}, if it holds
\begin{equation} \label{eq:CE_distr}
    \int_0^T \int_{\R^d} \partial_t \psi(t, x) + v_t(x) \cdot \nabla_x \, \psi(t, x) \d \mu_t(x) \d{t}
    = 0
\end{equation}
for all $\psi \in C_{\mathrm c}^\infty((0,T) \times \R^d)$. Here, for an open set $D \subset \R^l$, the set $C_{\mathrm c}^\infty(D)$ denotes the space of all infinitely diffentiable functions on $\R^l$ which are compactly supported on $D$.

Define the tensor product space $C_{\mathrm c}^\infty((0,T)) \otimes \bigotimes_{i=1}^d C_{\mathrm c}^\infty(\R)$ as the linear span of functions
\begin{align}\label{tensor-prod-function}
\psi: (0,T) \times \R^d \to \R, \quad (t, x_1, \dots, x_d) \mapsto \psi_0(t) \prod_{i=1}^d \psi_i(x_i),
\end{align}
with functions $\psi_0 \in C_{\mathrm c}^\infty((0,T)), \, \psi_i \in C_{\mathrm c}^\infty(\R), \, i=1,\dots, d$.

By the following well-known result, this subspace is dense in $C_{\mathrm c}^\infty((0,T) \times \R^d)$.
\begin{lemma}\label{lem:tensor-dense}
    The space $C_{\mathrm c}^\infty((0,T)) \otimes \bigotimes_{i=1}^d C_{\mathrm c}^\infty(\R)$ is dense in $C_{\mathrm c}^\infty((0,T) \times \R^d)$ in the following sense: 
    for all $\psi \in C_{\mathrm c}^\infty((0,T) \times \R^d)$, there exists
    a compact set $K \subset (0,T) \times \R^{d}$ containing $\supp \, \psi$, and
    a sequence $(\psi_n) \subset C_{\mathrm c}^\infty((0,T)) \otimes \bigotimes_{i=1}^d C_{\mathrm c}^\infty(\R)$ compactly supported on $K$, such that $D^\alpha \psi_n$ converges uniformly to $D^\alpha \psi$ on $K$ for any $\alpha \in \N_0^{d+1}$.
\end{lemma}
\begin{proof}
    Fix $\psi \in C_{\mathrm c}^\infty((0,T) \times \R^d)$. Since $\psi$ is compactly supported on $(0,T) \times \R^d$, there exist compact sets $A_0 \subset (0,T), A_i \subset \R, \, i=1,\dots, d$, such that
    \begin{equation}
        \supp \, \psi \subset \overset{\circ}{A} \subset  A \coloneqq A_0 \times \dots \times A_d.
    \end{equation}
    Now, consider bump functions $\xi_0, \dots, \xi_d \in C_{\mathrm c}^\infty(\R)$ such that 
    \begin{equation}\label{constant-one}
        \xi_i \equiv 1 \quad \text{on } A_i, \, i=0,\dots,d.
    \end{equation}
    Defining $K_i \coloneqq \supp \, \xi_i, ~ i=0,\dots,d$, we see that $K \coloneqq K_0 \times \dots \times K_d$ is a compact set containing $\supp \, \psi$.
    Note that $\xi_0$ can be chosen such that $K_0 \subset (0,T)$, i.e. $\xi_0 \in C_{\mathrm c}^\infty((0,T))$. 
    Since $K$ is compact, the Stone-Weierstrass theorem yields a sequence of polynomials $(p_n)$ on $(0,T) \times \R^{d}$ such that $D^\alpha p_n$ converges uniformly to $D^\alpha \psi$ on $K$ for all $\alpha \in \N_0^{d+1}$. %not trivial, also diagonal argument

    Next, define the functions
    \begin{equation}
        \psi_n (t, x) \coloneqq
        \xi_0(t) \xi_1(x_1) \dots \xi_d(x_d) p_n(t,x), \quad t \in (0,T), ~ x= (x_1, \dots, x_d) \in \R^d.
    \end{equation}
    Since $p_n$ is a polynomial, it can be written as
    \begin{equation}
       p_n(t,x) = \sum_{i = (i_0, \dots, i_d) \in \N_0^{d+1}, |i| \le m_n} c_i \, t^{i_0} x_1^{i_1} \dots x_d^{i_d} 
    \end{equation}
    with $m_n$ being the degree of $p_n$.
    Hence, the functions $\psi_n$ are linear combinations of tensor product functions \eqref{tensor-prod-function}, i.e. $\psi_n \in C_{\mathrm c}^\infty((0,T)) \otimes \bigotimes_{i=1}^d C_{\mathrm c}^\infty(\R)$. By construction, $\psi_n$ is compactly supported on $K$.

    Fix an arbitrary $\alpha \in \N_0^{d+1}$.
    By \eqref{constant-one}, we have $\psi_n \equiv p_n$ on $A$, hence, $D^\alpha \psi_n \equiv D^\alpha p_n$ in the interior of $A$ which contains $\supp \, \psi$. By construction, $D^\alpha \psi_n$ converges uniformly to $D^\alpha \psi$ on $\supp \, \psi$. 
    On $K \setminus \supp \, \psi$, we have $\psi \equiv 0$. By construction, $D^\beta p_n$ converges uniformly to $0$ on $K \setminus \supp \, \psi$ for any $\beta \in \N_0^{d+1}$. By the Leibniz product rule and the boundedness of any derivatives of $\xi_0, \dots, \xi_d$, it follows that also $D^\alpha \psi_n$ converges uniformly to $0 = D^\alpha \psi$ on $K \setminus \supp \, \psi$, and in total, on $K$.\footnote{Since it holds $D^\alpha \psi_n \equiv D^\alpha \psi \equiv 0$ outside of $K$, the uniform convergence of the derivatives actually holds on the whole domain $(0,T) \times \R^d$.}
\end{proof}

Furthermore, we need the following handy proposition, which can be found in \cite[Proposition 16.3]{Ambrosio2021LecturesOO}. It allows for a more manageable treatment of the weak continuity equation \eqref{eq:CE_distr}.
\begin{proposition}\label{prop:useful}
    Let $\mu_t : (0,T) \to \mathcal{P}(\R^d)$ be a narrowly continuous curve of probability measures and let $v_t$ be a vector field such that $v_t \in L_1(\mu_t)$ and $\|v_t\|_{L_1(\mu_t)} \in L_1(0,T)$. Then the following are equivalent:
    \begin{itemize}
        \item [(i)] the curve $\mu_t$ solves \eqref{eq:CE_distr}.
        \item [(ii)] for any $\psi \in C_{\mathrm c}^\infty(\R^d)$, it holds that $t \mapsto \int_{\R^d} \psi(x) \d \mu_t(x)$ is absolutely continuous on $[0,T]$ and its derivative is $\int_{\R^d} \nabla \psi(x) \cdot v_t(x) \d \mu_t(x)$.
    \end{itemize}
\end{proposition}

Finally, we are able to rigorously prove Lemma \ref{cont_eq_decomposes}.
\begin{proof}[Proof of Lemma \ref{cont_eq_decomposes}]
Take any smooth test function $\psi$ of the form \eqref{tensor-prod-function}. Then, using the factorization assumption \eqref{product-decomposition} we have
\begin{align}
    \int_0^T \int_{\R^d} \partial_t \psi(t, x) \d \mu_t(x) \d{t}
    %=  \int_0^T \int_{\R^d} \partial_t \psi_0(t) \prod_{i=1}^d
    %   \psi_i(x_i)   \d \mu_t(x) \d{t}
    =  \int_0^T \partial_t \psi_0(t) \prod_{i=1}^d \left( \int_{\R} 
       \psi_i(x_i)   \d \mu_t^i(x_i) \right) \d{t}.
\end{align}
Since by assumption, the 1D flows $\mu_t^i$ satisfy the continuity equation \eqref{decomposed-CE} in the distributional sense, Proposition \ref{prop:useful}, $(i) \Rightarrow (ii)$, yields that the functions $t \mapsto \int_{\R} \psi_i(x_i) \d \mu_t^i(x_i)$ are absolutely continuous with weak (and strong a.e.) derivative $\int_{\R} v_t^i(x_i) \frac{\d}{\d x_i} \psi_i(x_i) \d \mu_t^i(x_i)$.

By integration by parts and the product rule for absolutely continuous functions, it follows
\begin{align}
    &= - \int_0^T \psi_0(t) ~ \partial_t \prod_{i=1}^d \left( \int_{\R} 
       \psi_i(x_i)   \d \mu_t^i(x_i) \right) \d{t}\\
    &= - \sum_{i=1}^d \int_0^T \psi_0(t) ~ \left( \int_{\R} v_t^i(x_i) \frac{\d}{\d x_i} \psi_i(x_i) \d \mu_t^i(x_i)\right) \prod_{j \ne i}^d \left( \int_{\R} 
       \psi_j(x_j)   \d \mu_t^j(x_j) \right) \d{t}\\   
    &= - \sum_{i=1}^d \int_0^T \int_{\R^d} v_t^i(x_i) \frac{\partial}{\partial x_i} \psi(t, x) \d \mu_t(x) \d{t}\\   
    &= -\int_0^T \int_{\R^d}  v_t(x) \cdot \nabla_x \, \psi(t, x) \d \mu_t(x) \d{t},
\end{align}
where the vector field $v_t$ is given by \eqref{velo-decomposition}. Note that $v_t$ satisfies the integrability \eqref{eq:explode_1} by the assumptions $\mu^i_t \in AC^2((0,T);\allowbreak \mathcal P_2(\R))$.
Hence, above equations also hold true for any linear combination of test functions $\psi$ given by \eqref{tensor-prod-function}. Finally, by Lemma \ref{lem:tensor-dense} and Lebesgue's dominated convergence we infer the continuity equation in the variational formulation \eqref{eq:CE_distr} for all $\psi \in C_{\mathrm c}^\infty((0,T) \times \R^d)$.
\end{proof}

Note that our general proof does \emph{not} make any use of smoothness or Markov properties of some underlying stochastic process, but solely works on the level of probability flows and continuity equations, where the smoothness requirement is shifted to the test functions.

%\emph{Acknowledgements.} We thank Flinn Raffael for correcting a mistake in the proof of Lemma  \ref{cont_eq_decomposes}.

%--------------------------------------------------------------------------
%--------------------------------------------------------------------------

\subsection{Simulating the Kac Process} \label{app:simulating_kac}
%--------------------------------------------------
\begin{figure}[ht!] \label{kac_walk_a4c2}
    \centering
    \includegraphics[width = 0.18\textwidth]{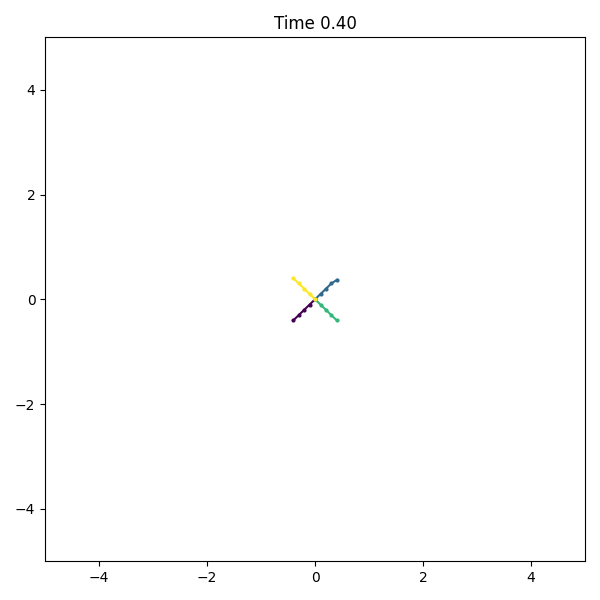}
    \includegraphics[width = 0.18\textwidth]{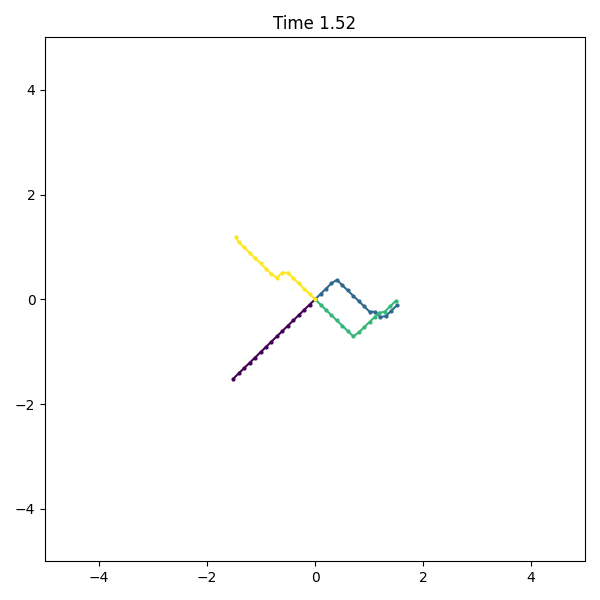}
    \includegraphics[width = 0.18\textwidth]{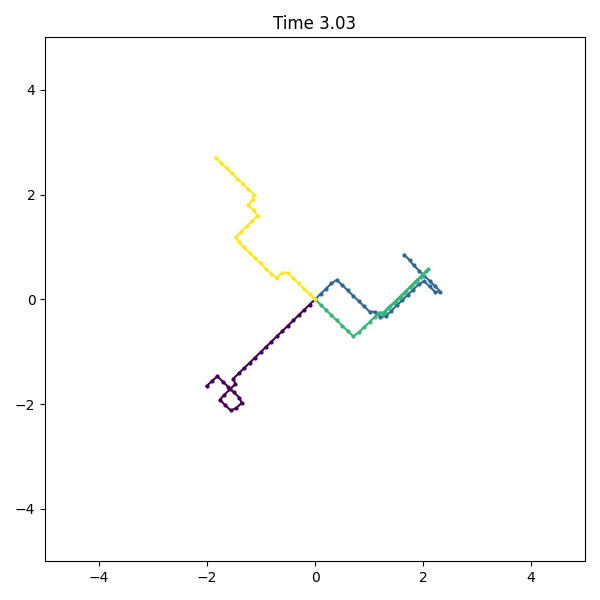}
    \includegraphics[width = 0.18\textwidth]{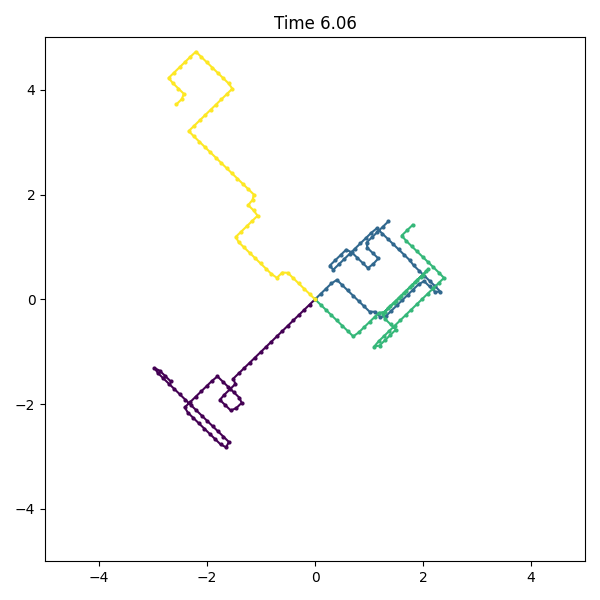}
    \includegraphics[width = 0.18\textwidth]{imgs/KAC_WALK_2D/a1c1_walk/frame_099.png}
    
    \includegraphics[width = 0.18\textwidth]{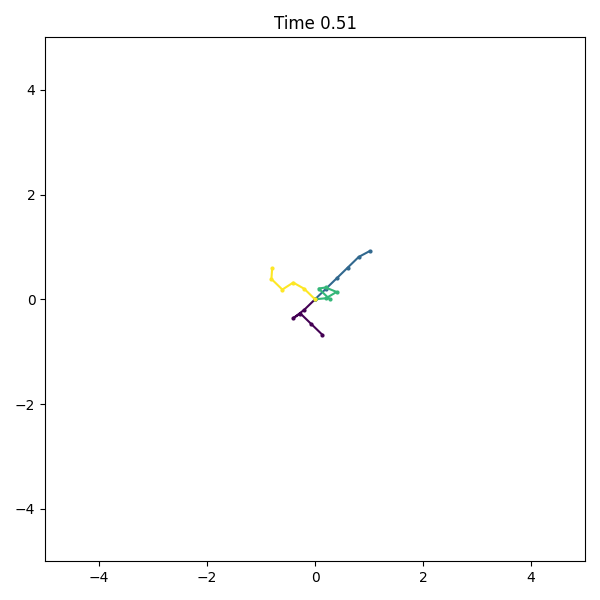}
    \includegraphics[width = 0.18\textwidth]{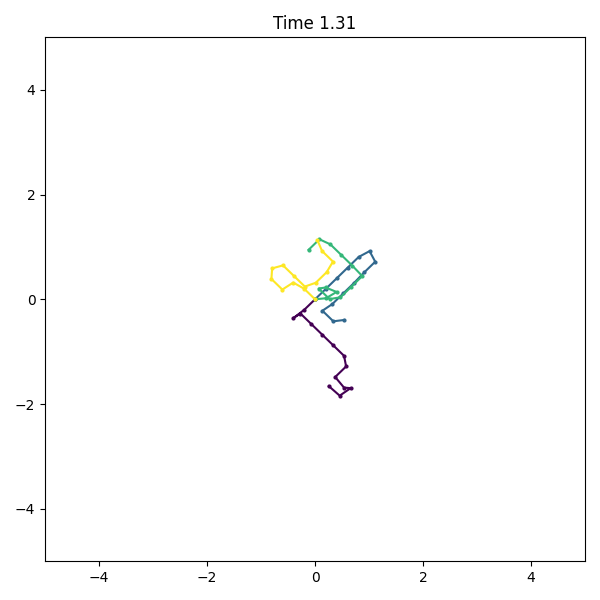}
    \includegraphics[width = 0.18\textwidth]{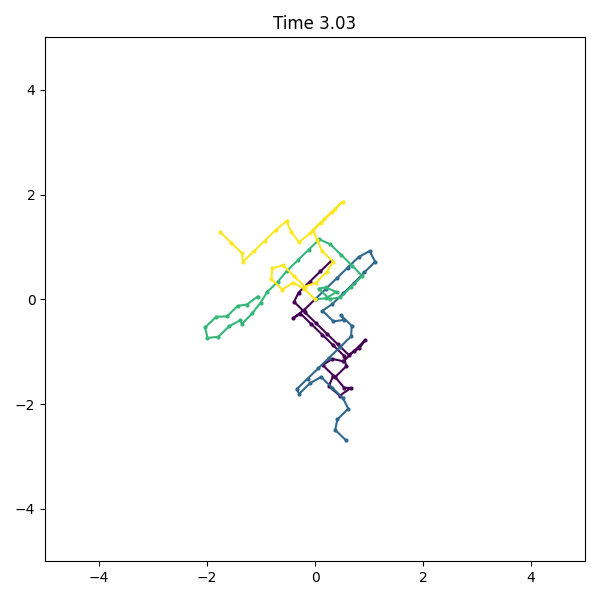}
    \includegraphics[width = 0.18\textwidth]{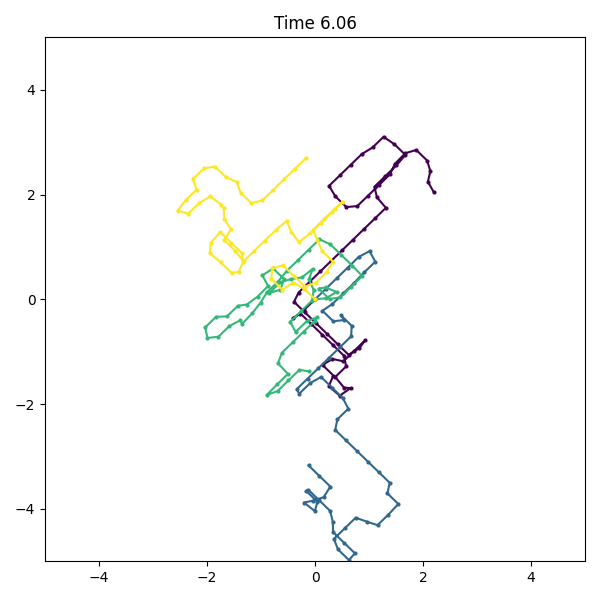}
    \includegraphics[width = 0.18\textwidth]{imgs/KAC_WALK_2D/a4c2_walk/frame_099.png}

    \includegraphics[width = 0.18\textwidth]{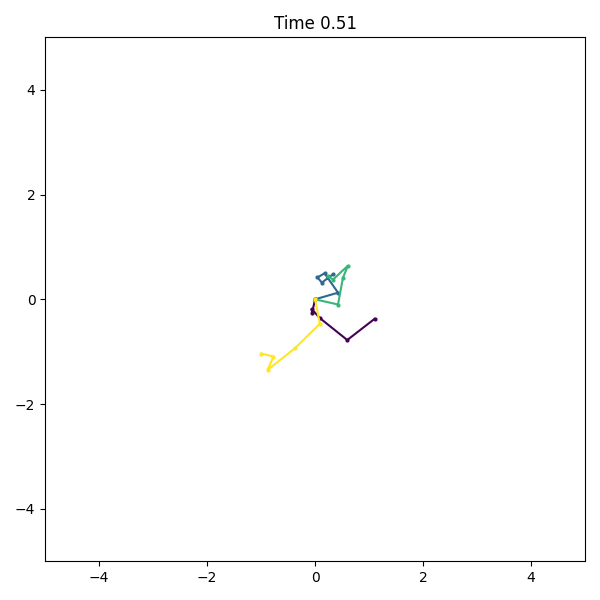}
    \includegraphics[width = 0.18\textwidth]{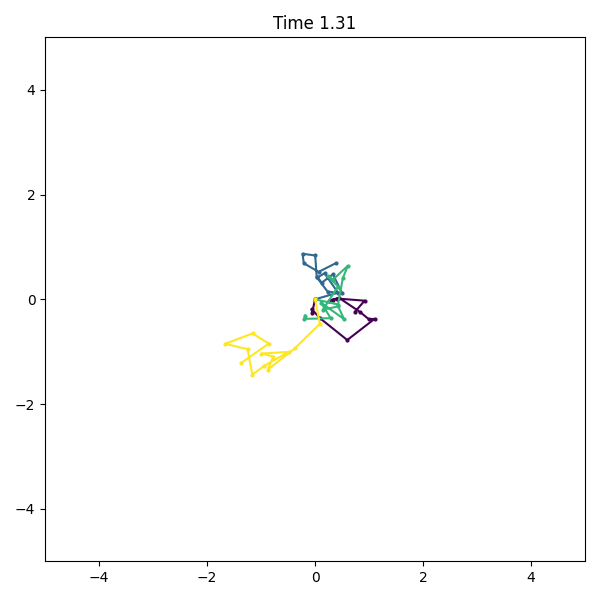}
    \includegraphics[width = 0.18\textwidth]{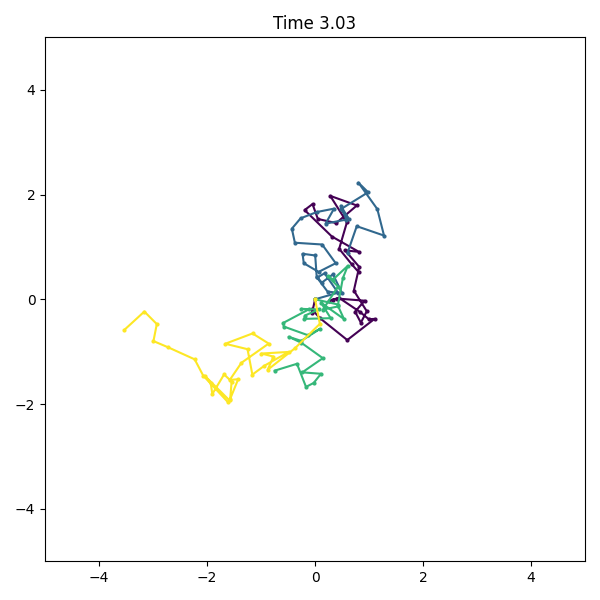}
    \includegraphics[width = 0.18\textwidth]{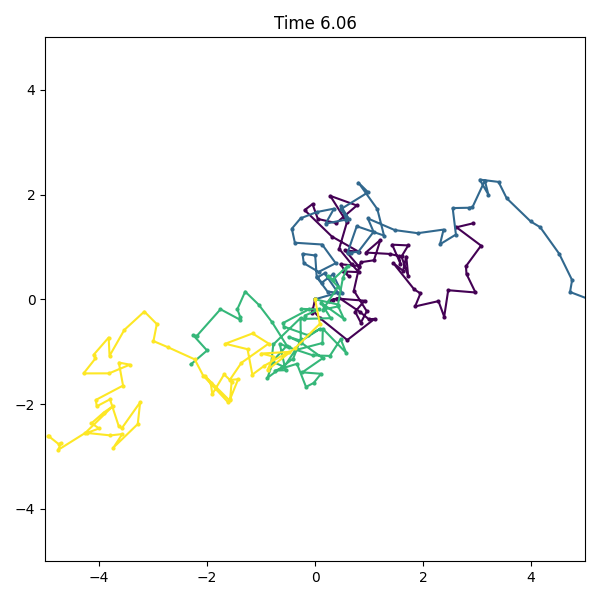}
    \includegraphics[width = 0.18\textwidth]{imgs/KAC_WALK_2D/a25c5_walk/frame_099.png}
       
    \caption{Paths of the componentwise Kac walk in 2D, simulated until time $T=10$ with $(a,c)=(1,1)$ (upper row), $(a,c) = (4,2)$ (middle row) and $(a,c) = (25,5)$ (lower row).}
\end{figure}

\begin{figure}[ht!] \label{brownian_walk_2d}
    \centering
    \includegraphics[width = 0.18\textwidth]{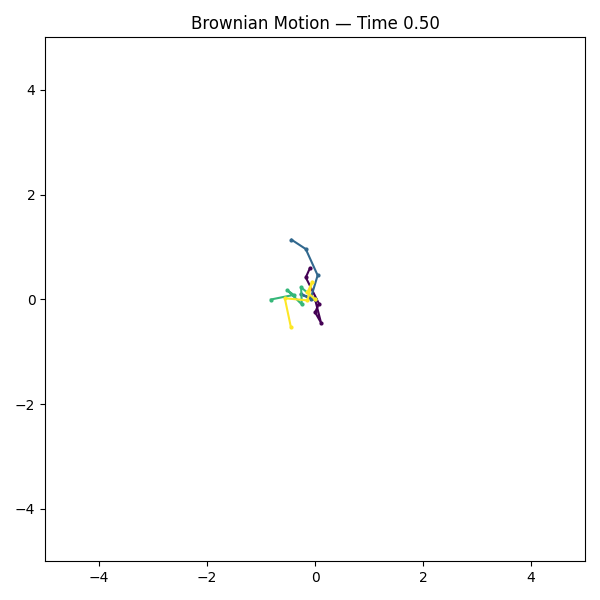}
    \includegraphics[width = 0.18\textwidth]{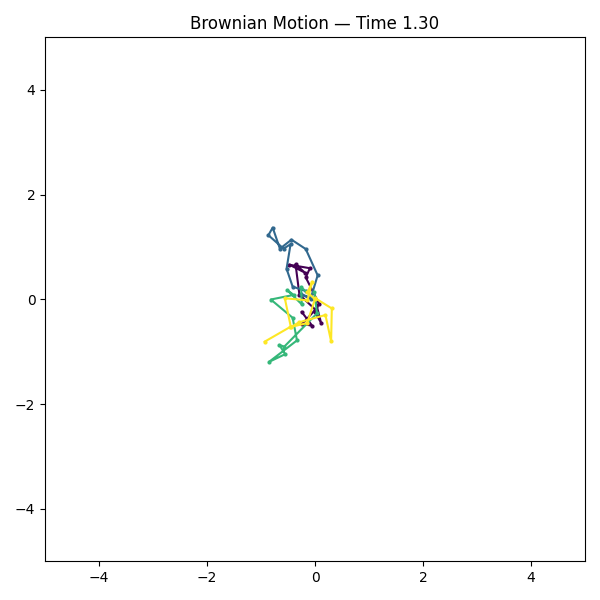}
    \includegraphics[width = 0.18\textwidth]{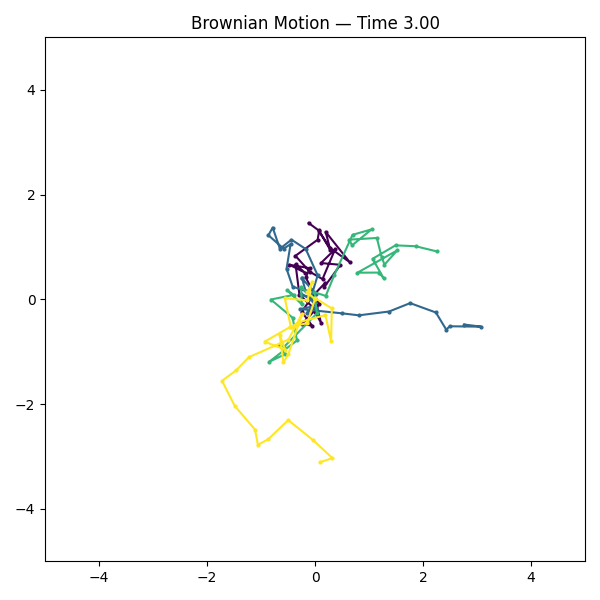}
    \includegraphics[width = 0.18\textwidth]{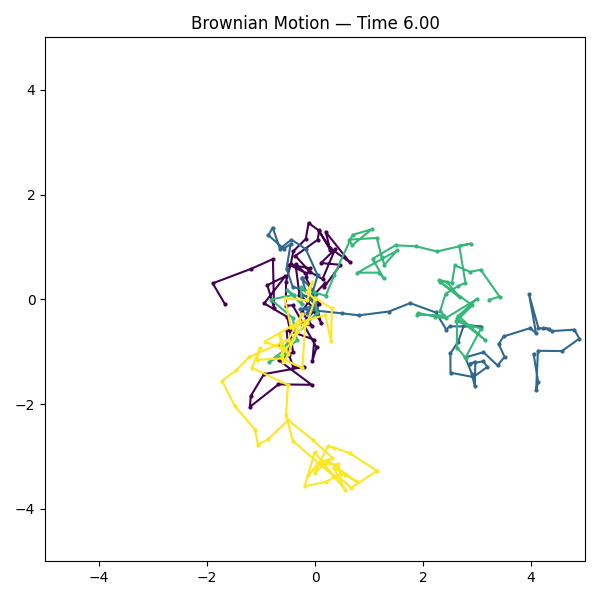}
    \includegraphics[width = 0.18\textwidth]{imgs/KAC_WALK_2D/brownian/frame_099.png}
       
    \caption{Paths of a standard Brownian motion in 2D, simulated until time $T=10$.}
\end{figure}

Recall that the Kac process $K_t$ starting in $0$ is given by 
\begin{align*}
    K(t) :=  {\rm B}_{\frac{1}{2}} c \,\tau_t,
\end{align*}
therefore we need to sample from the random times $\tau_t$. Since the process is i.i.d.\ we outline the procedure for the one-dimensional case, in higher dimensions the same procedure is done component-wise.

We follow the approach proposed in \cite{ZYM2018}. Let $N(t)$ be a Poisson process with rate $a > 0$, recall that $\tau_t = \int_0^t (-1)^{N(s)} \d s$. 
Suppose that the times at which the Poisson process jumps are given by $0 < j_1 < j_2 < \dots < j_{N(t)} < t$. Then $\tau_t$ can be expressed exactly as
\[
\tau_t = j_1 + \sum_{k=1}^{N(t)-1} (-1)^k (j_{k+1} - j_k) + (-1)^{N(t)} (t - j_{N(t)}).
\]
where we define $j_0 := 0$.
It is well known that the waiting times between jumps $s_k:=(j_{k+1} - j_k)$ are exponentially distributed with parameter $a$. We can therefore sample $s_k$ and compute the corresponding sum by determining the instance when their cumulative sum exceeds $t$.

To evaluate this quantity efficiently, we begin by sampling a fixed number of i.i.d.\ exponential random variables \( s_k \sim \mathrm{Exp}(a) \) for \( k = 1, \dots, \kappa \), where \( \kappa \) is a sufficiently large upper bound. We then compute the cumulative sums \( S_k = \sum_{j=1}^k s_j \).
Given a target time \( t \), we identify the smallest index \( n \) such that \( S_n > t \), and compute
\[
\tau_t = \sum_{k=1}^{n-1} (-1)^{k-1} s_k + (-1)^{n - 1}(t - S_{n-1}).
\]
In practice, we set \( \kappa \coloneqq \lceil 2a T \rceil + b \), where \( T \) is the maximum time of interest and \( b \) is a small buffer constant (e.g., \( 20 \)) to ensure sufficient coverage with high probability.

This procedure can be applied independently for multiple values of \( t \), allowing efficient simulation of the Kac process, see Algorithm \ref{simulating}.

\begin{algorithm}[ht!] 
\caption{Simulating the Kac process \texorpdfstring{$\tau_t$}{tau\_t}}
\begin{algorithmic}[1]
\STATE \textbf{Input:} Time \( t > 0 \), rate \( a > 0 \), upper bound \( T \), buffer \( b \)
\STATE Set \( \kappa \leftarrow \lceil 2aT \rceil + b \)
\STATE Sample \( s_1, \dots, s_\kappa \sim \mathrm{Exp}(a) \)
\STATE Compute cumulative sums: \( S_k = \sum_{j=1}^k s_j \) for \( k = 1, \dots, \kappa \)
\STATE Find the smallest index \( n \) such that \( S_n > t \)
\STATE Initialize \( \tau \leftarrow 0 \)
\FOR{ \( k = 1 \) to \( n-1 \) }
    \STATE \( \tau \leftarrow \tau + (-1)^{k-1} s_k \)
\ENDFOR
\STATE \( \tau \leftarrow \tau + (-1)^{n-1}(t - S_{n-1}) \)
\STATE \textbf{Output:} \( \tau_t = \tau \)
\end{algorithmic}
\label{simulating}
\end{algorithm}

\subsection{Implementation Details} \label{implementation}

For all experiments we sampled from the Kac process using the following Algorithm \ref{sampling}.
We outline the sampling procedure for the 1D Kac density, for the multidimensional process we use this procedure component-wise.

\begin{algorithm}[ht!] 
\caption{Sampling the 1D Kac density via inverse‐CDF}
    \begin{algorithmic}[1]
      \STATE Draw \(U\sim\mathcal U(0,1)\).
      \IF{\(U < e^{-a\,t}\)}
        \STATE Draw \(\sigma\in\{\pm1\}\) with \(\Pr(\sigma=\pm1)=\tfrac12\).
        \STATE \textbf{Output} \(\sigma\,c\,t\) from the atomic terms in \eqref{telegraph-formula-1d-dirac}.
      \ELSE
        \STATE Draw \(V\sim\mathcal U(0,1)\).
        \STATE Compute $x \;\approx\; F_{\rm cont}^{-1}(V;\,t)$, where \(F_{\rm cont}\) is the continuous‐part CDF in \eqref{telegraph-formula-1d-dirac-cont}, numerically approximated via a precomputed inverse‐CDF table.
        \STATE Draw \(\sigma\in\{\pm1\}\) with \(\Pr(\sigma=\pm1)=\tfrac12\).
        \STATE \textbf{Output} \(\sigma\,x\).
    \ENDIF
    \end{algorithmic}
    \label{sampling}
\end{algorithm}

\paragraph{2D Toy Experiment.}
As the target distribution we use a 9-component Gaussian mixture model (GMM) in $\mathbb{R}^2$ given by the modes
\[
\{\mu_i\}_{i=1}^9
=\Bigl\{(u_1,u_2)\;\Bigm|\;u_j\in\{-1,0,1\}\Bigr\}.
\]
We use uniform weights and a very small, shared isotropic covariance:
\[
w_i = \frac1K,
\qquad
\Sigma_i = \sigma^2 I_d,
\qquad
\sigma = 10^{-4},
\]
and define the mixture density
\[
p(x)
=\sum_{i=1}^K w_i\,\mathcal{N}\bigl(x;\,\mu_i,\Sigma_i\bigr).
\]
We set \(T=1\) for all models.  During inference we use the exact latent
\[
X_t \;:=\; X_0 + K_t,
\]
i.e.\ we draw \(X_0\) directly from our GMM and add the exact Kac noise \(K_t\), thereby avoiding any sampling approximation in this toy example. Analogously, in the diffusion case, we use a standard Brownian motion $B_t$ instead of $K_t$. All models were trained for \(500{,}000\) iterations with a learning rate of \(5\times10^{-4}\) and a batch size of $256$.  For validation we compute the negative log-likelihood (NLL) of \(N\) model-generated samples under the true GMM density:
\[
\mathrm{NLL}
\;=\;
-\frac{1}{N}\sum_{n=1}^N \log p\bigl(x^{(n)}\bigr),
\]
where \(x^{(n)}\) are drawn from the learned model. We set $N = 5000$. For the velocity field we used a fully‐connected MLP with three hidden layers of width \(w=256\) and ReLU activations. Specifically for the diffusion model we required time truncation, namely we train the flow on the interval \([\varepsilon,\,T]\) with \(\varepsilon=10^{-15}\). We used the torchdiffeq \cite{torchdiffeq} package to simulate the ODEs with ``dopri5'' adaptive step size solver. 

\paragraph{CIFAR Experiment.}
We trained all models for $400$k iterations, a batch size of $128$ and a constant learning rate of \(2\times10^{-4}\). We used the torchdiffeq \cite{torchdiffeq} package to simulate the ODEs with an explicit Euler scheme.  We use the time dependent U-Net architecture from \cite{nichol2021improved}. We use the same network parameters as in \cite{tong2024improving}, also employing an EMA with decay rate $0.9999$. We do not use any data augmentation for training. Note that in our implementation of the diffusion model with schedule $g(t)=t$, we use the common time truncation of training on the interval $[\varepsilon,1]$ with $\varepsilon=10^{-5}$ (see \cite{song2021scorebasedgenerativemodelingstochastic}). 
Both training and evaluation was done using fixed random seeds. 
The FID is computed on $50$k samples using the torch fidelity package \cite{obukhov2020torchfidelity}. The code is available at https://github.com/JChemseddine/telegraphers. 

\newpage
\subsection{Further numerical examples} \label{further_examples}
In the Figures \ref{fig:cifar-3} and \ref{fig:cifar-4}, we showcase the CIFAR-10 images generated by our Kac model for other damping/velocity pairs $(a,c)$. The corresponding FID scores can be found in Table \ref{fid_results}.

\begin{figure}[ht!] 
  \centering
  %--- first experiment ---
  \begin{minipage}[b]{0.3\textwidth}
    \centering
    \includegraphics[width=\linewidth]{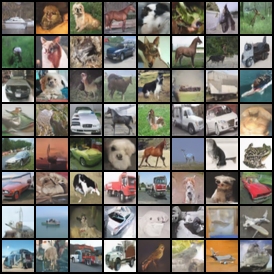}
    \par\smallskip
    {\small \( (a,c)=(25,3)\)}
  \end{minipage}
 \hfill
  %--- second experiment ---
  \begin{minipage}[b]{0.3\textwidth}
    \centering
    \includegraphics[width=\linewidth]{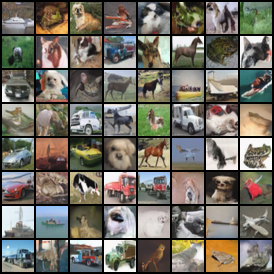}
    \par\smallskip
   {\small \( (a,c)=(100,5)\)}
  \end{minipage}
  \hfill
  %--- third experiment ---
  \begin{minipage}[b]{0.3\textwidth}
    \centering
    \includegraphics[width=\linewidth]{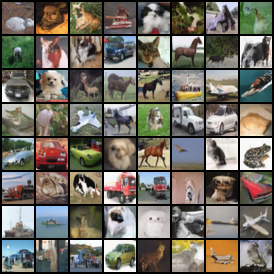}
    \par\smallskip
    {\small \( (a,c)=(900,30)\)}
  \end{minipage}

  \caption{Further generation results of the mean-reverting Kac models with $g(t)=t^2$. }
  \label{fig:cifar-3}
\end{figure}

\begin{figure}[ht!] 
  \centering
  %--- first experiment ---
  \begin{minipage}[b]{0.3\textwidth}
    \centering
    \includegraphics[width=\linewidth]{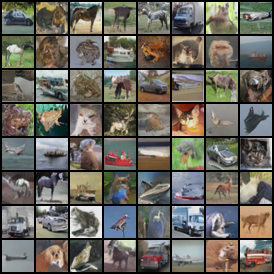}
    \par\smallskip
   {\small \( (a,c)=(100,1)\)}
  \end{minipage}
 \hfill
  %--- second experiment ---
  \begin{minipage}[b]{0.3\textwidth}
    \centering
    \includegraphics[width=\linewidth]{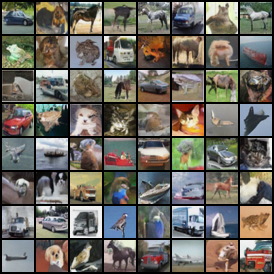}
    \par\smallskip
    {\small \( (a,c)=(100,3)\)}
  \end{minipage}
  \hfill
   %--- third experiment ---
  \begin{minipage}[b]{0.3\textwidth}
    \centering
    \includegraphics[width=\linewidth]{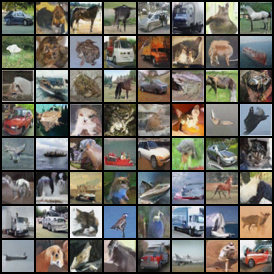}
    \par\smallskip
    {\small \( (a,c)=(900,10)\)}
  \end{minipage}

  \caption{Further generation results of the mean-reverting Kac models with $g(t)=t$. }
  \label{fig:cifar-4}
\end{figure}

\paragraph{Nearest Neighbour Analysis.}
To assess sample fidelity and detect potential memorization, we perform an $\ell_2$-nearest neighbour (NN) analysis in pixel space. For each configuration $(a,c)$ under consideration, we draw $50\,$k samples $\{\hat{x}_i^{(a,c)}\}_{i=1}^{50{,}000}$ from the trained model. For a subset of generated samples $\{\hat{x}_j^{(a,c)}\}_{j\in\mathcal{J}}$, we retrieve the closest training example under the Euclidean distance on raw pixels,
\begin{equation}
x_{\mathrm{NN}}(\hat{x}) \;=\; \arg\min_{x \in \mathcal{D}_{\mathrm{train}}} \|\hat{x}-x\|_2.
\end{equation}
We then visualize triplets consisting of the generated image $\hat{x}$, its nearest neighbour $x_{\mathrm{NN}}(\hat{x})$, and the pixel-wise difference. 
For each time schedule $g(t)=t$ and $g(t)=t^2$,
we showcase this analysis for three choices of parameters $(a,c)$, respectively. The results are reported in Figure \ref{fig:nn_analysis_t2} and \ref{fig:nn_analysis_t}.

% =========================
% NN analysis: g(t) = t^2
% =========================
\begin{figure}[ht!]
  \centering

  % --- (a,c) pair 1 ---
  \begin{minipage}{0.92\textwidth}
    \centering
    \includegraphics[width=\linewidth]{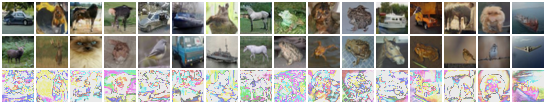}
    \par\smallskip
    {\small \( (a,c)=(25,1)\)}
  \end{minipage}

  \par\medskip

  % --- (a,c) pair 2 ---
  \begin{minipage}{0.92\textwidth}
    \centering
    \includegraphics[width=\linewidth]{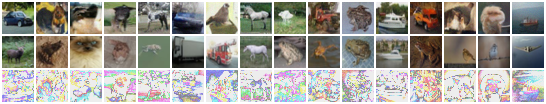}
    \par\smallskip
    {\small \( (a,c)=(100,5)\)}
  \end{minipage}

  \par\medskip

  % --- (a,c) pair 3 ---
  \begin{minipage}{0.92\textwidth}
    \centering
    \includegraphics[width=\linewidth]{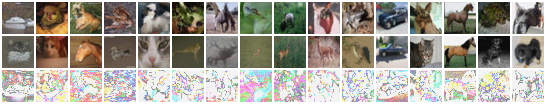}
    \par\smallskip
    {\small \( (a,c)=(900,10)\)}
  \end{minipage}

  \caption{Nearest neighbour analysis for three representative $(a,c)$ pairs with schedule $g(t)=t^2$.
  Each panel visualizes three rows: \emph{generated sample, $\ell_2$-nearest training neighbour, pixel-wise absolute difference}.}
  \label{fig:nn_analysis_t2}
\end{figure}

% =========================
% NN analysis: g(t) = t
% =========================
\begin{figure}[ht!]
  \centering

  % --- (a,c) pair 1 ---
  \begin{minipage}{0.92\textwidth}
    \centering
    \includegraphics[width=\linewidth]{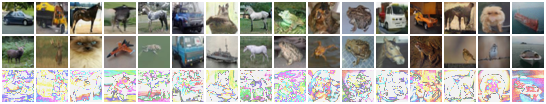}
    \par\smallskip
    {\small \( (a,c)=(25,2)\)}
  \end{minipage}

  \par\medskip

  % --- (a,c) pair 2 ---
  \begin{minipage}{0.92\textwidth}
    \centering
    \includegraphics[width=\linewidth]{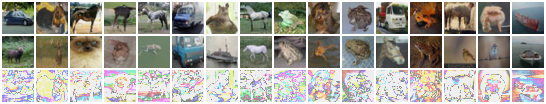}
    \par\smallskip
    {\small \( (a,c)=(100,3)\)}
  \end{minipage}

  \par\medskip

  % --- (a,c) pair 3 ---
  \begin{minipage}{0.92\textwidth}
    \centering
    \includegraphics[width=\linewidth]{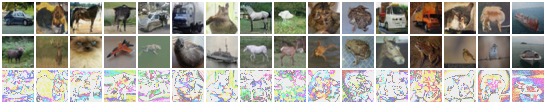}
    \par\smallskip
    {\small \( (a,c)=(900,10)\)}
  \end{minipage}

  \caption{Nearest neighbour analysis for three representative $(a,c)$ pairs with schedule $g(t)=t$.
  Each panel visualizes three rows: \emph{generated sample, $\ell_2$-nearest training neighbour, pixel-wise absolute difference}.}
  \label{fig:nn_analysis_t}
\end{figure}

%%%-----------------------

\end{document}